\documentclass[aos, authoryear]{imsart}

\RequirePackage{amsthm,amsmath,amsfonts,amssymb}
\RequirePackage[authoryear]{natbib} 
\RequirePackage[colorlinks,citecolor=blue,urlcolor=blue]{hyperref}
\RequirePackage{graphicx}

\startlocaldefs
\numberwithin{equation}{section}
\usepackage{amsmath}

\usepackage{amsthm}
\usepackage{cleveref}

\usepackage{comment}

\def\[#1\]{ \begin{align} #1 \end{align} }
\def\*[#1\]{\begin{align*}#1\end{align*}}

\allowdisplaybreaks

\newtheorem{innercustomgeneric}{\customgenericname}
\providecommand{\customgenericname}{}
\newcommand{\newcustomtheorem}[2]{%
  \newenvironment{#1}[1]
  {%
   \renewcommand\customgenericname{#2}%
   \renewcommand\theinnercustomgeneric{##1}%
   \innercustomgeneric
  }
  {\endinnercustomgeneric}
}

\newcustomtheorem{customthm}{Theorem}
\newcustomtheorem{customlemma}{Lemma}
\newcustomtheorem{customprop}{Proof of Proposition}
\newcustomtheorem{customcor}{Corollary}
\newcustomtheorem{customass}{Assumption}

\newcommand{\norm}[1]{\left\lVert#1\right\rVert}

\newcommand{\saddle}{\hat{t}_n}
\newcommand{\rv}{X_{n}}
\newcommand{\Reals}{\mathbb{R}}
\newcommand{\onep}{\mathbf{1}_p}
\newcommand{\zerop}{\mathbf{0}_p}
\newcommand{\zeropminus}{\mathbf{0}_{p-1} }
\newcommand{\approxpoint}{s_n}
\newcommand{\maximizer}{\hat{\theta}_n}
\newcommand{\maximizercons}{\hat{\theta}_{\psi}}
\newcommand{\maximizerconstild}{\hat{\theta}_{\tilde\psi}}
\newcommand{\loglike}{g}
\newcommand{\loglikeminus}{g_{\lambda\lambda}}
\newcommand{\logliken}{g_n}

\newcommand{\thirdpower}{c_3}
\newcommand{\fourthpower}{c_4}

\newcommand{\maxpsi}{\hat\psi}
\newcommand{\conthirdmat}{\loglike_{\psi\lambda\lambda}^{(3)}}
\newcommand{\densmarg}{f(\psi| X_n)}
\newcommand{\densmargapprox}{\hat{f}(\psi| X_n) }

\newcommand{\interest}{\psi}
\newcommand{\nuissance}{\lambda}
\newcommand{\suffsplit}{(s_1, s_2)}
\newcommand{\tsplit}{t_{\interest}, t_{\nuissance}}
\newcommand{\saddlenuis}{\hat{t}_{\nuissance}}
\newcommand{\saddleint}{\hat{t}_{\interest}}
\newcommand{\ppower}{c_p}
\newcommand{\thetanew}{\bar{\theta}}
\newcommand{\ynew}{\bar{y}}

\newtheorem{theorem}{Theorem}[section]
\newtheorem{assumption}{Assumption}
\newtheorem{remark}{Remark}[section]
\newtheorem{corollary}{Corollary}[section]
\newtheorem{lemma}[theorem]{Lemma}
\theoremstyle{remark}

\newtheorem{example}{Example}

\endlocaldefs

\begin{document}

\begin{frontmatter}
\title{Laplace and saddlepoint approximations in high dimensions}
\runtitle{Laplace and saddlepoint approximations in high dimensions}

\begin{aug}
\author[A]{\fnms{Yanbo} \snm{Tang}},
\and
\author[B]{\fnms{Nancy} \snm{Reid}}

\address[A]{Department of Mathematics, Imperial College London, London, UK
}

\address[B]{Department of Statistical Sciences,
University of Toronto, Toronto, Canada}

\end{aug}

\begin{abstract}
 We examine the behaviour of the Laplace and saddlepoint approximations in the high-dimensional setting, where the dimension of the model is allowed to increase with the number of observations. 
 Approximations to the joint density, the marginal posterior density and the conditional density are considered. 
Our results show that under the mildest assumptions on the model, the error of the joint density approximation is $O(p^4/n)$ if $p = o(n^{1/4})$ for the Laplace approximation and saddlepoint approximation, and $O(p^3/n)$ if $p = o(n^{1/3})$ under additional assumptions on the second derivative of the log-likelihood. 
Stronger results are obtained for the approximation to the marginal posterior density.
 \end{abstract}



\end{frontmatter}

\section{Introduction}
Analytical approximations derived from asymptotic theory are commonly used to provide accurate approximations to densities whose exact forms are unavailable.
Two widely-used density approximations are the saddlepoint and Laplace approximations, typically used in frequentist and Bayesian inference respectively.
The properties of these approximations are well-studied when the number of parameters, $p$, is fixed. 
However, they are not fully understood in the high-dimensional setting when $p$ is allowed to grow with the number of samples $n$. 
Some exceptions are \cite{shun}, who studied the approximation error of the Laplace approximation in high dimensions for regression models based on the linear exponential family and \cite{barber_laplace} who studied the Laplace approximation in the context of model selection for Bayesian regression models in the exponential family.
Under the assumption of a Gaussian prior and a weakly concave log-likelihood \cite{Spokoiny2022} studied the accuracy of the Gaussian approximation to the posterior distributions in high dimensions, and showed the approximation error can be controlled if $p_0^3 \ll n$, where $p_0$ is a measure of the effective dimension of the model.
\cite{kasprzak2022good} showed it is possible for the Gaussian approximation to converge to the true posterior at a rate of $O(p/n^{1/2})$ in total variation distance and \cite{katsevich2023tight} refined this result by demonstrating that the second order error term is $O(p^2/n)$ and furthermore proved that this rate cannot be improved in general by showing that it is tight for a particular logistic regression model. \cite{katsevich2023improved} then showed that the rate $O(p/n^{1/2})$ is achieved by the natural exponential family and logistic regression with random Gaussian design.

The lack of analysis of these approximation methods in high dimensions hampers the development of theory for commonly-used methods. 
One example is \cite{inla}, who noted that the theoretical accuracy of INLA (which uses both the Laplace and Gaussian approximation) when used for high-dimensional spatial models is not well understood.
The Laplace approximation is also used in the evaluation of integrals in mixture models for frequentist inference, in the derivation of the Bayesian information criterion (BIC), and in generalized additive models for approximating the marginal likelihood \citep{wood_laplace}.
Similarly, the saddlepoint approximation is pivotal in the development of likelihood-based approximations, including approximate conditional inference, modified profile likelihoods and directional inference.

In this paper we establish rigorous rates of convergence for the Laplace and saddlepoint approximation when $p$ is allowed to grow with $n$ for general models, and discuss how these rates can be improved by leveraging the structure of some particular models. 
The Laplace approximation aspect of this work is an extension of \cite{shun}, who noted that at the time ``It does not seem feasible at the present to develop useful general theorems for approximating arbitrary high-dimensional integrals".
Extending their results to more general settings also provides better justification for existing results derived from their work such as in \cite{laplace_splines}, who studied the behaviour of the Laplace approximation for smoothing splines and \cite{ogden2021error} who studied the Laplace approximation in high dimensions when the likelihood derivatives grow at different rates, such as in mixture models.

\cite{saddle_high}, gives examples where the saddlepoint approximation can fail in the high-dimensional setting; we are not aware of any other work on the topic of the saddlepoint approximation in high-dimensions.

In this paper we use the term Laplace approximation to refer to an integral approximation to the normalizing constant of the posterior as described in Section \ref*{sec:laplace}. 
Sometimes the term Laplace approximation is used to mean a Gaussian approximation to the posterior \citep{kasprzak2022good,katsevich2023improved,katsevich2023tight,Spokoiny2022}.
Although the Laplace and Gaussian approximations to the posterior distribution are related, in that they both provide a tractable analytical approximation to the posterior distribution, results obtained for one does not imply the same for the other. But these results on the Gaussian approximation may provide insight into what general error rates are achievable.

We also examine the use of the saddlepoint and Laplace approximation to ratios of integrals. 
These arise in conditional inference in the linear exponential families and in approximations to the marginal posterior density.
The results obtained for the marginal posterior allow for a more aggressive growth of $p$ in $n$, as cancellations occur in the ratio of certain error terms.

Section \ref{sec:notation} describes the notation that will be used throughout the main sections of the paper and the supplementary materials.
Section \ref{sec:laplace} examines the Laplace approximation in high dimensions, with an example in the linear exponential family.  
Section \ref{sec:ratios} describes some additional cancellations that may occur when examining ratios of density approximations for the Laplace approximation.
Section \ref{section:saddle_dens} examines the saddlepoint approximation in high dimensions. 
Section \ref{sec:ratio_saddle} examines the use of the saddlepoint approximation in conditional inference in linear exponential family models.
Section \ref{sec:conclusion} provides some discussion of the limitations of this work and potential directions for improvement. All proofs are deferred to the Supplementary Materials. 

\section{Notation} \label{sec:notation}
The Euclidean ball centered at $x$ with radius $\delta$ is denoted by $B_{x}(\delta)$, the Cartesian product of sets $[a_j, b_j]$ for $j = 1,\dots, p$ is $\prod_{j = 1}^p [a_j, b_j]$ and $S^C$ is the complement of the set $S$.

The ordered eigenvalues of a $p \times p$ real valued symmetric matrix $A$ are denoted by $\lambda_1(A) \geq \lambda_{2}(A) \geq \dots \geq \lambda_p(A)$, the maximum singular value of $A$ is $\norm{A}_{op}$ and 
\*[
\norm{A}_\infty = \max_{j = 1, \dots, p} \sum_{k = 1}^{p} |a_{jk}|,
\] 
where $a_{jk}$ is the $(j,k)^{th}$ entry of $A$.
The $p \times p$ identity matrix is denoted by $I_{p}$, $\onep$ is a column vector of $1$'s of length $p$ and $\zerop$ a column vector of $0$'s of length $p$. A useful inequality is Rayleigh's quotient
\begin{align*}
 \norm{z}_2^2 \lambda_p(A) \leq  z^\top A z \leq \norm{z}_2^2 \lambda_1(A),
    \end{align*}
for any real valued vector $z$ of length $p$. 

The moment generating function of a random variable $Y$ is denoted by $M_Y(t) = \text{E}[\exp(tY)]$, $K_Y(t) = \log\{M_Y(t)\}$ is the cumulant generating function and $\xi_Y(t) = \text{E}[\exp(itY)]$ the characteristic function.
The $j$-th derivative of a function $f: \mathbb{R}^p \rightarrow \mathbb{R}$ is denoted by 
$f^{ (j) }$, and subscripts are used to refer to specific elements, for example:
\begin{align*}
    f^{(3)}_{jkl}(\theta) = \frac{\partial^3}{\partial \theta_j \partial \theta_k \partial \theta_l} f(\theta),
\end{align*}
and 
\*[
f^{(2)}_{\psi\lambda}(\theta) = \frac{\partial^2}{\partial\psi\partial\lambda}f(\theta), 
\]
where $\theta = (\psi, \lambda)$. We extend this notation to higher-order derivatives in the obvious way.

 Let $g(n)$ be a sequence of real numbers. We use $g(n) = O(a_n)$ to mean that $\exists N_0, B: \forall n> N_0, \ |g(n)| \leq Ba_n $. A vector or matrix is said to be $O(a_n)$ if its entries are $O(a_n)$ uniformly, meaning the constants in the $O$ term are uniformly bounded.

The density of a multivariate normal random variable with mean $\mu$ and covariance matrix $\Sigma$ evaluated at a vector $x$ is $\phi(x;\mu, \Sigma )$.

\section{Laplace approximation}\label{sec:laplace}
We consider a sequence of data $X_n$ from a model with density $f(X_n| \theta_0)$.  
Let $\pi(\theta)$ be the prior distribution on the parameter space $\Theta = \mathbb{R}^{p}$ and $l_n(\theta; X_n)$ be the log-likelihood function.  
Define $\loglike_n(\theta; X_n) = \log\{ \pi(\theta)\} + l_n(\theta; X_n)$. 
In what follows we may sometimes suppress the dependence of $g_n (\theta; X_n)$ on $n$ and $X_n$. Define $\maximizer$ as the maximizer of the function $g_n(\theta; X_n)$.
The posterior density is
\[\label{eq:posterior_def}
f(\theta|X_n) =  \frac{\exp\{ g_n(\theta; X_n) - g_n(\maximizer; X_n)  \}}{\int_{\mathbb{R}^p}  \exp\{ g_n(\theta; X_n) - g_n(\maximizer; X_n)\} d\theta}  , 
\]
where we have normalized the function $\loglike_n(\theta)$ by its maximum value, $g_n(\maximizer)$. 
For Theorem \ref{thm:laplace} to hold, it is not necessary for $l_n(\theta; X_n)$ to be a log-likelihood function, so long as it satisfies the Assumptions below. 
If $l_n(\theta; X_n)$ is not a log-likelihood function, the posterior is sometimes referred to as the Gibbs posterior; for example see \citet{jiang2008} and \citet{grunwald2017}.

\cite{tierney1986} derived the Laplace approximation to joint and marginal posterior distributions and posterior moments. 
Applying the Laplace approximation to the normalizing constant leads to
\[ \hat{f}(\theta|X_n ) = \frac{\det\{-g_n^{(2)}(\hat\theta)\}^{1/2}}{ (2\pi)^{p/2}  } \exp\{ g_n(\theta; X_n) - g_n(\maximizer ; X_n)   \} \label{eq:laplace_approxiamtion}.\]
The formal expansions in \cite{shun} suggest that for general models, this Laplace approximation to the normalizing constant has relative accuracy $O(p^6/n)$, and $O(p^3/n)$ for the linear exponential family. 
However, this result was derived by assuming that the model is infinitely differentiable and implicitly assuming that the order of an infinite summation and integration may be interchanged, which is not always the case. 
We extend their result to general models which are not infinitely differentiable and under more precise conditions.

As in \citet[\S2]{laplace_validity}, we consider the observed data to be subsequences of a given, fixed infinite sequence of realizations. An approach similar to that in \cite{bilodeau2021stochastic} could be used to extend the results to the stochastic case, we leave this as future work.

Theorem \ref{thm:laplace} examines the general model. 
For specific models, one can use the same general steps as in this proof but use additional information (or assumptions) on the model to refine the results.

\subsection{Main theorem}
%

Let $\delta > 0$ be constant with respect to $p$ and $n$, and $\gamma_n^2 = \log(n)p/n$. 
\begin{assumption} \label{ass:delta_decay_lap}
\begin{align*}
    &\frac{\det\{-g_n^{(2)}(\maximizer)\}^{1/2} }{ (2\pi)^{p/2}  } \int_{  B^{C}_{\maximizer}(\delta) }  \exp\{ g_n(\theta; X_n)  - g_n(\maximizer ; X_n)  \} d\theta    =O\left( a_{n,p} \right), 
\end{align*}  
for a sequence $a_{n,p} \rightarrow 0$ as $n \rightarrow \infty$ and $p \rightarrow \infty$.
\end{assumption}

\begin{assumption} \label{ass:hess_lap}
The eigenvalues of the Hessian matrix of $g_n(\theta)$  satisfy:
\*[ 0< \eta_1 n \leq \lambda_p[ -g^{(2)}_n(\theta) ] \leq \lambda_1[ -g^{(2)}_n(\theta)] \leq \eta_2 n < \infty ,\]
for all $\theta \in B_{\maximizer}(\delta)$, and $\lVert \{-g^{(2)}_n(\maximizer)\}^{-1/2} \rVert_{\infty} = O(p^{c_{\infty}} n^{-1/2})$ for some $0 \leq  c_{\infty} \leq 1/2$.
\end{assumption}

\begin{assumption} \label{ass:third_lap}
The eigenvalues of the sub-matrices $\loglike^{(3)}_{\cdot \cdot l}(\theta)$ with $(j,k)^{th}$ entry $ [\loglike^{(3)}_{\cdot \cdot l}(\theta)]_{jk} = g^{(3)}_{jkl}(\theta) $ satisfy
\begin{align*}
\eta_{3} n^{\thirdpower} \leq \lambda_p[  \loglike^{(3)}_{\cdot \cdot l} (\maximizer) ] &\leq \lambda_1[  \loglike^{(3)}_{\cdot \cdot l} (\maximizer) ] \leq \eta_{4} n^{\thirdpower} ,
\end{align*}
for $l = 1, \dots, p$ and some $\eta_3, \eta_4 \in \Reals$.
\end{assumption}

\begin{assumption}\label{ass:fourth_lap}
The eigenvalues of the sub-matrices  $\loglike^{(4)}_{\cdot \cdot lm}(\theta)$ with $(j,k)^{th}$ entry $[\loglike^{(4)}_{\cdot \cdot lm}(\theta)]_{jk} = \loglike^{(4)}_{jklm}(\theta)$, satisfy
\begin{align*}
 \eta_{5} n^{\fourthpower} \leq \lambda_p[  \loglike^{(4)}_{\cdot \cdot lm} (\theta) ] &\leq \lambda_1[  \loglike^{(4)}_{\cdot \cdot lm} (\theta) ] \leq \eta_{6} n^{\fourthpower} ,
\end{align*}
for all $\theta \in  B_{\maximizer}(2^{1/2}\gamma_n)$ and for all $l, m = 1, \cdots, p$ and some $\eta_5, \eta_6 \in \Reals$.
\end{assumption}


Assumption \ref{ass:delta_decay_lap} limits the size of the integral outside of a Euclidean ball with radius $\delta$, and is adapted from Assumption iii) in \cite{laplace_validity}. 
This will typically be satisfied for models with concave log-likelihood functions, as in the linear exponential family. 
The eigenvalue restrictions in Assumptions \ref{ass:hess_lap}--\ref{ass:fourth_lap} are needed to control the growth of the Hessian and higher-order derivatives, and are similar to those in \cite{non-uniform}.
The constant $c_{\infty}$ is a measure of the dependence among the elements of $\theta$, and the restriction of $c_{\infty} \leq 1/2$ is natural as $\lVert \{\loglike^{(2)}(\maximizer)\}^{-1/2} \rVert_{\infty} \leq p^{1/2} \lVert \{\loglike^{(2)}(\maximizer)\}^{-1/2} \rVert_{op} = O(p^{1/2}/n^{1/2})$.
Cases where $c_{\infty} < 1/2$ can arise when the Hessian is block diagonal or banded, and if the Hessian is block diagonal and the blocks are of fixed size, then $c_{\infty} = 0$.
We give another example where $c_{\infty} = 0$ in Corollary \ref{cor:logistic_laplace}.   
The constants $\thirdpower$ and $\fourthpower$ will typically be $\leq 1$. 
An example where $\thirdpower  = (1+\alpha)/2 + \log\log(n)/\log(n)$ is given in \S\ref{example:logis_laplace}.

\begin{theorem} \label{thm:laplace}
Let $p = O(n^\alpha)$, $ \alpha < \min\{ (3 - 2\thirdpower)/(3 + 2c_{\infty}) , (4 -2 \fourthpower)/(5 + 4c_{\infty} )\}$. For a sequence  $\{ X_n\}$ satisfying Assumptions \ref{ass:delta_decay_lap}--\ref{ass:fourth_lap}, and in Assumption \ref{ass:delta_decay_lap}, $a_{n,p} = \max\left\{p^{3 + 2c_{\infty}}/n^{3 - 2\thirdpower}, p^{2 + 2c_{\infty}}/n^{2 - \fourthpower} \right\}$ ,
\*[\frac{f(\theta|X_n)}{\hat{f}(\theta|X_n)} =    1 +  O\left\{ \max\left(\frac{p^{3 + 2 c_{\infty}} }{n^{3 - 2\thirdpower} }, \frac{p^{2 + 2 c_{\infty}} }{n^{2 - \fourthpower}}  \right) \right\} . \]
\end{theorem}




\begin{remark}
    Note that the approximation error of $\hat{f}(\theta|X_n)$ is uniform in $\theta$.
    The functional form of the posterior is known, so the approximation accuracy of the normalizing constant directly translates into uniform accuracy for the density approximation of the posterior.
\end{remark}

\begin{remark}\label{remark:mle}
The assumptions may also be stated for the maximum likelihood estimate (mle) rather than the posterior mode in Assumptions \ref{ass:delta_decay_lap}--\ref{ass:fourth_lap}. However in doing so, we will need to account for the prior separately by expanding the ratio $\pi(\theta)/\pi(\hat\theta_{\text{mle}})$. We examine this more closely in the proof of Corollary \ref{cor:logistic_laplace}.
In this case Assumption \ref{ass:delta_decay_lap} can be replaced by a stricter but perhaps easier to check condition inspired by the one given in \cite{laplace_validity}
\*[ \limsup_{n \rightarrow \infty} \{\logliken(\hat\theta_{\text{mle}}) - \logliken(\theta) \} \leq  -cn^{\epsilon} , \]
for all $\{\theta: \lVert\theta - \hat\theta_{\text{mle}}\rVert_2 > \delta \} $, and for some $\epsilon,c > 0$ independent of $n$ and $p$. 
\end{remark}

\begin{remark}
Assumption \ref{ass:delta_decay_lap} may be removed and the radius $\delta$ in Assumption \ref{ass:hess_lap} changed to $\gamma_n$ if we directly assume the integral over $B^C_{\maximizer}(\gamma_n)$ is $O(a_{n,p})$. 
This may be easier to show in some models than verifying Assumptions \ref{ass:delta_decay_lap} and \ref{ass:hess_lap}, in particular for concave log-likelihoods.
\end{remark}

\begin{remark}
Our results can also be easily extended to the calculation of deterministic integrals of the form 
\*[ \int_{\mathbb{R}^p} \exp\{ n f(x) \} dx ,  \]
as $n, p \rightarrow \infty$, with slight modifications of the conditions. These types of integrals are typically considered in the numerical analysis literature. 
Similarly the result of Theorem \ref{thm:laplace} can be applied to numerical approximation when integrating out random effects, under Assumptions \ref{ass:delta_decay_lap}--\ref{ass:fourth_lap}. 
\end{remark}

\subsection{Some examples}\label{example:logis_laplace} 
The following is an example in which the order of the approximation error is reduced by exploiting the specific structure of the model. 

\begin{example}\label{example:laplace_logistic}
\textbf{Logistic regression}. Consider
\begin{align}\label{eq:logistic_eq}
y_j \sim \text{Bern}\{ p(x_j^\top\beta) \}, \quad p(z) = \frac{\exp(z)}{1+	\exp(z)},
\end{align}
where the vectors $x_j \overset{\text{iid}}{\sim} N(0, I_p)$ for $j = 1, \dots, n$.
Let $X$ be the matrix of covariates with the $j$-th row $x_j$, and  the $(j, k)^{th}$ entry $x_{jk}$. 
We assume that the data generating parameter $\beta_0 = \zerop$. 
Based on \citet[Section B.4]{non-uniform}, $\max_{j = 1, \dots, n} |x_j^\top \hat\beta_{mle}| = O \{( p/n )^{1/2} \}$ with probability tending to $1$ in the joint distribution of the data $(X, Y)$.

For the sake of simplicity, we consider a model with independent Gaussian priors, $\beta_j \sim N(0,1)$. 

\begin{corollary}\label{cor:logistic_laplace}
Under model (\ref{eq:logistic_eq}), $p = O(n^\alpha)$ for $\alpha < 2/5$ and Condition 1 and 2 in \cite{non-uniform},
\begin{align*}
 \lim_{n \rightarrow \infty} \mathbb{P}_{(X_n, Y_n)}\left[  \frac{f(\beta|X, Y)}{\hat{f}(\beta|X, Y) }  = 1 + O \left( \frac{p^2\log(n)}{n}  \right) \right]  = 1, 
\end{align*}
where, 
\*[
\hat{f}(\beta|X_n, Y_n) = \frac{\det\{ -l_n^{(2)}(\hat\beta_{mle}) \}^{1/2}}{(2\pi)^{p/2}} \frac{\pi(\beta)}{\pi(\hat{\beta}_{mle})} \exp\{l_n(\beta) -  l_n(\hat\beta_{mle})\}.
\]
\end{corollary}

\begin{remark}
The assumptions used in this example resemble those made in \citet[Section 6]{shun}, for linear exponential models. 
For example, the requirement that the cumulants are approximately constant in \citet{shun} is satisfied if the regression parameter $\beta = \zerop$.
The error of the approximation in Corollary \ref{cor:logistic_laplace} is better than the $p^3/n$ error in \citet[Section 6]{shun}, due to the fact that the third log-likelihood derivative of the Bernoulli likelihood is $0$ at $\hat{p} = 1/2$.
\end{remark}
\end{example}


In the example which follows, we examine generalized linear models and remove the assumption that $\beta_0 = \zerop$. In the extension to generalized linear models, the error rate is poorer than for the example logistic regression (with $\beta_0 = \zerop$). 
The improved error rate of the logistic model is due to the symmetry of the Bernoulli distribution with probability of success of $p = 1/2$.
This symmetry reduces the contribution of the skewness in error term involved in the logistic example.
In general it is no longer the case that for some values of $\beta$ the distribution of $Y$ will be symmetric, hence the error rate will be worse. 

\begin{example}\label{ex:glm}
    \textbf{Generalized linear models}. Consider a generalized linear model with known dispersion parameter and let the notation be the same as in Example \ref{example:laplace_logistic}, 
    \begin{align}\label{eq:glm_eq}
    E[y_j| x_j] = \rho^{-1}(\eta_j), \quad \eta_j = x_j^\top \beta,
    \end{align}
    where $\rho(\cdot)$ is the link function.
    Based on \citet[Section B.4]{non-uniform}, $\max_{j = 1, \dots, n} |x_j^\top (\hat\beta_{mle} - \beta_0)| = O \{( p/n )^{1/2} \}$ with probability tending to $1$ in the joint distribution of the data $(X, Y)$ as $p$ and $n$ increase.
    We consider a model with independent Gaussian priors, $\beta_j \sim N(0,1)$. 
        
    \begin{corollary}\label{cor:logistic_glm}
    Under model (\ref{eq:logistic_eq}) with $p = O(n^\alpha)$ for $\alpha < 1/3$, $|\eta_i| \leq B$ for all $i = 1, \dots, n$ and Conditions 1 and 2 in \cite{non-uniform},
    \begin{align*}
         \lim_{n \rightarrow \infty} \mathbb{P}_{(X_n, Y_n)}\left[  \frac{f(\beta|X, Y)}{\hat{f}(\beta|X, Y) }  = 1 + O \left( \frac{p^3\log(n)}{n}  \right) \right]  = 1, 
    \end{align*}
    where,
    \*[
        \hat{f}(\beta|X_n, Y_n) = \frac{\det\{ -l_n^{(2)}(\hat\beta_{mle}) \}^{1/2}}{(2\pi)^{p/2}} \frac{\pi(\beta)}{\pi(\hat{\beta}_{mle})} \exp\{l_n(\beta) -  l_n(\hat\beta_{mle})\}.
    \]
    \end{corollary}

\end{example}

\begin{remark}
    \cite{shun} assume the cumulants of $Y$ to be approximately constant, therefore, for any $\beta_0 \neq \zerop$ the rates in Corollary \ref{cor:logistic_glm} are not directly comparable to those in \cite{shun}. However in the case that $\beta_0 = \zerop$, the cumulants are constant and our results has an additional logarithmic factor which emerges from our proof, but can perhaps be eliminated with different techniques. 
\end{remark}
\begin{remark}
    The results of Corollary \ref{cor:logistic_laplace} and \ref{cor:logistic_glm} can hold with a different choice of prior, with some slight adjustments to the proof. For twice differentiable priors which are concentrated around 0, we can add a second-order Taylor expansion around the true value of the MLE to the fourth order Taylor expansion of the log-likelihood function.
\end{remark}

\begin{example}
\textbf{P-splines}. Smoothing splines are a popular nonparametric estimator for smooth functions. Such models typically assume each observation is generated from a exponential family model $f(y_j; \eta_j)$ where 
\*[
    \eta_j = x_j^\top \beta + g(t), 
\]
where $t$ is the smoothing variable which governs the behaviour of the functional relationship, usually taken as a continuous variable such as time or age, and $x_j^\top\beta$ specifies a parametric linear relationship. The function $g(t)$ is modeled through a set of polynomial basis called \emph{basis-splines} or B-splines, often taken to be piecewise continuous polynomials. We approximate the unknown function $g(t)$ by a spline of degree $q$ with equally spaced $m$ knots, 
\*[
g(t) = \sum_{k = 1}^m \alpha_k B_{k}(t).   
\]
Typically a Gaussian model is assumed for the coefficients $\alpha_k$ in which case such models can be expressed as a generalized linear mixed model with Gaussian random effects $\alpha \sim N(0, \sigma_\alpha^2 I_{m - 1})$
For more detail on splines and their uses, see \cite{deboor}. A penalty term is often used in order to prevent overfitting, typically placed on the $(l-1)$st derivative of the spline.
The Laplace approximation is used to approximate the marginal likelihood in such models, see \cite{wood_laplace}.
The marginal likelihood is used for inference on the fixed effects $\beta$, obtained by integrating out the random-spline effects:
\*[ L_{M}(\beta, \sigma_\alpha) = \frac{1}{\sigma_\alpha^{m- 1}} \int_{\mathbb{R}^{m-1}} \exp\{l(\beta, \sigma_\alpha, \alpha)\} dN(0, \sigma_\alpha^2 I_{m - 1}),\]
where $l(\beta, \sigma_\alpha, \alpha)$ is the log-likelihood function. This integral, and thus the marginal likelihood, is typically intractable for non-Gaussian response models.

As most smooth functions are only locally polynomial, we need the number of knots to increase with $n$ to obtain a consistent estimate of $g(t)$. \citet[Theorem 2]{laplace_splines} used the result of \cite{shun} to show the approximation error of the Laplace approximation for the marginal likelihood tends to $0$ if the number of knots satisfies:
\*[
m \leq Cn^{\frac{1}{2q + 3}}.     
\]
However for this model, the assumption that the cumulants are approximately constant may not hold if the true fixed effect $\beta_0 \neq \zerop$, therefore it is not clear that the result of \cite{shun} will hold. More specifically, it is not clear if their Equation (3.6) correctly quantify the error introduced by the Laplace approximation on the likelihood. 

However, the error terms obtained in this paper are similar, the main difference being that due to the finite order Taylor expansion the remainder term is evaluated at a changing point and this introduces a logarithmic factor in the error rate.
Hence, their Theorem 2 can be alternatively shown from our Theorem \ref{thm:laplace}, where Assumption \ref{ass:delta_decay_lap} can be shown to hold by the concavity of the likelihood, by the same argument as for the GLM examples. Assumption \ref{ass:hess_lap} -- \ref{ass:fourth_lap} can be verified by the expressions for the second, third and fourth derivative provided in \cite{laplace_splines}, these can be found between their equations (3.7) and (3.8). 
The overall size of the remainder can then be calculated through the same detailed combinatorial arguments used in the original paper so we will not reproduce them here.
\end{example}

\section{Ratio of integral approximations}\label{sec:ratios}
An unnormalized marginal posterior density approximation can be obtained by applying the Laplace approximation to the numerator and denominator of a ratio of two similar integrals. 
It is possible that some error terms cancel, and this leads to an improvement in the asymptotic error rates or the speed at which $p$ is allowed to increase as $n$ increases.   
Let $\theta = (\interest, \nuissance)$, where $\interest \in \mathbb{R}$ is the parameter of interest and $\nuissance \in \mathbb{R}^{p-1}$ is the nuisance parameter.
The marginal posterior density for $\psi$ is 
\[\label{eq:marg_posterior} 
\densmarg =   \frac{\int_{R^{p - 1}} \exp\{ \loglike_n(\interest, \nuissance) \} d\nuissance }{\int_{\mathbb{R}^p} \exp\{ \loglike_n(\theta) \} d\theta }.
\]
Applying Laplace approximations to the numerator and denominator, respectively, gives 
\*[
\densmargapprox = \frac{ \det\{- \loglike^{(2)}(\maximizer)  \}^{1/2}}{ (2\pi)^{1/2}  \det\{- \loglikeminus^{(2)} (\maximizercons) \}^{1/2}} \exp\{ \loglike(\interest, \hat\nuissance_\psi) - \loglike(\hat{\interest}, \hat{\nuissance} ) \},
\]
where $\loglikeminus^{(2)}(\theta)$ denotes the block of the Hessian associated with the nuisance parameters evaluated at $\theta$, $\hat\lambda_\psi = \text{argsup}_{\lambda} \loglike(\psi, \lambda) $ and $\hat\theta_\psi = (\psi, \hat\lambda_\psi)$.
In the $p$-fixed asymptotic regime, this approximation has a relative error of $O(1/n)$, and a relative error of $O(1/n^{3/2})$ for $\psi$ such that $|\psi - \hat\psi| = O(1/n^{1/2})$, if the density is renormalized.
We examine the marginal approximation in general models and then in the linear exponential family.

\subsection{Marginal laplace approximation- general models}
Consider the parametrization of the model in which 
the parameter of interest is orthogonal to the nuisance parameters. 
Under this parametrization, the expected information $\mathbb{E}[j_{\lambda\psi}(\psi, \lambda)] = 0$, and the observed information $j_{\lambda\psi}(\psi, \lambda) = O_p(n^{1/2})$ \citep{ortho} .
In the Bayesian context the analogous properties, $\mathbb{E}[g_{\psi\lambda}^{(2)}(\theta)] = 0$, and $g_{\psi\lambda}^{(2)}(\theta_0) = O_p(n^{1/2})$ hold under the orthogonal parametrization if the prior for the parameter of interest is independent of the prior for the nuisance parameters.

The orthogonal parametrization is helpful because under this parametrization the constrained mode $\hat\theta_\psi$ is less sensitive to changes in $\psi$; this statement is made more precise in Lemma \ref{lemma:nuissance_size}.  
This implies that for values of $\psi$ near $\hat\psi$, $\hat\theta_\psi$ and $\hat\theta$ are quite close and this leads to the cancellation of some error terms. 

We require the following additional assumptions, the first of which is a higher-order extension of Assumptions \ref{ass:third_lap} and \ref{ass:fourth_lap}.
The second helps limit the sensitivity of the constrained mode to changes in $\psi$. 
\begin{assumption}\label{ass:exp_ratio}
There exists an integer $\zeta > 4 $, such that for all integers $k$ satisfying $4 < k \leq \zeta $, 
\*[ 
B_k n \leq \lambda_p \left[ \loglike^{(k)}_{\cdot \cdot j_{1} \cdots j_{k -2}}(\theta )\right] \leq \lambda_1 \left[ \loglike^{(k)}_{\cdot \cdot j_{1} \cdots j_{k -2}}(\theta )\right]\leq C_k n,
\]
for all $\theta \in B_{\maximizer}(2^{1/2}\gamma_n)$ and $j_1, \dots, j_{k-2} \in \{1, \dots, p\}$.
\end{assumption}

\begin{assumption}\label{ass:general_lap_const}
The sequence, $\maximizer$, satisfies
\*[ 
\lVert\maximizer - \theta_0\rVert_2 = O\left\{ \left( \frac{p}{n} \right)^{1/2} \right\}, \quad \lVert\maximizer - \maximizercons\rVert_2 = O\left\{ \left( \frac{p}{n} \right)^{1/2} \right\}, 
\]
for $\psi \in \{\psi: |\psi - \hat\psi| =  O(\log(n)^{1/2}/n^{1/2}) \}$ where $\theta_0$ is the data-generating parameter. Furthermore, under the orthogonal parametrization
\*[
g_{\psi\lambda}^{(2)}(\theta_0) = O(n^{1/2})
\]
uniformly in $\lambda$.
\end{assumption}

\begin{remark}
Assumption \ref{ass:general_lap_const} is satisfied by some linear exponential family models see \cite{portnoy_exp}, and for some linear models, see \citet{portnoy_linear}. 
\end{remark}

\begin{theorem}\label{th:general_laplace}
If for $\alpha < 1/2 - 1/(2\zeta - 2)$ the integrals in the numerator and denominator of (\ref{eq:marg_posterior}) satisfy Assumptions \ref{ass:delta_decay_lap} -- \ref{ass:general_lap_const} under the orthogonal parametrization, then, 
\*[ 
\frac{\densmarg }{\densmargapprox } =  1 + O( e_{n,p} )   ,
\]
where
\*[
e_{n,p} 
= \max\left\{\frac{p^2\log(n)^2}{n}, \frac{p^{\zeta -1}\log(n)^{\zeta/2}}{n^{ (\zeta - 2)/2 }},\frac{p \log(n)^{1/2}}{n^{3/2 - \thirdpower}}   \right\},
\]
for all $\psi \in \{ \psi: |\psi - \maxpsi| \leq O(\log(n)^{1/2}/n^{1/2}) \}$, where
 $\zeta$ is defined in Assumption \ref{ass:exp_ratio}, $\thirdpower, \fourthpower \leq 1 $ and in Assumption \ref{ass:delta_decay_lap} holds with $e_{n,p}$ replaces $a_{n,p}$.
\end{theorem}

\begin{corollary}
Under the same Assumptions as Theorem \ref{th:general_laplace}, if additionally $\thirdpower = \fourthpower = 1$, then
\*[
\frac{ \densmarg }{ \densmargapprox} =  1 + O\left[ \max\left\{\frac{p^2\log(n)^2}{n}, \frac{p^{\zeta -1}\log(n)^{\zeta/2}}{n^{ (\zeta - 2)/2 }}  \right\} \right] ,\]
for $\psi \in \{ \psi: |\psi - \maxpsi| = O(\log(n)^{1/2}/n^{1/2}) \}$ and $\alpha < 1/2 - 1/(2\zeta - 2)$.
\end{corollary}

\begin{remark}\label{remark:minmax}
 Applying Theorem \ref{thm:laplace} to the numerator and denominator of (\ref{eq:laplace_approxiamtion}) and combine this with Theorem \ref{th:general_laplace} we can obtain a potentially improved estimate of the approximation error
 \*[ 
\frac{ \densmarg }{\densmargapprox } =  1 + O( e_{n,p} )   ,
\]
where,
\*[
e_{n,p}
&=
\min \left[\max\left( \frac{p^{3 + 2 c_{\infty}}}{n^{3 - 2\thirdpower}} , \frac{p^{2 + 2 c_{\infty}} }{n^{2 - \fourthpower}}  \right), \max\left\{ \frac{p\log(n)^{1/2}}{n^{3/2 - \thirdpower}}, \frac{p^2\log(n)^2}{n},   \frac{p^{\zeta - 1}\log(n)^{\zeta/2}}{n^{(\zeta -2)/2}} \right\} \right],
\]
if Assumption \ref{ass:delta_decay_lap} holds with $e_{n,p}$ replacing $a_{n,p}$. 
\end{remark}

\subsection{Marginal laplace approximation - linear exponential family}\label{subsection:laplace_exp}
Let $X$ be a $n \times p$ matrix of covariates with $(j,k)$ entry $x_{jk}$ and $j$th row  $x^\top_j$. We assume the density of $y_j$ is that of a full exponential family model with canonical parameter $\theta = (\psi, \tau)$.  The  log-likelihood function for an independent sample $y_1, \dots, y_n$ is
\[\label{eq:linear_exp_fam}
l(\psi, \tau; y) = \psi\Sigma_{j=1}^n(y_j x_{j1}) + \Sigma_{k = 2}^{p}\tau_k\Sigma_{j=1}^n (y_j x_{jk}) - \Sigma_{j=1}^nK( x_j^\top\theta). 
\]

As noted in \cite{ortho}, under the mean parametrization  $\lambda_{k} = \mathbb{E}[\Sigma_{j=1}^n (y_j x_{jk})/n]$ for $k = 1, \dots, p -1$, $\lambda$ is orthogonal to $\psi$. Also, under this parametrization $j_{\psi\lambda}(\hat\theta_\psi) = 0$ and supposing that the prior for $\psi$ and $\lambda$ are independent, this implies that $\loglike^{(2)}_{\psi\lambda} (\hat\theta_\psi)= 0$ and therefore $\maximizercons = (\psi, \hat\nuissance)$. 
The division by $n$ ensures that $\lambda$ stays bounded as $n \rightarrow \infty$ \citep{tang2020modified}. 

The result of Theorem \ref{thm:ratio_exp_laplace} is the same as that of Theorem \ref{th:general_laplace}, but Assumption \ref{ass:general_lap_const} is no longer needed as $\hat\lambda_\psi = \hat\lambda$ for the linear exponential family.

\begin{theorem}\label{thm:ratio_exp_laplace}
If for $\alpha \leq 1/2 - 1/2(\zeta - 1)$, the integrals in the numerator and denominator of (\ref{eq:marg_posterior}) satisfy Assumptions \ref{ass:delta_decay_lap}--\ref{ass:exp_ratio} under the orthogonal parametrization then, 
\*[ 
\frac{ \densmarg}{\densmargapprox } =  1 + O(e_{n,p})   ,
\]
where,
\*[
e_{n,p} =  \max\left\{\frac{p^2\log(n)^2}{n}, \frac{p^{\zeta -1}\log(n)^{\zeta/2}}{n^{ (\zeta - 2)/2 }},\frac{p \log(n)^{1/2}}{n^{3/2 - \thirdpower}}   \right\},
\]
for $\{ \psi: |\psi - \maxpsi| = O(\log(n)^{1/2}/n^{1/2}) \}$ , where $\zeta$ is defined in Assumption \ref{ass:exp_ratio}, $\thirdpower, \fourthpower \leq 1 $ and Assumption \ref{ass:delta_decay_lap} holds with $e_{n,p}$ replacing $a_{n,p}$.
\end{theorem}

\begin{corollary}
Under the same assumptions as Theorem \ref{thm:ratio_exp_laplace}, if  $\thirdpower = \fourthpower = 1$, then
\*[
\frac{\densmarg  }{\densmargapprox} =  1 + O\left[ \max\left\{ \frac{p^2\log(n)^2}{n},   \frac{p^{\zeta - 1}\log(n)^{\zeta/2}}{n^{(\zeta -2)/2}} \right\} \right] ,\]
for all $\psi \in \{ \psi: |\psi - \maxpsi| = O(\log(n)^{1/2}/n^{1/2}) \}$ and $\alpha < 1/2 - 1/(2\zeta - 2)$.
\end{corollary}

Remark \ref{remark:minmax} applies to Theorem \ref{thm:ratio_exp_laplace} as well, meaning that we may apply Theorem \ref{thm:laplace} to the numerator and denominator of (\ref{eq:marg_posterior}) and combined this with Theorem \ref{thm:ratio_exp_laplace} to obtain a potentially improved error rate.

\begin{remark}
It can be shown that under Assumption \ref{ass:delta_decay_lap} the posterior mass for the marginal distribution of $\psi$ concentrates in a $O\{\log(n)^{1/2}n^{-1/2}\}$ neighbourhood of $\hat\psi$ using the same proof technique as Lemma $\ref{lemma:annulus}$. 
\end{remark}

\begin{remark}
Theorems \ref{th:general_laplace} and \ref{thm:ratio_exp_laplace} still hold if the parameter of interest is a vector, so long as its dimension does not scale with $n$. It may be of interest to extend these Theorems to the case where the dimension of $\psi$ is increasing with $n$.
\end{remark}

\section{Saddlepoint approximation}\label{section:saddle_dens}

\subsection{Complex notation}
We use complex scalars, vectors and matrices below; with real and imaginary parts $\Re(\cdot)$ and $\Im(\cdot)$, respectively, and modulus $|\cdot|$; for example 
\*[A = \Re(A) + i \Im(A). \]
We write a function taking complex input and returning a real number as $f(t) = f(x,y)$, where $t = x + iy \in \mathbb{C}^p$ and $x, y \in \mathbb{R}^p$.
When taking a directional derivative of $f(x,y)$, we denote the $k$-th order derivative along the $x$ (real) and $y$ (imaginary) axes by $f^{(x, k)}$ and $f^{(y, k)}$, respectively.

\subsection{Main theorem}
The key result which allows us to approximate the density of a $p$-dimensional random variable $\rv$ through the saddlepoint approximation is Levy's inversion theorem. Let
\*[\log\{M_{\rv}(t )\} = K_{\rv}(t)= U_{\rv}(x,y) + iV_{\rv}(x,y),\]
where $M_{\rv}(t)$ is the moment generating function of $\rv$, while $U_{\rv}(x,y)$ and $V_{\rv}(x,y)$ are the real and imaginary components of the cumulant generating function. Using Levy's inversion theorem
\begin{align}
     f_{\rv}(\approxpoint) &=  \frac{1}{(2\pi)^{p}} \int_{\mathbb{R}^p} M_{\rv}(it ) \exp\{-it^\top \approxpoint\} dt \nonumber \\
     &= \frac{1}{(2\pi)^{p}} \int_{\mathbb{R}^p}  \exp\{ K_{\rv} (0, y) - iy^\top \approxpoint\} dy. \label{eq:inversion}
\end{align}
We may deform the path of integration component-wise in (\ref{eq:inversion}), so long as there are no singularities or the singularities are not enclosed in the contour drawn by the new and old paths, by Cauchy's residual theorem. 
A strategic choice of deformation is to integrate along a line which crosses the \textit{saddlepoint}, defined as the point $\saddle$ such that
\begin{align}
     \frac{\partial}{\partial x}K_{\rv}(x, 0)|_{x = \saddle} = \approxpoint. \label{eq:saddle_def}
\end{align}
Then
\*[f_X(s)
&=\frac{1}{(2\pi )^{p}} \int_{\mathbb{R}^p }  \exp\{ K_{\rv} (\saddle, y ) - \saddle^\top \approxpoint - iy^\top \approxpoint\} dy, \]
since, as noted by \cite[Proof of Lemma 1]{kolassa2003} this is equivalent to (\ref{eq:inversion}), although here we choose to denote the change in the path of integration by a location change in the exponential term. 
Along this path, Laplace's method \citep{laplace_original} is then used to estimate the integral, which results in the following density approximation:
\[ \hat{f}_{\rv}(\approxpoint) = \frac{\exp\{K_{\rv}(\saddle, 0) - \saddle^\top \approxpoint \} }{ (2\pi)^{p/2} |U^{(x,2) }(\saddle, 0)|^{1/2}}. \label{eq:saddlepoint_approx}\]
We show that under regularity conditions, an upper bound on the approximation error is obtained 
if $p = O(n^{\alpha})$ for certain values of $\alpha < 1$. 
The proof given here differs from \cite{daniels1954}, who defined a new path of integration implicitly in order to make the integrand exactly locally quadratic.
We found this approach quite difficult to adapt to the high-dimensional setting, as the order of terms in the expansions are no longer obvious.
Instead we follow a similar approach to the proof of Theorem \ref{thm:laplace}, with some modifications. 
%

\begin{remark}
Note that 
\*[ \det\{U_{\rv}^{(x,2) }(\saddle, 0)\}^{1/2} = \det\{K^{(2) }(\saddle)\}^{1/2}, \quad \frac{\partial}{\partial x}K_{\rv}(x, 0) = K_{\rv}^{(1)}(x), \]
if the cumulant generating function $K_{\rv}(\cdot)$ is seen as a map from $\mathbb{R}^p \rightarrow \mathbb{R}$, as in \cite{daniels1954, kolassa2006series}.
We also allow the cumulant generating function to be evaluated at a point which may contain a non-zero imaginary component.
\end{remark}
We write $U(\cdot, \cdot) = U_{\rv}(\cdot , \cdot)$ and $V(\cdot, \cdot) = V_{\rv}(\cdot, \cdot)$. Fix $\delta > 0 $, $\gamma_n^2 = \log(n)p/n $.
\begin{assumption} \label{ass:delta_decay}
\begin{align*}
     &\left| \frac{ \det\{U^{(x,2) }(\saddle, 0)\}^{1/2} }{(2\pi)^{p/2} } \int_{  B_{\zerop}^C(\delta)}   \exp\{ K_{\rv} (\saddle , y) - K_{\rv} (\saddle, 0) - iy^\top \approxpoint  \} dy \right| \\
    &= O\left( a_{n,p} \right),
\end{align*}
for a sequence $a_{n,p} \rightarrow 0$ as $n \rightarrow \infty$.
\end{assumption}

\begin{assumption} \label{ass:hess}
The eigenvalues of the second derivative of the real part of the cumulant generating function satisfy:
\*[ 0< \eta_1 n \leq \lambda_p\left[   U^{(x, 2)}(\saddle, y) \right] \leq \lambda_1\left[  U^{(x, 2)}(\saddle, y) \right] \leq \eta_2 n ,\]
for all $y \in B_{\zerop}(\delta)$, and $\norm{\{U^{(x, 2)}(\saddle, y)\}^{-1/2}}_\infty = O(p^{c_{\infty}}/n^{1/2})$.
\end{assumption}

\begin{assumption} \label{ass:third_cum}
The eigenvalues of the sub-matrices $U^{(x,3)}_{\cdot \cdot l}$, whose $j,k$ entries are $ [U^{(x,3)}_{\cdot \cdot l}(\saddle, 0 )]_{jk} = U^{(x,3)}_{jkl}(\saddle, 0 )$ satisfy
\*[ \eta_{3} n^{\thirdpower} \leq \lambda_p[ U_{\cdot\cdot l}^{(x,3)}(\saddle, 0 )] \leq \lambda_1[ U_{\cdot\cdot l}^{(x,3)}(\saddle, 0 ) \}] \leq \eta_{4} n^{\thirdpower} ,\]
for all $l= 1, \dots, p$, for some constants $\eta_3, \eta_4 \in \mathbb{R}$.
\end{assumption}

\begin{assumption} \label{ass:fourth_cum}
The eigenvalues of the sub-matrices $U^{(x,4)}_{\cdot \cdot lm}$ and $V^{(x,4)}_{\cdot \cdot lm}$, whose $(j,k)$ entries are $ [U^{(x,4)}_{\cdot \cdot lm}(\saddle , y)]_{jk} = U^{(x,4)}_{jklm}(\saddle, y )$ and $ [V^{(x,4)}_{\cdot \cdot lm}(\saddle , y)]_{jk} = V^{(x,4)}_{jklm}(\saddle, y )$ satisfy
\begin{align*}
    \eta_{5} n^{\fourthpower} \leq \lambda_p[ U^{(x,4)}_{\cdot \cdot lm}(\saddle , y)] \leq \lambda_1[ U^{(x,4)}_{\cdot \cdot lm}(\saddle , y) ] \leq \eta_{6} n^{\fourthpower}, \\
    \eta_{5} n^{\fourthpower} \leq \lambda_p[ V^{(x,4)}_{\cdot \cdot lm}(\saddle , y)] \leq \lambda_1[ V^{(x,4)}_{\cdot \cdot lm}(\saddle , y) ] \leq \eta_{6} n^{\fourthpower},
\end{align*}  
for all $y \in B_{\zerop}(2^{1/2}\gamma_n)$ and for all $l, m = 1, \cdots, p$.

\end{assumption}

These assumptions are similar to Assumptions \ref{ass:delta_decay_lap}--\ref{ass:fourth_lap} in Section \ref{sec:laplace}. 

\begin{theorem} \label{th:density_approx}
For a sequence $\approxpoint$ satisfying Assumptions \ref{ass:delta_decay}--\ref{ass:fourth_cum}, with Assumption \ref{ass:delta_decay} holding with $a_{n,p} = \max\left(p^{3 + 2c_{\infty}}/n^{3 - 2\thirdpower}, p^{2 + 2c_{\infty}} /n^{2 - \fourthpower} \right)$, the saddlepoint approximation (\ref{eq:saddlepoint_approx}) satisfies
\begin{align*}
 \frac{f_{X_n}(\approxpoint)}{\hat{f}_{\rv}(\approxpoint)} =  1 + O\left\{ \max\left(\frac{p^{3 + 2c_{\infty} }}{n^{3 - 2\thirdpower}},  \frac{p^{2+ 2c_{\infty}}}{n^{2 - \fourthpower}} \right)   \right\} ,
\end{align*}
for $p = O( n^\alpha)$, $\alpha <  (4 - 2\fourthpower)/(5 + 4c_{\infty}) $. \label{th:density_pointwise}
\end{theorem}

The comments in \S\ref{sec:laplace} on improving the error rate apply here, due to the similarity in the approaches. 

\begin{remark}
In a $p$-fixed setting, where $\alpha = 0 $, we recover the usual $\{1 + O(n^{-1})\}$ relative error rate as in \cite{daniels1954}. This gives an alternative proof for the accuracy of the saddlepoint approximation in the $p$-fixed case, although our assumptions differ.
\end{remark}

\begin{remark}
Theorem \ref{th:density_pointwise} is stated for general random vectors that have potentially dependent components. If the components of the random vectors are independent or perhaps block dependent, one can obtain better results than Theorem \ref{th:density_pointwise}. In particular in the independent component case, one may simply apply the saddlepoint approximation to each component, and take the product of the marginal approximations as the approximation to the joint density.
\end{remark}

\begin{remark}
Assumption \ref{ass:delta_decay} is satisfied in a $p$-fixed asymptotic regime if:
\*[ \int_{\mathbb{R}^p}| \xi_{\rv} (t) |dt  < \infty,  \]
but in high-dimensional settings it is possible that as $p \rightarrow \infty$, this integral to tends infinity. 
Consider for example: 
\*[ \int_{\mathbb{R}^p}| \xi_{Z} (t) |dt  = \int_{\mathbb{R}^p} \exp\{ -\frac{1}{2}t^\top t \} dt = (2\pi)^{p/2} \rightarrow \infty, \quad p \rightarrow \infty, \]
where $Z$ is a multivariate normal random variable, with mean 0 and identity covariance matrix.
\end{remark}


\subsection{Uniformity of the approximation }
In some applications, uniform accuracy for the density approximation over a set of points is desired. 
As in the finite-dimensional case, this can be achieved by adding some form of uniformity in the assumptions. 
Let $A_n \subset \mathbb{R}^p$ be the set of points at which the density approximation is desired,  $T_n$ denote the set of saddlepoints obtained for points $s_n \in A_n$, and $\delta > 0$ be a constant independent of $p$ and $n$. 

\begin{customass}{$7^\prime$}
\begin{align*}
   &\left| \frac{ \det\{U^{(x,2) }(\saddle, 0)\}^{1/2} }{(2\pi)^{p/2} } \int_{  B_{\zerop}^C(\delta)}   \exp\{ K_{\rv} (\saddle , y) - K_{\rv} (\saddle, 0) - iy^\top \approxpoint  \} dt \right| \\
    &= O\left\{ \max\left(\frac{p^{3 + 2c_{\infty}}}{n^{3 - 2\thirdpower}},\frac{p^{2+ 2c_{\infty}}}{n^{2 - \fourthpower}} \right)\right\}
\end{align*}  
for all $\saddle \in T_n$, uniformly in $s_n \in A_n$.
\end{customass}

\begin{customass}{$8^{\prime}$}
The eigenvalues of the second derivative of the real part of the cumulant generating functions satisfy:
\*[ 0< \eta_1 n \leq \lambda_p\left[   U^{(x, 2)}(\saddle, y) \right] \leq \lambda_1\left[   U^{(x, 2)}(\saddle, y) \right] \leq \eta_2 n ,\]
and $\norm{\{U^{(x, 2)}(\saddle, y)\} ^{-1/2}}_\infty = O(p^{c_{\infty}}/n^{1/2})$
for all $\saddle \in T_n$ and $y \in B_{\zerop}(\delta)$. 
\end{customass}

\begin{customass}{$9^\prime$}
The eigenvalues of $U^{(x,3)}_{\cdot \cdot lm}$, whose $(j,k)$ entries are $ [U^{(x,3)}_{\cdot \cdot lm}(\saddle, 0 )]_{jk} = U^{(x,3)}_{jkl}(\saddle, 0 )$ satisfy
\*[ \eta_{3} n^{\thirdpower} \leq \lambda_p[ U^{(x,3)}(\saddle, 0 )] \leq \lambda_1[ U^{(x,3)}(\saddle, 0 ) \}] \leq \eta_{4} n^{\thirdpower} ,\]
for all $\saddle \in T_n$ and $l= 1, \dots, p$, for some constants $\eta_3, \eta_4 \in \mathbb{R}$.
\end{customass}

\begin{customass}{$10^\prime$}
The eigenvalues of the sub-matrices $U^{(x,4)}_{\cdot \cdot lm}$ and $V^{(x,4)}_{\cdot \cdot lm}$, whose $(j,k)$ entries are $ [U^{(x,4)}_{\cdot \cdot lm}(\saddle , y)]_{jk} = U^{(x,4)}_{jklm}(\saddle, y )$ and $ [V^{(x,4)}_{\cdot \cdot lm}(\saddle , y)]_{jk} = V^{(x,4)}_{jklm}(\saddle, y )$ satisfy
\begin{align*}
    \eta_{5} n^{\fourthpower} \leq \lambda_p[ U^{(x,4)}_{\cdot \cdot lm}(\saddle , y)] \leq \lambda_1[ U^{(x,4)}_{\cdot \cdot lm}(\saddle , y) ] \leq \eta_{6} n^{\fourthpower} \\
    \eta_{5} n^{\fourthpower} \leq \lambda_p[ V^{(x,4)}_{\cdot \cdot lm}(\saddle , y)] \leq \lambda_1[ V^{(x,4)}_{\cdot \cdot lm}(\saddle , y) ] \leq \eta_{6} n^{\fourthpower}
\end{align*}  
for all $\saddle \in T_n$ and $y \in B_{\zerop}(2^{1/2}\gamma_n)$ and for all $l, m = 1, \cdots, p$.
\end{customass}

\begin{corollary}
Under Assumptions $7^\prime$ -- $10^\prime$,
\begin{align*}
\frac{f_{X_n}(\approxpoint)}{\hat{f}_{\rv}(\approxpoint)} = 1 + O\left\{ \max\left(\frac{p^{3+ 2c_{\infty}}}{n^{2 - \thirdpower}},  \frac{p^{2 + 2c_{\infty}} }{n^{2 - \fourthpower}} \right)   \right\}  ,
\end{align*}
for $p = O( n^\alpha)$, and $\alpha <  (4 - 2\fourthpower)/(5 + 4c_{\infty}) $ uniformly in $s_n \in A_n$.
\end{corollary}
\begin{remark}
The assumptions required for the uniformity of the density approximation are more strict for the saddlepoint approximation than for the Laplace approximation in \S\ref{sec:laplace}, 
because the inversion required to obtain the density must be performed point-wise for the saddlepoint approximation, whereas the Laplace approximation provides the entire posterior density.
\end{remark}

\begin{example} \textbf{Exponential regression}. 
Let $y_j$, $j = 1, \dots, n$, be independent observations from an exponential distribution with rate parameter $\lambda_j = x_j^{\top}\beta_0$. $x_j^{\top}$ is the $j$-th row of the design matrix $X$, and is independently generated from an isotropic Gaussian distribution with covariance $\sigma_0 I$ for some $\sigma_0 > 0$ and $\beta_0$ is the data generating vector parameter of length $p$. The likelihood is:
\*[
    l(\beta; X, Y) = \sum_{j = 1}^n \log(x_j^\top \beta) - x_j^
    \top\beta y_j = \sum_{j = 1}^n \log(x_j^\top \beta) - \sum_{k = 1}^p \sum_{j = 1}^n x_{jk} y_j \beta_j,
\] 
and the vector of sufficient statistics is $S = ( -\sum^n_{j = 1}x_{j1}y_j, \dots, -\sum^n_{j = 1}x_{jp}y_j)$. 
We are interested are estimating the density for a value of $S = s$ such that $A_1 < |x_j^\top \beta| < A_2$ for all $j = 1, \dots, n$ and we assume that the the true data generating parameter $B_1 < |x_j^\top \beta_0| < B_2$. 

    \begin{corollary}
        For an exponential regression model satisfying the assumptions listed above:
        \*[
            \frac{f_{X_n}(\approxpoint)}{\hat{f}_{\rv}(\approxpoint)} = 1 + O\left( \frac{p^3\log(n)}{n} \right).
        \]
    \end{corollary}
    The proof is similar to that of Example \ref{example:logis_laplace}, and the rate obtained matches that of Example \ref{ex:glm}. 
\end{example}

\section{Conditional inference}\label{sec:ratio_saddle}
We now consider conditional inference in the linear exponential family (\ref{eq:linear_exp_fam}), see \cite{davison_cond}.
The results are stated and proved for a scalar parameter of interest, although the results still hold if the dimension of the parameter of interest does not grow with $n$. 
We modify the notation for the cumulant generating function, let $t = (\tsplit) = (x_{\interest}, x_{\nuissance}) +i (y_{\interest}, y_{\nuissance})$ for $ x = (x_{\interest}, x_{\nuissance}), y =  (y_{\interest}, y_{\nuissance}) \in \mathbb{R}^p$ and
\*[
K_{\suffsplit}(\tsplit)&= K_{\suffsplit}\{  (x_{\interest}, x_{\nuissance}) +i (y_{\interest}, y_{\nuissance}) \} \\
&= U\{ (x_{\interest}, x_{\nuissance}), (y_{\interest}, y_{\nuissance}) \} + iV\{ (x_{\interest}, x_{\nuissance}), (y_{\interest}, y_{\nuissance}) \},
\] 
where $s_1$ is the component of the minimal sufficient statistic associated with the parameter of interest $\psi$, and $s_2$ is the component of the minimal sufficient statistic associated with the nuisance parameters $\nuissance$.

The conditional distribution of $s_1$ given $s_2$ is free of $\psi$, so
\*[ \log\{ f(s_1, s_2; \interest, \nuissance) \} = \log\{ f(s_1| s_2; \interest) \} + \log\{ f(s_2; \interest, \nuissance) \}, \]
and inference may be based on $\log\{ f(s_1| s_2; \interest) \}$ with the implicit assumption that there is minimal information lost by ignoring the second component. In most practical circumstances the conditional distribution is not known and needs to be approximated, and we can use the saddlepoint approximation to approximate the numerator and denominator of
\[\label{eq:conditional_dist}
f(s_1|s_2; \psi) = \frac{f(s_1, s_2; \psi, \lambda)}{f(s_2; \psi, \lambda)},
\]
to obtain an approximation of the conditional density, see \citet[~\S7]{kolassa2006series}.
This is sometimes called the double saddlepoint approximation, 
as it requires us to solve two separate saddlepoint equations. 
The double saddlepoint approximation is
\[
\hat{f}(s_1| s_2; \interest) &= \left( \frac{ \det\left[ U^{(x_{\nuissance},2) }\{(0,\tilde{t}_\lambda ), \zerop \} \right] }{2\pi \det\left[ U^{(x,2) }\{(\saddleint,\saddlenuis), \zerop \} \right]} \right)^{1/2}\nonumber \\
\quad&\times \exp\left[ K_{\suffsplit}(\saddleint,\saddlenuis) - K_{\suffsplit}(0,\tilde{t}_\lambda) +\tilde{t}_\lambda^\top s_2 - (\saddleint, \saddlenuis)^\top (s_1, s_2) \right], \label{eq:double_saddle}
\]
where the saddlepoints are the solutions to 
\*[ 
\frac{\partial}{\partial t} K_{\suffsplit}(\tsplit)|_{(\saddleint, \saddlenuis)} = \left(\begin{matrix}  s_1 \\ s_2 \end{matrix} \right), \quad \frac{\partial}{\partial t_{\nuissance}} K_{\suffsplit}( 0, t_{\nuissance})|_{\tilde{t}_\lambda} = s_2.
\]

%

\begin{corollary}\label{thm:saddle_cond}
If numerator and denominator of (\ref{eq:conditional_dist}) satisfy Assumptions \ref{ass:delta_decay}--\ref{ass:fourth_cum}, then

\*[ \frac{ f(s_1|s_2; \psi) }{\hat{f}(s_1|s_2; \psi) } = 1 +  O\left\{ \max\left(\frac{p^{3+2c_{\infty}}}{n^{3 - 2\thirdpower}},\frac{p^{2 + 2c_{\infty}}}{n^{2 - \fourthpower}} \right) \right\}, \]
where Assumption \ref{ass:delta_decay} holds with $a_{n,p} = \max(p^{3 + 2c_\infty}/n^{3 - 2\thirdpower}, p^{2 +2c_{\infty}}/n^{2 - \fourthpower})$, for $\alpha <  (4 - 2\fourthpower)/(5 + 4c_{\infty}) $
\end{corollary}

The proof is immediate from applying Theorem \ref{th:density_approx} to the numerator and denominator of (\ref{eq:conditional_dist}). 
The saddlepoints in this example can also be written as functions of the mle and constrained mle, $(\saddleint, \saddlenuis) = (\hat\psi_{mle} - \psi, \hat{\lambda}_{mle} - \lambda )$, $\tilde{t}_\lambda = \hat{\lambda}_{\psi, mle} - \lambda$ \citep[\S4]{davison_cond}. 
It is more difficult to show that a cancellation in the ratio of error terms occur for the approximate conditional density as the saddlepoint equations cannot be solved independently, i.e. $\hat\lambda_{\psi,mle} \neq \hat\lambda$, which is why the result does not improve on Theorem \ref{thm:ratio_exp_laplace}.

The saddlepoint approximation is exact after renormalization for, and only for the gamma, normal and inverse normal distributions, and this fact can be exploited to construct exact tests in high dimensions. 
The following example is taken from \cite{vector_linear}, in which a directional test constructed from the saddlepoint approximation was shown to perform quite well numerically; we show that this test is in fact exact. This has been also noted by \cite{saddle_high}, although the argument and exposition differs slightly. 

\begin{example}\label{ex:exp_means}
    \textbf{Equality of exponential means}. Let $y_{jk}$ be independent random exponential variables parametrized by rates $\eta_j$, 
    for $j = 1, \dots, g$ and $k = 1, \dots, m$. We assume sample size to be the same for notational convenience, the results shown are valid 
    if the sample size in each group is of the same asymptotic order. 
    We wish to test the hypothesis $\eta_1 = \dots = \eta_g$, the alternative 
    hypothesis is that at least one equality does not hold. Let $u_j = \sum_{k = 1}^m y_{jk}$. The log-likelihood under the $\eta$ parametrization is
    \*[
        l(\eta; y) = \sum_{j = 1}^g( -u_j \eta_j + m \log\eta_j), 
    \]
    although for working under the null it is better to consider the parametrization $\psi_j = \eta_{j + 1} - \eta_{j}$,
    for $j = 1, \dots, g -1$ and $\lambda = \eta_1$. The null hypothesis is then $\psi_1 = \dots = \psi_{g-1} = 0$. The log-likelihood function under this alternative parametrization is: 
    \*[
        l(\psi, \lambda; y) =  -\sum_{j = 1}^{g - 1}  \sum_{k = 1}^{j} u_k \psi_j - \sum_{j =1}^g u_i \lambda  + \sum_{j =1}^{g-1} \log() \lambda + \sum_{k =1} \psi_k )  +\log(\lambda)
    \]
    which is still of the linear exponential form, although there is now a high level of dependence among the components of the sufficient statistics. 

    In the Supplementary Materials we show that under the null and alternative, the joint approximation on the numerator and denominator are individually accurate to $1 + O(g/m) = 1 + O(p^2/n)$ for values of the sufficient statistic on the line element $s(t)$ defined in \cite{vector_linear}.
    However, due to the specific structure of the exponential distribution the directional test is in fact exact for all $m \geq 2$ regardless of the number of groups, the saddlepoint approximation is exact after renormalization for the gamma distribution and the directional test only requires the ratio of densities. 

    It is also possible to exploit this fact to construct exact tests for normal covariances in the high-dimensional setting as in \cite{directional_normal}. Although this will only be possible for directional tests for the normal, inverse normal or gamma distributions, as these are the only distributions for which the saddlepoint approximation is exact after renormalization, see \citet[~\S4]{kolassa2006series}.
\end{example}

\section{Conclusion}\label{sec:conclusion}
Although we have provided a reasonable worst case approximation error for the Laplace and saddlepoint approximations with Theorems \ref{thm:laplace} and \ref{th:density_pointwise}, these might be pessimistic for some applications.
In particular the Laplace approximation is often used in spatial models where the number of parameters exceed the number of observations, and empirically these approximations seem to be quite accurate. 
It may be possible to obtain stronger results by examining such models individually and using the techniques developed in this work.
Some interesting extensions would be:
\begin{itemize}
\item Expand on the role of the prior in $p > n $ settings. To extend our results on the Laplace approximation to $p > n$, it is essential to characterize the prior distribution in greater detail, as in high dimensions the prior acts as a regularizer. As suggested by the Associate Editor, a good starting point would be to examine the Laplace approximation for normal means model or possibly regression models with the horseshoe prior.
\item Examine the Laplace approximation for complex non-parametric or semi-parametric models.
Although empirically the use of the Laplace approximation seems to produce good results for approximating the density of these models, hence the success of INLA \citep{inla}, the theoretical justification remains limited.  
Based on our expansions, the posterior of the model will need to look highly Gaussian in the sense that the cumulants need to be small for the approximation error to be asymptotically negligible.
\item Consider the tail area approximations that can be obtained from the double saddlepoint and the marginal Laplace approximation, see for example \cite{reid2003}. 
Typically these are used for inference to approximate $p$-values and confidence regions.  
\item Extend the marginal and conditional approximation to the case where the dimension of the parameter of interest is increasing with the number of observations. This may extend the results of \cite{vector_linear} and \cite{vector} to the high-dimensional regime.
\item Derive lower bounds on the approximation error of the Laplace and saddlepoint approximation, this would extend the result of \cite{bilodeau2023tightness} to the high-dimensional setting.
It is unclear at the moment if the upper bounds in the major theorems have matching lower bounds, based on empirical observations, we hypothesize that a lower bound will most likely be met by a highly non-linear model. 
For the Gaussian approximation a lower bound of $O(p^{1/2}/n)$
\item Examine the effect of renormalizing the approximation to the marginal posterior density and the approximation to the conditional distribution. 
Since the dimension of the parameter of interest tends to be small, it may still be possible (although still potentially quite computationally involved) to renormalize the marginal approximation. 
This may lead to an improvement in the accuracy of the approximation as in \cite{tierney1986}.
\end{itemize} 
 
\section{Acknowledgements}
The authors would like to thank the Associate Editor and the Reviewers for their comments and questions.
We would also like to thank Michaël Lalancette, Blair Bilodeau and Alex Stringer for their comments and feedback on earlier versions of the manuscript.
This research was partially supported by the Natural Sciences and Engineering Research Council of Canada and the Vector Institute. 
This work was partially completed while YT was affiliated with the University of Toronto during his PhD and while YT was a Chapman Fellow at Imperial College.



\bibliographystyle{imsart-nameyear}
\bibliography{biblio}
\appendix

\newpage
\setcounter{page}{1}
\title{Supplementary materials for: ``Laplace and saddlepoint approximations in high dimensions''}
\section{Proof of main theorems}

\subsection{Proof of Theorem \ref{thm:laplace}}

\begin{customthm}{\ref{thm:laplace}}
    Let $p = O(n^\alpha)$, $ \alpha < \min\{ (3 - 2\thirdpower)/(3 + 2c_{\infty}) , (4 -2 \fourthpower)/(5 + 4c_{\infty} )\}$. For a sequence  $\{ X_n\}$ satisfying Assumptions \ref{ass:delta_decay_lap}--\ref{ass:fourth_lap}, and in Assumption \ref{ass:delta_decay_lap}, $a_{n,p} = \max\left\{p^{3 + 2c_{\infty}}/n^{3 - 2\thirdpower}, p^{2 + 2c_{\infty}}/n^{2 - \fourthpower} \right\}$ ,
    \*[\frac{f(\theta|X_n)}{\hat{f}(\theta|X_n)} =    1 +  O\left\{ \max\left(\frac{p^{3 + 2 c_{\infty}} }{n^{3 - 2\thirdpower} }, \frac{p^{2 + 2 c_{\infty}} }{n^{2 - \fourthpower}}  \right) \right\} . \]
\end{customthm}

\begin{proof}
    Combining (\ref{eq:posterior_def}) and (\ref{eq:laplace_approxiamtion})
    \begin{align*}
        &\frac{\hat{f}(\theta'|X_n)}{f(\theta^\prime|X_n)} =  \frac{ \det\{-g_n^{(2)}(\maximizer) \}^{1/2}}{ (2\pi)^{p/2}  } \int_{ \mathbb{R}^p }  \exp\{ g_n(\theta^\prime; X_n)  - g_n(\maximizer ; X_n)  \} d\theta^\prime \\
        &=\frac{\det\{-g_n^{(2)}(\hat\theta) \}^{1/2}}{ (2\pi)^{p/2}  } \left[ \int_{ B_{\maximizer}(\delta) }  \exp\{ \loglike_n(\theta^\prime ; X_n)  - \loglike_n(\maximizer ; X_n)  \} d\theta^\prime \right. \\
        &\quad+ \left. \int_{ B^C_{\maximizer}(\delta) }  \exp\{ g_n(\theta^\prime ; X_n)  - g_n(\maximizer ; X_n)  \} d\theta^\prime \right]\\
        &= \frac{\det\{-g_n^{(2)}(\hat\theta) \}^{1/2}}{ (2\pi)^{p/2}  }  \int_{ B_{\maximizer}(\delta) }  \exp\{ g_n(\theta^\prime; X_n)  - g_n(\maximizer ; X_n)  \} d\theta^\prime + O\left\{ \max \left(\frac{p^{3 + 2c_\infty }}{n^{3 - 2\thirdpower} }, \frac{p^{2+ 2c_\infty} }{n^{2 - \fourthpower}}  \right) \right\} ,
    \end{align*}
    by Assumption \ref{ass:delta_decay_lap}. By Lemma \ref{lemma:annulus}, 
    \*[
    &\frac{\det\{-g_n^{(2)}(\hat\theta) \}^{1/2}}{ (2\pi)^{p/2}  }  \int_{ B_{\maximizer}(\delta) }  \exp\{ g_n(\theta^\prime; X_n)  - g_n(\maximizer ; X_n)  \} d\theta^\prime \\
    &= \frac{\det\{-g_n^{(2)}(\hat\theta) \}^{1/2}}{ (2\pi)^{p/2}  } \Big[ \int_{ B_{\maximizer}(\gamma_n) }  \exp\{ g_n(\theta^\prime; X_n)  - g_n(\maximizer ; X_n)  \} d\theta^\prime \\
    &\quad+ \int_{ B_{\maximizer}(\delta) \cap B^C_{\maximizer}(\gamma_n) }  \exp\{ g_n(\theta^\prime; X_n)  - g_n(\maximizer ; X_n)  \} d\theta^\prime \Big] \\
    &= \frac{\det\{-g_n^{(2)}(\maximizer)\}^{1/2}}{ (2\pi)^{p/2}  }  \int_{  B_{\maximizer}(\gamma_n) }  \exp\{ g_n(\theta^\prime; X_n)  - g_n(\maximizer ; X_n)  \} d\theta^\prime  + O\left( n^{-\eta_1 p/4} \right).
    \]
    The second term decays exponentially fast in $p$, so we need only consider the truncated integral:
    \begin{align}
        &\frac{\det\{-g_n^{(2)}(\maximizer) \}^{1/2}}{ (2\pi)^{p/2}  }  \int_{  B_{\maximizer}(\gamma_n) }  \exp\{ g_n(\theta^\prime ; X_n)  - g_n(\maximizer ; X_n)  \} d\theta^\prime  \nonumber \\
        &= \frac{ \det\{-g_n^{(2)}(\maximizer) \}^{1/2} }{(2\pi)^{p/2} } \int_{ B_{\zerop}(\gamma_n) }  \exp\left\{ \frac{1}{2} \theta^{\top} \loglike^{(2)}(\maximizer)\theta  + R_{3,n}(\theta, \maximizer) + R_{4,n}(\theta, \tilde{\theta}) \right\} d\theta \label{eq:laplace_secondeq} \\
        &=  \int_{ B_{\zerop}(\gamma_n) }  \exp\left\{ R_{3,n}(\theta, \maximizer) + R_{4,n}(\theta, \tilde{\theta}) \right\} \phi\left[ \theta; 0, \{-\loglike^{(2)}(\maximizer)\}^{-1} \right] d\theta \label{eq:laplace_expansion},
    \end{align}
     where,
    \begin{align*}
        R_{3,n}(\theta, \maximizer) = \frac{1}{6}\sum_{j = 1}^p \theta_j \left\{ \theta^\top \loglike^{(3)}_{\cdot\cdot j}(\maximizer) \theta \right\}, \quad R_{4,n}(\theta, \tilde{\theta})= \frac{1}{24}\sum_{j = 1}^p\sum_{k = 1}^p \theta_j \theta_k \left\{ \theta^\top \loglike^{(4)}_{\cdot\cdot jk}(\tilde{\theta}) \theta \right\}, 
    \end{align*}
    and $\tilde{\theta} = \tau(\theta) \theta +\{ 1- \tau(\theta)\}\maximizer$, where  $0 \leq \tau(\theta) \leq 1$.
    Equation (\ref{eq:laplace_secondeq}) follows from a fourth-order Taylor expansion and a change of variable to $\theta = \theta^\prime - \maximizer$.
    Applying another change of variable $\thetanew =n^{-1/2} \Sigma^{1/2}\theta$, where $\Sigma^{1/2}$ is a square root of the matrix $-\loglike^{(2)}(\maximizer)$, 
     \*[
     (\ref{eq:laplace_expansion}) &= \int_{ E_{\zerop}(\gamma_n, n^{-1/2}\Sigma^{1/2}) }  \exp\left\{ \bar{R}_{3,n}(\thetanew) + \bar{R}_{4,n}(\thetanew, \tilde{\theta})  \right\} \phi\left( \thetanew; 0, I_p/n \right) d\thetanew , 
     \]
    where $E_{\zerop}(\gamma_n, n^{-1/2}\Sigma^{1/2})$ is an ellipsoid defined by $\lVert n^{1/2} \Sigma^{-1/2} \thetanew \rVert_2 \leq \gamma_n$, and 
    \begin{align*}
        \bar{R}_{3,n}(\thetanew) &= \frac{1}{6}\sum_{j = 1}^p \thetanew_j \left\{ \thetanew^\top A_j \thetanew \right\}, \quad \bar{R}_{4,n}(\thetanew, \tilde{\theta}) = \frac{1}{24}\sum_{j = 1}^p\sum_{k = 1}^p \thetanew_j \thetanew_k \left\{ \thetanew^\top B_{jk}(\tilde{\theta}) \thetanew \right\}. 
    \end{align*}
    By Lemma \ref{lemma:spherical} the matrices $ \norm{A_j}_{op} = O(p^{c_{\infty}}n^{\thirdpower }) $ and $\lVert B_{jk}(\tilde\theta) \rVert_{op} = O(p^{2c_{\infty}}n^{\fourthpower}) $  for all $j, k = 1 \dots , p$ and for all $\thetanew \in E_{\zerop}(\gamma_n, n^{-1/2}\Sigma^{1/2})$. 
    An upper bound can be obtained by expanding,
    \begin{align}
        &\exp[    \bar{R}_{3,n}(\thetanew) + \bar{R}_{4,n}(\thetanew, \tilde{\theta})  ] \nonumber\\
        &= 1 + \bar{R}_{3,n}(\thetanew) + \bar{R}_{4,n}(\thetanew, \tilde{\theta}) + \frac{1}{2}\{ \bar{R}_{3,n}(\thetanew) + \bar{R}_{4,n}(\thetanew, \tilde{\theta}) \}^2 \exp(R_{\exp}) \nonumber\\
        &\leq 1 + \bar{R}_{3,n}(\thetanew) + \bar{R}_{4,n}(\thetanew, \tilde{\theta})   + \{ \bar{R}^2_{3,n}(\thetanew) + \bar{R}^2_{4,n}(\thetanew, \tilde{\theta})  \}  \exp[ \max\{  0, \bar{R}_{3,n}(\thetanew) + \bar{R}_{4,n}(\thetanew, \tilde{\theta})\} ], \label{eq:thm2_terms}
    \end{align}
    where $R_{\exp}$ lies between $0$ and $\bar{R}_{3,n}(\thetanew) + \bar{R}_{4,n}(\thetanew, \tilde{\theta})$, we used Young's inequality, $2xy \leq x^2 + y^2$, and $\exp(-s) < \exp(0)$ for $s > 0$ in (\ref{eq:thm2_terms}). 
    It remains to consider the integrals of the terms in (\ref{eq:thm2_terms}) against a normal density. The integral of $\bar{R}_{3, n}(\thetanew)$ is $0$,
    as it is the integral of an odd polynomial over a symmetric set against the density of a centered multivariate normal. 
    By Lemma \ref{lemma:expectation }
    \begin{align*}
        \left| \int_{ E_{\zerop}(\gamma_n, n^{-1/2}\Sigma^{1/2} ) } \bar{R}_{4,n}(\thetanew, \tilde{\theta}) \phi\left( \thetanew; 0, I_p/n \right)   d\thetanew \right|
        &= O\left( \frac{p^{2 + 2c_{\infty}}}{n^{2 - \fourthpower}} \right).   
    \end{align*}
    By Lemmas \ref{lemma:expectation } and  \ref{lemma:th2_exp}, and the Cauchy-Schwarz inequality 
    \begin{align*}
          &\left| \int_{ E_{\zerop}(\gamma_n, n^{-1/2}\Sigma^{1/2}) }\left\{ R^2_{3,n}(\thetanew) + R^2_{4,n}(\thetanew, \tilde{\theta}) \right\} \exp[ \max\{  0, \bar{R}_{3,n}(\thetanew) + \bar{R}_{4,n}(\thetanew, \tilde{\theta})\} \phi\left( \thetanew; 0, I_p/n \right)   d\thetanew \right| \\
         &\leq \left[ \int_{ E_{\zerop}(\gamma_n, n^{-1/2}\Sigma^{1/2})}\left\{ \bar{R}^2_{3,n}(\thetanew) + \bar{R}^2_{4,n}(\thetanew, \tilde{\theta})  \right\}^2 \phi\left( \thetanew; 0,  I_p/n \right)  d\thetanew\right.\\
         &\quad \left. \times \int_{ E_{\zerop}(\gamma_n, n^{-1/2}\Sigma^{1/2} )}  \exp[2 \max\{  0, \bar{R}_{3,n}(\thetanew) + \bar{R}_{4,n}(\thetanew, \tilde{\theta})\}]  \phi\left( \thetanew; 0,  I_p/n \right)  d\thetanew  \right]^{1/2} \\
         &=  O\left\{ \max\left(\frac{p^{3 + 2c_{\infty}}}{n^{3 - 2\thirdpower}} , \frac{p^{4+ 4c_{\infty}} }{n^{4 - 2\fourthpower}} \right) \right\},
    \end{align*}
    where the order of the first integral is obtained using Lemma \ref{lemma:expectation } and the inequality $(x^2+ y^2)^2 \leq 2x^4 + 2y^4 $. 
    The second integral is bounded using Lemma \ref{lemma:th2_exp}. The lower bound can be obtained by noting
    \*[
        &\exp[    \bar{R}_{3,n}(\thetanew) + \bar{R}_{4,n}(\thetanew, \tilde{\theta})  ] \nonumber\\
        &= 1 + \bar{R}_{3,n}(\thetanew) + \bar{R}_{4,n}(\theta, \tilde{\theta}) + \frac{1}{2}\{ \bar{R}_{3,n}(\thetanew) + \bar{R}_{4,n}(\thetanew, \tilde{\theta}) \}^2 \exp(R_{\exp}) \nonumber\\
        &\geq 1 + \bar{R}_{3,n}(\thetanew) + \bar{R}_{4,n}(\thetanew, \tilde{\theta}),
    \]
    the integral of $\bar{R}_{3,n}(\thetanew)$ is $0$, and the integral of $\bar{R}_{4,n}(\thetanew, \tilde{\theta})$ is $O(p^{2 + 2c_{\infty}}/n^{2 - \fourthpower})$ by the same arguments as above. 
    The integral of $1$ against the normal density of the set $E_{\zerop}(\gamma_n, n^{-1/2}\Sigma^{1/2})$ is $1 + O(n^{-\eta_1 p/4}) $ by the same arguments as used in Lemma \ref{lemma:annulus}.

    \end{proof}
    
\subsection{Proof of Theorem \ref{th:density_approx}}

\begin{customthm}{\ref{th:density_approx}}
    For a sequence $\approxpoint$ satisfying Assumptions \ref{ass:delta_decay}--\ref{ass:fourth_cum}, with Assumption \ref{ass:delta_decay} holding with $a_{n,p} = \max\left(p^{3 + 2c_{\infty}}/n^{3 - 2\thirdpower}, p^{2 + 2c_{\infty}} /n^{2 - \fourthpower} \right)$, the saddlepoint approximation (\ref{eq:saddlepoint_approx}) satisfies
    \begin{align*}
     \frac{f_{X_n}(\approxpoint)}{\hat{f}_{\rv}(\approxpoint)} =  1 + O\left\{ \max\left(\frac{p^{3 + 2c_{\infty} }}{n^{3 - 2\thirdpower}},  \frac{p^{2+ 2c_{\infty}}}{n^{2 - \fourthpower}} \right)   \right\} ,
    \end{align*}
    for $p = O( n^\alpha)$, $\alpha <  (4 - 2\fourthpower)/(5 + 4c_{\infty}) $. \label{th:density_pointwise}
\end{customthm}

\begin{proof}
    \textbf{Upper bound:}
    By Assumption \ref{ass:delta_decay}, we can account for the contribution of the integrand outside a ball of radius $\delta$ by 
    \begin{align*}
        \frac{f_{\rv}(\approxpoint)}{\hat{f}_{\rv}(\approxpoint)}  = &\frac{ \det\{U^{(x,2) }(\saddle, 0)\}^{1/2} }{(2\pi)^{p/2} } \int_{ \mathbb{R}^p}   \exp\{ K_{\rv} (\saddle , y) - K_{\rv} (\saddle, 0) - iy^\top \approxpoint  \} dy \\
        =  &\frac{ \det\{U^{(x,2) }(\saddle, 0)\}^{1/2} }{(2\pi)^{p/2} } \int_{ B_{\zerop}(\delta)}  \exp\{K_{\rv} (\saddle , y) - K_{\rv} (\saddle, 0) - iy^\top \approxpoint \} dy  \\
        &\quad + \frac{ \det\{U^{(x,2) }(\saddle, 0)\}^{1/2} }{(2\pi)^{p/2} } \int_{ B_{\zerop}^C(\delta)}  \exp\{ K_{\rv} (\saddle , y) - K_{\rv} (\saddle, 0) - iy^\top \approxpoint  \} dy\\
       = &\frac{ \det\{U^{(x,2) }(\saddle, 0)\}^{1/2} }{(2\pi)^{p/2} } \int_{ B_{\zerop}(\delta)}   \exp\{ K_{\rv} (\saddle , y) - K_{\rv} (\saddle, 0) - iy^\top \approxpoint  \} dy \\
       &\quad+ O\left\{ \max\left(\frac{p^{3+2c_{\infty}}}{n^{3 - 2\thirdpower}},\frac{p^{2 + 2c_{\infty}}}{n^{2 - \fourthpower}} \right)\right\}, 
    \end{align*}
    by Assumption \ref{ass:delta_decay}.
    Lemma \ref{lemma:annulus} shows the contribution of the integral outside of $B_{\zerop}(\gamma_n)$ is negligible. 
    Therefore, we need only show that:
    \begin{align*}
       & \left| \frac{ \det\{U^{(x,2) }(\saddle, 0)\}^{1/2} }{(2\pi)^{p/2} }  \int_{ B_{\zerop}(\gamma_n)}   \exp\{ K_{\rv} (\saddle , y) - K_{\rv} (\saddle, 0) - iy^\top \approxpoint   \} dy \right| \\
       &\leq  1 + O\left( \frac{p^{2 + 2c_{\infty}} }{n^{2 - \fourthpower}} \right) .
    \end{align*}
    By a fourth-order Taylor expansion, 
    \begin{align}
         &\frac{ \det\{U^{(x,2) }(\saddle, 0)\}^{1/2} }{(2\pi)^{p/2} } \int_{ B_{\zerop}(\gamma_n)}  \exp\{ y^\top U^{(y, \prime)}(\saddle, 0) - iy^\top \approxpoint -\frac{1}{2} y^{\top}U^{(x,2) }(\saddle, 0) y \nonumber \\
         &\qquad\qquad\qquad\qquad\qquad\qquad\quad+  R_{3,n}(y, 0, \saddle) + R_{4,n}(y, \tilde{y}, \saddle)\} dy \nonumber \\
         &= \frac{ \det\{U^{(x,2) }(\saddle, 0)\}^{1/2} }{(2\pi)^{p/2} } \int_{ B_{\zerop}(\gamma_n)}  \exp\left\{ -\frac{1}{2} y^{\top}U^{(x,2) }(\saddle, 0) y  + R_{3,n}(y, 0, \saddle) + R_{4,n}(y, \tilde{y}, \saddle) \right\} dy, \label{eq:th1_main}
    \end{align}
    where the equality follows by Lemma \ref{lemma:cauchy-riemann} through higher-order Cauchy-Riemann equations, and
    \begin{align*}
        R_{3,n}(y, 0, \saddle) &= \frac{-i}{6}\sum_{j = 1}^p y_j \left\{ y^\top U^{(x,3)}_{\cdot\cdot j}(\saddle, 0) y \right\},\\
        R_{4,n}(y, \tilde{y}, \saddle)&= \frac{1}{24}\sum_{j = 1}^p\sum_{k = 1}^p y_j y_k \left[ y^\top\left\{  U^{(x,4)}_{\cdot\cdot jk}(\saddle, \tilde{y}) + iV^{(x,4)}_{\cdot\cdot jk} (\saddle, \tilde{y})  \right\}y\right],
    \end{align*}
    for some $\tilde{y} = \tau(y) y $, where  $0 \leq \tau(y) \leq 1$. 
    Following the same steps as in the proof of Theorem \ref{thm:laplace}, we apply a change of variable $\ynew = n^{-1/2}\Sigma^{1/2}y $, where $\Sigma^{1/2}\Sigma^{1/2} = U^{(x,2) }(\saddle, 0)$. Then,
    \*[
    (\ref{eq:th1_main})
         &= \int_{ E_{\zerop}(\gamma_n, n^{-1/2}\Sigma^{1/2})}  \exp\left\{  \bar{R}_{3,n}(\ynew, 0, \saddle) + \bar{R}_{4,n}(\ynew, \tilde{y}, \saddle) \right\} \phi(\ynew; 0, I_p/n) d\ynew,
    \]
    where, 
    \begin{align*}
        \bar{R}_{3,n}(\ynew, 0, \saddle) &= \frac{-i}{6}\sum_{j = 1}^p \ynew_j \left\{ \ynew^\top A_{j} \ynew \right\},\\
        \bar{R}_{4,n}(\ynew, \tilde{y}, \saddle)&= \frac{1}{24}\sum_{j = 1}^p\sum_{k = 1}^p \ynew_j \ynew_k \left[ \ynew^\top\left\{  B_{jk}(\tilde{y}) + iC_{jk}(\tilde{y})  \right\}\ynew\right],
        \end{align*}
    for some matrices $\norm{A_{j}}_{op} = O(p^{c_{\infty}}n^{\thirdpower })$, $\norm{B_{jk}(\tilde{y})}_{op} = O(p^{2c_{\infty}}n^{\fourthpower})$ and $\norm{C_{jk}(\tilde{y})}_{op} = O(p^{2c_{\infty}}n^{\fourthpower })$ by the same argument as in Lemma \ref{lemma:spherical} and Assumptions \ref{ass:hess}--\ref{ass:fourth_cum}.
    The $R_{3,n}(\ynew, 0, \saddle)$ term can be ignored in the upper bound as
    \begin{align*}
        |\exp\{ \bar{R}_{3,n}(\ynew, 0, \saddle) \}| = \left|\exp \left[ \frac{-i}{6}\sum_{j = 1}^p \ynew_j \left\{ \ynew^\top A_{j} \ynew \right\} \right] \right| = 1,
    \end{align*}
    since the sum is real valued and $|\exp(ix)| = 1$ for $x \in R$. Similarly the imaginary part of $\bar{R}_{4,n}(\ynew, \tilde{y}, \saddle)$ can also be ignored in the upper bound.
    For the real part of $\bar{R}_{4,n}(\ynew, \tilde{y}, \saddle)$, we use a first-order Taylor series expansion of the exponential function, 
    \begin{align*}
        &\exp[  \Re\{ \bar{R}_{4,n}(\ynew, \tilde{y}, \saddle)\} ] = 1 + \Re\{ \bar{R}_{4,n}(\ynew, \tilde{y}, \saddle) \}\exp(R_{\exp})   \\
        &\leq  1 + \Re\{ \bar{R}_{4,n}(\ynew, \tilde{y}, \saddle)\} \exp( \max[  0, \Re\{ \bar{R}_{4,n}(\ynew, \tilde{y}, \saddle)\}])   ,
    \end{align*}
    where $R_{\exp}$ is a real number lying between $0$ and $ \Re\{ \bar{R}_{4,n}(\ynew, \tilde{y}, \saddle)\}$. Thus, 
    \begin{align*}
         (\ref{eq:th1_main}) 
         &= \int_{E_{\zerop}(\gamma_n, n^{-1/2}\Sigma^{1/2}) }   \exp[  \Re\{ \bar{R}_{4,n}(\ynew, \tilde{y}, \saddle) \} ] \phi\left( \ynew; 0, I_p/n \right)  d\ynew \\
         &\leq \int_{E_{\zerop}(\gamma_n, n^{-1/2}\Sigma^{1/2}) } \left\{ 1 + \Re\{ \bar{R}_{4,n}(\ynew, \tilde{y}, \saddle)\} \exp( \max[  0, \Re\{ \bar{R}_{4,n}(\ynew, \tilde{y}, \saddle) \}] ) \right\} \phi\left( \ynew; 0, I_p/n \right)  d\ynew
    \end{align*}
    Consider,
    \begin{align*}
        &\int_{ E_{\zerop}(\gamma_n, n^{-1/2}\Sigma^{1/2}) }   \Re\{ \bar{R}_{4,n}(\ynew, \tilde{y}, \saddle)\}\exp( \max[  0, \Re\{ \bar{R}_{4,n}(\ynew, \tilde{y}, \saddle) \} ] )\phi\left( \ynew; 0, I_p/n \right)    d\ynew \\
        &\leq \left[ \int_{ E_{\zerop}(\gamma_n, n^{-1/2}\Sigma^{1/2}) } \exp(2 \max[  0, \Re\{ \bar{R}_{4,n}(\ynew, \tilde{y}, \saddle) \} ] )\phi\left( \ynew; 0, I_p/n \right) d\ynew \right.
        \\
        &\quad \left. \times \int_{ E_{\zerop}(\gamma_n, n^{-1/2}\Sigma^{1/2}) }   \Re\{ \bar{R}_{4,n}(\ynew, \tilde{y}, \saddle)\}^2\phi\left( \ynew; 0, I_p/n \right) d\ynew \right]^{1/2} = O\left( \frac{p^{2 + 2c_{\infty}} }{n^{2 - \fourthpower}} \right), 
    \end{align*}
    by Lemmas \ref{lemma:th1_exp} and \ref{lemma:expectation } as
    \begin{align*}
         &  \int_{ E_{\zerop}(\gamma_n, n^{-1/2}\Sigma^{1/2} ) }    \Re\{ \bar{R}_{4,n}(\ynew, \tilde{y}, \saddle) \}^2 \phi\left[ \ynew; 0, I_p/n \right] d\ynew \\
         &\leq \int_{ \mathbb{R}^p }   \left[\sum_{j = 1}^p\sum_{k = 1}^p  \ynew_j \ynew_k \left\{ \ynew^\top B_{jk}(\tilde{y}) \ynew   \right\}   \right]^2 \phi\left( \ynew ;0, I_p/n \right)  d\ynew = O\left( \frac{p^{4 + 4c_{\infty}} }{n^{4 - 2\fourthpower}} \right).
    \end{align*}

    \textbf{Lower bound}
    The contribution outside of $B_{\zerop}(\gamma_n)$ can be ignored by the same arguments for the upper bound. Applying the same change of variable, it is sufficient to lower bound the real part of the integral  
    
    \begin{align*}
          (\ref{eq:th1_main})  &\geq \int_{ E_{\zerop}(\gamma_n, n^{-1/2}\Sigma^{1/2} ) }  \Re\left[  \exp\left\{ \bar{R}_{3,n}(\ynew, 0, \saddle) + \bar{R}_{4,n}(\ynew, \tilde{y}, \saddle)  \right\}  \right] \phi\left( \ynew; 0, I_p/n \right)    d\ynew,\\
          &=\int_{ E_{\zerop}(\gamma_n, n^{-1/2}\Sigma^{1/2} ) } \cos \left[ \Im\left\{  \bar{R}_{3,n}(\ynew, 0, \saddle) + \bar{R}_{4,n}(\ynew, \tilde{y}, \saddle) \right\}\right] \\
          &\quad\times \exp \left[ \Re \{ \bar{R}_{4,n}(\ynew, \tilde{y}, \saddle)  \} \right] \phi\left( \ynew; 0, I_p/n \right) d\ynew\\
          &\geq \int_{ E_{\zerop}(\gamma_n, n^{-1/2}\Sigma^{1/2} ) } \left[1 - \Im\left\{  \bar{R}_{3,n}(\ynew, 0, \saddle) + \bar{R}_{4,n}(\ynew, \tilde{y}, \saddle) \right\}^2 \right]\\
          &\quad \times \exp \left[ \Re \{ \bar{R}_{4,n}(\ynew, \tilde{y}, \saddle)  \} \right] \phi\left[ \ynew; 0, I_p/n \right] d\ynew \\
          &\geq \int_{ E_{\zerop}(\gamma_n, n^{-1/2}\Sigma^{1/2} ) } \left[1 - 2\Im\left\{  \bar{R}_{3,n}(\ynew,0, \saddle)\right\}^2 - 2\Im\left\{\bar{R}_{4,n}(\ynew, \tilde{y}, \saddle) \right\}^2   \right ]\\
          &\qquad \times \exp \left[ \Re \{   \bar{R}_{4,n}(\ynew, \tilde{y}, \saddle) \} \right] \phi\left( \ynew; 0, I_p /n \right) d\ynew \\
          &=  1 - O\left\{ \max\left(\frac{p^{3 + 2c_{\infty}}}{n^{3 - 2\thirdpower}},  \frac{p^{2 + 2c_{\infty}}}{n^{2 - \fourthpower}} \right)   \right\} ,
    \end{align*}
     where we have used Euler's identity, the lower bound $\cos(x) > 1- x^2$ and Young's inequality. 
     The last equality can be obtained by expanding $\exp[\Re\{\bar{R}_{4,n}(\ynew, \tilde{y}, \saddle) \}]$ as for the upper bound, and applying Lemma \ref{lemma:th1_exp} 
     and \ref{lemma:expectation }. 

    \end{proof}

\subsection{Proof of Theorem \ref{thm:ratio_exp_laplace}}
We show the proof of Theorem \ref{thm:ratio_exp_laplace} first as it captures the main ideas of the proof of the general case while and is easier to digest than the proof of the general case.

\begin{customthm}{\ref{thm:ratio_exp_laplace}}
If for $\alpha \leq 1/2 - 1/2(\zeta - 1)$, the integrals in the numerator and denominator of (\ref{eq:marg_posterior}) satisfy Assumptions \ref{ass:delta_decay_lap}--\ref{ass:exp_ratio} under the orthogonal parametrization then, 
\*[ 
\frac{ \densmarg}{\densmargapprox } =  1 + O(e_{n,p})   ,
\]
where,
\*[
e_{n,p} =  \max\left\{\frac{p^2\log(n)^2}{n}, \frac{p^{\zeta -1}\log(n)^{\zeta/2}}{n^{ (\zeta - 2)/2 }},\frac{p \log(n)^{1/2}}{n^{3/2 - \thirdpower}}   \right\},
\]
for $\{ \psi: |\psi - \maxpsi| = O(\log(n)^{1/2}/n^{1/2}) \}$ , where $\zeta$ is defined in Assumption \ref{ass:exp_ratio}, $\thirdpower, \fourthpower \leq 1 $ and Assumption \ref{ass:delta_decay_lap} holds with $e_{n,p}$ replacing $a_{n,p}$.
\end{customthm}

\begin{proof}
In the notation of Theorem \ref{thm:laplace},with $\interest$ the $p$-th component of $\theta$, 
\[\label{eqn:main_ratio}
\frac{\densmargapprox}{\densmarg} &=  \frac{ (2\pi)^{(p-1)/2} \det \{ -\loglike^{(2)}(\maximizer)\}^{1/2} }{ (2\pi)^{p/2} \det\{-\loglikeminus^{(2)}(\maximizercons)\}^{1/2}  } \frac{\int_{ \mathbb{R}^p }  \exp\{ g_n(\theta' ; X_n)  - g_n(\maximizer ; X_n)  \} d\theta'}{\int_{ R^{p-1} }  \exp\{ \loglike_n(\interest', \nuissance' ; X_n)  - \loglike_n(\maximizercons ; X_n)  \} d\nuissance' }.
\]
The proof strategy is to seek cancellation of terms in the numerator and denominator.
For the numerator, we follow the proof of Theorem \ref{thm:laplace}, although with a $\zeta$-th order Taylor expansion. It follows from Lemma \ref{lemma:annulus} 
and Assumption \ref{ass:delta_decay_lap} that the integral outside the set $[\maxpsi - \gamma_n, \maxpsi + \gamma_n ] \times B_{\hat\lambda}(\gamma_n) \supset B_{\maximizer}(\gamma_n)$ 

\begin{align}
&\frac{ \det \{- \loglike^{(2)}(\maximizer)\}^{1/2} }{ (2\pi)^{p/2}  } \int_{ \mathbb{R}^p }  \exp\{ g_n(\theta'; X_n)  - g_n(\maximizer ; X_n)  \} d\theta'\nonumber \\
&\quad =  \int_{[ - \gamma_n,  	\gamma_n] \times B_{\zeropminus}(\gamma_n) }  \exp\left\{ \sum_{j = 3}^{\zeta - 1} R^\interest_{j,n}(\theta, \maximizer) + \sum_{j = 3}^{\zeta - 1} R^{\nuissance}_{j,n}(\lambda, \maximizer) + R_{\zeta ,n}(\theta, \tilde{\theta}) \right\} \nonumber \\
&\quad\times \phi\left[ \theta; 0, \{-\loglike^{(2)}(\maximizer)\}^{-1} \right] d\theta  
+O( e_{n,p} ), \label{eq:ratio_num}
\end{align}
where we applied a change of variable $\theta = \theta'  -\maximizer$, $\tilde{\theta}$ lies on a line segment between $\theta$ and $\maximizer$ and 
\begin{align*}
     &R^{\nuissance}_{j,n}(\lambda,\theta^\star) = \frac{1}{j!}\sum_{k_1 \dots k_j = 1}^{p-1 } \nuissance_{k_1} \cdots \nuissance_{k_j} \loglike^{(j)}_{k_1 \dots k_j} (\theta^{\star}), \\
      &R^\interest_{j,n}(\theta, \theta^\star) = \frac{1}{j!} \sum_{k = 1}^{j} {j\choose k} \psi^k  \sum_{l_1 \dots l_{j - k} =1}^{p - 1} \lambda_{l_1} \cdots \lambda_{l_{j - k} } \loglike^{(j)}_{\psi \dots \psi l_1 \dots l_{j - k} } ( \theta^\star ), \\
      &R_{\zeta,n}(\theta, \theta^\star ) = \frac{1}{\zeta!} \sum_{k_1 \dots k_{\zeta} =1}^{p } \theta_{k_1} \cdots \theta_{k_{\zeta } } \loglike^{(\zeta)}_{ k_1 \dots k_{\zeta} } (\theta^\star).
\end{align*}
The terms are grouped so the parameter of interest only appears in $R_{j,n}^\psi$, the expression counts all of the terms in which $\psi$ appears at least once, 
and $R_{j,n}^\nuissance$ only contains the nuisance parameters. 
Using Lemma \ref{lemma:laplace_ratio_sup},
\begin{align}
(\ref{eq:ratio_num})  &=  
\left\{1 + O\left( \frac{p^{\zeta -1}\log(n)^{\zeta/2}}{n^{(\zeta - 2)/2}} \right) \right\} \nonumber\\
\quad& \times \int_{ [ - \gamma_n,  \gamma_n] \times B_{\zeropminus}(\gamma_n) } 
   \exp\left\{ \sum_{j = 3}^{\zeta - 1} R^\interest_{j,n}(\theta, \maximizer) + \sum_{j = 3}^{\zeta - 1} R^{\nuissance}_{j,n}(\lambda, \maximizer) \right\} \phi\left[ \theta; 0, \{-\loglike^{(2)}(\maximizer)\}^{-1} \right] d\theta \nonumber\\
  &= \left\{1 + O\left( \frac{p^{\zeta -1}\log(n)^{\zeta/2}}{n^{(\zeta - 2)/2}} \right) \right\} \\
  &\times \int_{B_{\zeropminus}(\gamma_n)}  \int_{ [ - \gamma_n,  \gamma_n]} 
   \exp \left\{\sum_{j = 3}^{\zeta - 1}R^\interest_{j,n}(\theta, \maximizer) \right\} \phi\left[ \interest; 0, \{-\loglike_{\interest\interest}^{(2)}(\maximizer)\}^{-1} \right] d\interest  \nonumber \\
   &\times  \exp\left\{  \sum_{j = 3}^{\zeta - 1} R^{\nuissance}_{j,n}(\lambda, \maximizer) \right\} \phi\left[ \nuissance; 0, \{-\loglike_{\nuissance\nuissance}^{(2)}(\maximizer)\}^{-1} \right] d\nuissance,   \nonumber \\
   &=\left( 1 + O\left[ \max\left\{ \frac{p^{\zeta -1}\log(n)^{\zeta/2}}{n^{(\zeta - 2)/2}}, \frac{p^2\log(n)^2}{n} \right\} \right]  \right)\nonumber
 \\
 &\times \int_{B_{\zeropminus}(\gamma_n)}  \exp\left\{  \sum_{j = 3}^{\zeta - 1} R^{\nuissance}_{j,n}(\lambda, \maximizer) \right\} \phi\left[ \nuissance; 0, \{-\loglike_{\nuissance\nuissance}^{(2)}(\maximizer)\}^{-1} \right] d\nuissance,   \nonumber
\end{align}
where the equality follows due to the fact that the covariance is block diagonal and by Lemma \ref{lemma:marg_prob}.
For the denominator we use a similar expansion to obtain
\[ \label{eqn:ratio_1}
&\frac{ \det\{ - \loglikeminus^{(2)}(\maximizercons)\}^{1/2}}{ (2\pi)^{(p-1)/2}  } \int_{ R^{p-1} }  \exp\{ \loglike_n(\interest, \nuissance' ; X_n)  - \loglike_n(\maximizercons ; X_n)  \} d\nuissance' \\
&=  \int_{B_{\zeropminus}(\gamma_n) }   \exp\left\{ \sum_{j = 3}^{\zeta - 1} R^{\nuissance}_{j,n}(\lambda, \maximizercons) + R^{\nuissance}_{\zeta ,n}(\lambda, \tilde{\theta}) \right\} \phi\left[ \lambda; 0, \{-\loglikeminus^{(2)}(\maximizercons)\}^{-1} \right] d\nuissance \nonumber \\
&=\left[ 1 + O\left( \frac{p^{\zeta}\log(n)^{\zeta/2}}{n^{(\zeta - 2)/2}} \right) \right] \int_{B_{\zeropminus}(\gamma_n) }   \exp\left\{ \sum_{j = 3}^{\zeta - 1} R^{\nuissance}_{j,n}(\lambda, \maximizercons)  \right\} \phi\left[ \lambda; 0, \{-\loglikeminus^{(2)}(\maximizercons)\}^{-1} \right] d\nuissance,\nonumber
\]
by Lemma \ref{lemma:laplace_ratio_sup}.
The denominator is close to the numerator, except the normal density in the integral are parametrized by different covariance matrices and $R^{\nuissance}_{3,n}$ is evaluated at $\hat\theta_\psi$ in the denominator and $\maximizer$ in the numerator. 
This suggests we should ``switch" the normal density in the denominator by considering the following Radon-Nikodym derivative and re-center the expression on the numerator at $\maximizercons$ by
\*[
&\exp\left\{ \sum_{j = 3}^{\zeta - 1} R^{\nuissance}_{j,n}(\lambda, \maximizercons)  \right\} \phi\left[ \lambda; 0, \{-\loglikeminus^{(2)}(\maximizercons)\}^{-1} \right] \\
&= \exp\left\{ \sum_{j = 3}^{\zeta - 1} R^{\nuissance}_{j,n}(\lambda, \maximizercons) - \sum_{j = 3}^{\zeta - 1} R^{\nuissance}_{j,n}(\lambda, \maximizer)  \right\} \frac{\phi\left[\lambda; 0, \{-\loglikeminus^{(2)}(\maximizercons)\}^{-1} \right]}{\phi\left[ \lambda; 0, \{-\loglikeminus^{(2)}(\maximizer)\}^{-1} \right]}\\
&\times \exp\left\{ \sum_{j = 3}^{\zeta - 1} R^{\nuissance}_{j,n}(\lambda, \maximizer) \right\}\phi\left[ \lambda; 0, \{-\loglikeminus^{(2)}(\maximizer)\}^{-1} \right].
\]
We first consider the ratio of normal densities, for $\lambda$ such that $\norm{\nuissance}_2\leq \gamma_n$:
\*[
 \Lambda(\nuissance) &= \frac{\phi\left[ \lambda; 0, \{-\loglikeminus^{(2)}(\maximizercons)\}^{-1} \right]}{\phi\left[ \lambda; 0, \{-\loglikeminus^{(2)}(\maximizer)\}^{-1} \right]} = \left[ \frac{\det\{\loglikeminus^{(2)}(\maximizer)\}}{\det\{ \loglikeminus^{(2)}(\maximizercons)\} }\right]^{1/2} \exp\left[ \frac{1}{2} \nuissance^\top\left\{\loglikeminus^{(2)}(\maximizercons) - \loglikeminus^{(2)}(\maximizer) \right\}\nuissance \right] \\
 &=  \left[\frac{\det\{\loglikeminus^{(2)}(\maximizer)\}}{\det\{\loglikeminus^{(2)}(\maximizercons)\}}\right]^{1/2} \exp\left[ \frac{1}{2} \nuissance^\top\left\{ (\psi - \hat\psi)\loglike_{\psi\lambda\lambda}^{(3)}(\hat{\theta}_{\tilde\interest}) \right\}\nuissance \right] \\
 &\leq \left[ \frac{\det\{\loglikeminus^{(2)}(\maximizer)\}}{\det\{\loglikeminus^{(2)}(\maximizercons)\}}\right]^{1/2} \exp\left[ \frac{1}{2}\gamma_n^2 \norm{ (\psi - \hat\psi)\loglike_{\psi\lambda\lambda}^{(3)}(\hat{\theta}_{\tilde\interest}) }_{op} \right] \\
 &= \left\{ 1 + O\left( \frac{p \log(n)^{1/2}}{n^{3/2 - \thirdpower}}  \right) \right\} \exp\left\{ O\left( \frac{p\log(n)}{n^{3/2 - \thirdpower}} \right) \right\} =   1 + O\left( \frac{p \log(n)^{1/2}}{n^{3/2 - \thirdpower}}  \right), 
\]
where $\tilde\psi$ lies between $\psi$ and $\hat\psi$,  by Lemma \ref{Lemma:ratio_determinants} and Assumption \ref{ass:third_lap} and the fact that $|\psi -\hat\psi| = O\{\log(n)^{1/2}n^{-1/2}\}$. The change in the evaluation point of $R^{\nuissance}_{j,n}$ contributes an error of
\*[
\left|\sum_{j = 3}^{\zeta - 1} R^{\nuissance}_{j,n}(\lambda, \maximizer) - \sum_{j = 3}^{\zeta - 1} R^{\nuissance}_{j,n}(\lambda, \maximizercons)\right| = O\left\{ \max\left( \frac{\log(n)^{2}p^2}{n^{2 - \fourthpower}}, \frac{\log(n)^{5/2}p^3}{n^{3/2}} \right) \right\},
\]
by Lemma \ref{lemma:laplace_ratio_center} for all $\nuissance \in B_{\zeropminus}(\gamma_n)$.
Using the above and combining all results on the numerator and denominator we obtain:
\*[
|(\ref{eqn:main_ratio})| &=  \frac{\left( 1 + O\left[ \max\left\{ \frac{p^{\zeta -1}\log(n)^{\zeta/2}}{n^{(\zeta - 2)/2}}, \frac{p^2\log(n)^2}{n} \right\} \right] \right) }{\left[ 1 + O\left( \frac{p^{\zeta}\log(n)^{\zeta/2}}{n^{(\zeta - 2)/2 }} \right) \right] }\\
   &\quad \times \frac{ \int_{B_{\zeropminus}(\gamma_n) } 
   \exp\left\{ \sum_{j = 3}^{\zeta - 1} R^{\nuissance}_{j,n}(\lambda, \maximizer) \right\} \phi\left[ \nuissance; 0, \{-\loglikeminus^{(2)}(\maximizer)\}^{-1} \right] d\nuissance}{\int_{B_{\zeropminus}(\gamma_n) } \Lambda(\nuissance)   \exp\left\{ \sum_{j = 3}^{\zeta - 1} R^{\nuissance}_{j,n}(\lambda, \maximizercons)  \right\} \phi\left[ \lambda; 0, \{-\loglikeminus^{(2)}(\maximizer)\}^{-1} \right] d\nuissance}\\
   &= \frac{\left( 1 + O\left[ \max\left\{ \frac{p^{\zeta -1}\log(n)^{\zeta/2}}{n^{(\zeta - 2)/2}}, \frac{p^2\log(n)^2}{n} \right\} \right] \right)}{\left[1 + O\left\{ \max\left( \frac{\log(n)^{2}p^2}{n^{2 - \fourthpower}}, \frac{\log(n)^{5/2}p^3}{n^{3/2}} \right) \right\}\right] \left\{ 1 + O\left( \frac{p \log(n)^{1/2}}{n^{3/2 - \thirdpower}}  \right)\right\}} \\
   &\quad \times \frac{\int_{B_{\zeropminus}(\gamma_n) } 
   \exp\left\{ \sum_{j = 3}^{\zeta - 1} R^{\nuissance}_{j,n}(\lambda, \maximizer) \right\} \phi\left[ \nuissance; 0, \{-\loglikeminus^{(2)}(\maximizer)\}^{-1} \right] d\nuissance}{\int_{B_{\zeropminus}(\gamma_n) }    \exp\left\{ \sum_{j = 3}^{\zeta - 1} R^{\nuissance}_{j,n}(\lambda, \maximizer)  \right\} \phi\left[ \lambda; 0, \{-\loglikeminus^{(2)}(\maximizer)\}^{-1} \right] d\nuissance}\\
   &= 1 + O\left[   \max\left\{ \frac{p\log(n)}{n^{3/2 - \thirdpower}}, \frac{p^2\log(n)^2}{n}, \frac{p^{\zeta - 1}\log(n)^{\zeta/2}}{n^{(\zeta -2)/2}} \right\}  \right], 
\]
for values of $\alpha < 1/2 - 1/(2\zeta - 2)$.
The ratio of integrals cancel as the integral is finite by Lemma \ref{lemma:ratio_cancellation}. This completes the proof.
\end{proof}

\subsection{Proof of Theorem \ref{th:general_laplace}}

\begin{customthm}{\ref{th:general_laplace}}
    If for $\alpha < 1/2 - 1/(2\zeta - 2)$ the integrals in the numerator and denominator of (\ref{eq:marg_posterior}) satisfy Assumptions \ref{ass:delta_decay_lap} -- \ref{ass:general_lap_const} under the orthogonal parametrization then 
    \*[ 
    \frac{\densmarg }{\densmargapprox } =  1 + O( e_{n,p} )   ,
    \]
    where
    \*[
    e_{n,p} 
    = \max\left\{\frac{p^2\log(n)^2}{n}, \frac{p^{\zeta -1}\log(n)^{\zeta/2}}{n^{ (\zeta - 2)/2 }},\frac{p \log(n)^{1/2}}{n^{3/2 - \thirdpower}}   \right\},
    \]
    for all $\psi \in \{ \psi: |\psi - \maxpsi| \leq O(\log(n)^{1/2}/n^{1/2}) \}$, where
     $\zeta$ is defined in Assumption \ref{ass:exp_ratio}, $\thirdpower, \fourthpower \leq 1 $ and Assumption \ref{ass:delta_decay_lap} holds with $e_{n,p}$ replacing $a_{n,p}$.
    \end{customthm}

    \begin{proof}
    The proof structure remains largely unchanged from that of Theorem \ref{thm:ratio_exp_laplace}, however the order of some of the terms considered in the proof are different, since the dependence between the constrained mode and $\psi$ is stronger than in the case of the linear exponential family.
    There are also some additional difficulties encountered due to $\loglike^{(2)}_{\interest\nuissance}(\hat\theta_\psi) \neq 0$.
    We highlight the steps where additional considerations are needed.
    
    The first change in the proof is in (\ref{eq:ratio_num}), as the information matrix isn't necessarily block diagonal. We instead split the normal density into a product of the conditional density of $\psi|\lambda$ and the marginal density of $\lambda$
    \*[
    &\phi[ \theta; 0, \{-\loglike^{(2)}(\maximizer)\}] \\
    &= \phi\left[\psi; -\loglike^{(2)}_{\psi\lambda}(\maximizer) \{\loglike^{(2)}_{\psi\psi}(\maximizer)\}^{-1}\lambda, \{-\loglike^{(2)}_{\interest\interest}(\maximizer)\}^{-1} \right]\\
    &\quad\times \phi\left(\nuissance; 0, [-\loglike^{(2)}_{\nuissance\nuissance}(\maximizer) + \loglike^{(2)}_{\lambda\psi}(\maximizer)\loglike^{(2)}_{\psi\psi}(\maximizer)^{-1} \loglike^{(2)}_{\psi\lambda}(\maximizer) ]^{-1}  \right),
    \]
    by using the block inversion formula and standard properties of the multivariate normal \citep[Chapter 2.3]{bishop}. 
    The integral with respect to the conditional density 
    \*[ 
    &\int_{ [ - \gamma_n,  \gamma_n]} 
       \exp \left\{\sum_{j = 3}^{\zeta - 1}R^\interest_{j,n}(\theta, \maximizer) \right\} \phi\left[ \interest; -\loglike^{(2)}_{\psi\lambda}(\maximizer) \{\loglike^{(2)}_{\psi\psi}(\maximizer)\}^{-1}\lambda, \{-\loglike_{\interest\interest}^{(2)}(\maximizer)\}^{-1} \right] d\interest \\
       &= 1 + O\left\{ \frac{p^2\log(n)^{2}}{n} \right\},
    \]
    for $\nuissance \in  B_{\zeropminus}(\gamma_n)$ by Lemma \ref{lemma:cond_prob_ration}. 
    Similarly, the marginal density of $\nuissance$ takes on a different form from that found in the denominator, we account for this by considering
    \*[
    \frac{\phi\left(\nuissance;  0 , [-\loglike^{(2)}_{\nuissance\nuissance}(\maximizer) + \loglike^{(2)}_{\lambda\psi}(\maximizer)\loglike^{(2)}_{\psi\psi}(\maximizer)^{-1} \loglike^{(2)}_{\psi\lambda}(\maximizer) ]^{-1}  \right)}{\phi\left(\nuissance; 0, [-\loglike^{(2)}_{\nuissance\nuissance}(\maximizer) ]^{-1}  \right)} = 1 + O\left\{ \frac{p^2\log(n)}{n} \right\}, 
    \]
    for values of $\lambda \in B_{\zeropminus}(\gamma_n)$ by Lemma \ref{lemma:radon_n_gen}. 
    
    Next we show that for $\norm{\lambda}_2 \leq \gamma_n$ 
    \[\label{eq:gen_rnderiv}
     \Lambda(\nuissance) &= \frac{\phi\left[ \lambda; 0, \{-\loglikeminus^{(2)}(\maximizercons)\}^{-1} \right]}{\phi\left[ \lambda; 0, \{-\loglikeminus^{(2)}(\maximizer)\}^{-1} \right]} = \left[\frac{\det\{\loglikeminus^{(2)}(\maximizer)\}}{\det\{\loglikeminus^{(2)}(\maximizercons)\}}\right]^{1/2} \exp\left[ -\frac{1}{2} \nuissance^\top\left\{\loglikeminus^{(2)}(\maximizercons) - \loglikeminus^{(2)}(\maximizer) \right\}\nuissance \right]\\
     &= 1 + O\left\{ \frac{\log(n)^{1/2}p}{n^{3/2 -\thirdpower}} \right\}, \nonumber 
     \]
    and 
    \[\label{eq:higher_diffs_gen}
    &\left|\sum_{j = 3}^{\zeta - 1} R^{\nuissance}_{j,n}(\lambda, \maximizer) - \sum_{j = 3}^{\zeta - 1} R^{\nuissance}_{j,n}(\lambda, \maximizercons)\right| = O\left\{ \max\left( \frac{\log(n)^{2}p^2}{n^{2 - \fourthpower}}, \frac{\log(n)^{5/2}p^3}{n^{3/2}} \right) \right\},
    \]
    holds.
    We then plug in these rates into the proof of Theorem \ref{thm:ratio_exp_laplace} to obtain the stated result. 
    We first bound (\ref{eq:gen_rnderiv}). Following the steps in the proof of Lemma \ref{Lemma:ratio_determinants}, 
    
    \[
    \det\{-\loglikeminus^{(2)}(\maximizer)\} &= \det\left[ -\loglikeminus^{(2)}(\maximizercons) - (\psi - \hat\psi)  \left\{ \loglike_{\psi\lambda\lambda}^{(3)}(\hat\theta_{\tilde{\psi}})|_{\psi = \tilde\psi} + \sum_{j = 1}^{p-1} \frac{\partial \hat\nuissance_j}{\partial\interest}(\tilde\psi) \loglike^{(3)}_{\lambda_j\lambda\lambda}(\hat\theta_{\tilde{\psi}}) \right\} \right] \label{eq:det_general_1} \\
    &= \det\{-\loglikeminus^{(2)}(\maximizercons) \} \nonumber\\
    &\quad \times \det\left[ I  + (\maxpsi - \psi) \{ -\loglikeminus^{(2)}(\maximizercons) \}^{-1} \left\{ \loglike_{\psi\lambda\lambda}^{(3)}(\hat\theta_{\tilde{\psi}})|_{\psi = \tilde\psi} + \sum_{j = 1}^{p-1} \frac{\partial \hat\nuissance_j}{\partial\interest}(\tilde\psi) \loglike^{(3)}_{\lambda_j \lambda\lambda}(\hat\theta_{\tilde{\psi}}) \right\} \right]\nonumber \\
    &=: \det \{-\loglikeminus^{(2)}(\maximizercons) \} \det ( I  + A ), \nonumber 
    \]
    for some value of $\tilde\psi$ between $\psi$ and $\hat\psi$. The maximal singular value of A is
    \[
    \norm{A}_{op} &= \norm{(\maxpsi - \psi) \{- \loglikeminus^{(2)}(\maximizercons) \}^{-1} \left\{ \loglike^{(3)}_{\psi\lambda\lambda}(\hat\theta_{\tilde{\psi}})|_{\psi = \tilde\psi} + \sum_{j = 1}^{p-1} \frac{\partial \hat\nuissance_j}{\partial\interest}(\tilde\psi) \loglike_{\lambda_j\lambda\lambda}^{(3)}(\hat\theta_{\tilde{\psi}}) \right\}}_{op} \label{eq:det_general_2}\\
    &\leq \norm{(\maxpsi - \psi) \{ -\loglikeminus^{(2)}(\maximizercons) \}^{-1}}_{op} \left\{ \norm{\loglike^{(3)}_{\psi\lambda\lambda}(\hat\theta_{\tilde{\psi}})|_{\psi = \tilde\psi}}_{op} + \sum_{i = 1}^{p-1}\norm{\frac{\partial \hat\nuissance_j}{\partial\interest}(\tilde\psi) \loglike^{(3)}_{\lambda_j \lambda\lambda}(\hat\theta_{\tilde{\psi}})}_{op} \right\}\nonumber \\
    &\leq O\left\{ \frac{\log(n)^{1/2}}{n^{3/2}} \right\} \left\{ O(n^{\thirdpower}) + O(n^{\thirdpower}) \norm{\frac{\partial \hat\nuissance_j}{\partial\interest}(\tilde\psi)}_1 \right\} \leq O\left\{ \frac{\log(n)^{1/2}}{n^{3/2 - \thirdpower}} \right\}   \left(1 + p^{1/2} \norm{\frac{\partial \hat\nuissance_j}{\partial\interest}(\tilde\psi)}_2 \right)\nonumber\\
    &= O \left\{ \frac{\log(n)^{1/2}}{n^{3/2 - \thirdpower}} \right\} \left\{1 +  O \left( \frac{p}{n^{1/2}}\right) \right\} = O \left\{ \frac{\log(n)^{1/2}}{n^{3/2 - \thirdpower}} \right\} \nonumber,
    \]
    by Assumptions \ref{ass:exp_ratio} and \ref{ass:general_lap_const}, Lemma \ref{lemma:nuissance_size} and finally the fact that we restrict $\alpha < 1/2 - 1/(2\zeta - 2)$. Next by following the argument outlined in \ref{Lemma:ratio_determinants} from (\ref{eq:det_general_1}) and (\ref{eq:det_general_2}), the above implies
    \[\label{eq:general_lap_exp}
    \left\{\frac{|\loglikeminus^{(2)}(\maximizer)|}{|\loglikeminus^{(2)}(\maximizercons)|}\right\}^{1/2} = 1 +O\left\{ \frac{\log(n)^{1/2} p}{n^{3/2-\thirdpower}} \right\}.
    \]
    We also have,
    \*[
    &\exp\left[ -\frac{1}{2} \nuissance^\top\left\{\loglikeminus^{(2)}(\maximizercons) - \loglikeminus^{(2)}(\maximizer) \right\}\nuissance \right] = 1 + O\left\{ \frac{\log(n)^{1/2}p}{n^{3/2 -\thirdpower}} \right\},
    \] 
    as $\norm{\lambda}_2 \leq \gamma_n$ and
    \*[
    &\norm{\loglikeminus^{(2)}(\maximizercons) - \loglikeminus^{(2)}(\maximizer)}_{op} = \norm{(\psi - \hat\psi)  \left\{ \loglike^{(3)}_{\psi\lambda\lambda}(\hat\theta_{\tilde{\psi}})|_{\psi = \tilde\psi} + \sum_{j = 1}^{p-1} \frac{\partial \hat\nuissance_j}{\partial\interest}(\tilde\psi) g^{(3)}_{\nuissance_j\lambda\lambda}(\hat\theta_{\tilde{\psi}}) \right\}}_{op} \\
    &= O\left( \frac{\log(n)^{1/2} p}{n^{3/2- \thirdpower} } \right),
    \]
    by the same calculation as performed above, (\ref{eq:general_lap_exp}) is then obtained by applying Rayleigh's quotient.
    Finally it remains to show (\ref{eq:higher_diffs_gen}) holds. First consider $3<j \leq \zeta - 1 $, then
    \*[
    &R^{\nuissance}_{j,n}(\lambda,\maximizer) = \frac{1}{j!}\sum_{k_1 \dots k_j = 1}^{p-1 } \nuissance_{k_1} \cdots \nuissance_{k_j} \loglike^{(j)}_{k_1 \dots k_j} (\maximizer)\\
    &=\frac{1}{j!}\sum_{k_1 \dots k_j =1 }^{p-1 } \nuissance_{k_1} \cdots \nuissance_{k_j} \loglike^{(j)}_{k_1 \dots k_j} (\maximizercons) \\
    &+  \frac{(\psi  -\hat\psi)}{j!}\sum_{k_1 \dots k_j =1}^{p-1 } \nuissance_{k_1} \cdots \nuissance_{k_j} \left\{ \loglike^{(j + 1)}_{\psi k_1 \dots k_j} (\hat{\theta}_{\tilde{\psi}}) + \sum_{l = 1}^{p-1} \frac{\partial \hat\nuissance_l}{\partial\interest}(\tilde\psi) g^{(j +1)}_{\lambda_l k_1 \dots k_j}(\hat\theta_{\tilde{\psi}})\right\} \\
    &= R^{\nuissance}_{j,n}(\lambda,\maximizercons) +  \frac{(\psi  -\hat\psi)}{j!}\sum_{k_1 \dots k_j = 1}^{p-1 } \nuissance_{k_1} \cdots \nuissance_{k_j} \left\{ \loglike^{(j + 1)}_{\psi k_1 \dots k_j} (\hat{\theta}_{\tilde{\psi}}) + \sum_{l = 1}^{p-1} \frac{\partial \hat\nuissance_l}{\partial\interest}(\tilde\psi) g^{(j +1)}_{\lambda_l k_1 \dots k_j}(\hat\theta_{\tilde{\psi}})\right\},
    \]
    therefore, 
    \*[
    &\left|R^{\nuissance}_{j,n}(\lambda,\maximizer) - R^{\nuissance}_{j,n}(\lambda,\maximizercons)\right| \\
    &= \left|\frac{(\psi  -\hat\psi)}{j!}\sum_{k_1 \dots k_j = 1}^{p-1 } \nuissance_{k_1} \cdots \nuissance_{k_j} \left\{ \loglike^{(j + 1)}_{\psi k_1 \dots k_j} (\hat{\theta}_{\tilde{\psi}}) + \sum_{l = 1}^{p-1} \frac{\partial \hat\nuissance_l}{\partial\interest}(\tilde\psi) g^{(j +1)}_{\lambda_l k_1 \dots k_j}(\hat\theta_{\tilde{\psi}})\right\} \right|\\
    &\leq O\left( \frac{\log(n)^{1/2}}{n^{1/2}} \right)\left| \sum_{k_1 \dots k_{j- 2}=1 }^{p-1 } \nuissance_{k_1} \cdots \nuissance_{k_{j- 2}} \left[ \lambda^{\top} \left\{ \loglike^{(j + 1)}_{\cdot \cdot \psi k_1 \dots k_{j-2}} (\hat{\theta}_{\tilde{\psi}}) + \sum_{l = 1}^{p-1} \frac{\partial \hat\nuissance_{l}}{\partial\interest}(\tilde\psi) g^{(j +1)}_{\cdot \cdot \lambda_l k_1 \dots k_{j - 2}}(\hat\theta_{\tilde{\psi}})\right\} \lambda\right] \right|.
    \]
    The maximum singular value of
    \*[
    \norm{\loglike^{(j + 1)}_{\cdot \cdot \psi k_1 \dots k_{j - 2}} (\hat{\theta}_{\tilde{\psi}}) + \sum_{l = 1}^{p-1} \frac{\partial \hat\nuissance_{l}}{\partial\interest}(\tilde\psi) g^{(j +1)}_{\cdot \cdot \lambda_l k_1 \dots k_{j - 2}}(\hat\theta_{\tilde{\psi}})}_{op} = O(n),
    \]
    by the same argument as used in (\ref{eq:det_general_2}) and Assumption \ref{ass:exp_ratio}, implying
    \*[
    \left|R^{\nuissance}_{j,n}(\lambda,\maximizer) - R^{\nuissance}_{j,n}(\lambda,\maximizercons)\right| = O\left( \frac{p^{j - 1}\log(n)^{(j-1)/2} }{n^{(j- 1)/2}} \right), 
    \]
    by the same calculation as Lemma \ref{lemma:laplace_ratio_center}, as for the case that j = 3, 
    \*[
    \left|R^{\nuissance}_{3,n}(\lambda,\maximizer) - R^{\nuissance}_{3,n}(\lambda,\maximizercons)\right| = O\left( \frac{\log(n)^2p^2}{n^{2 - \fourthpower}} \right), 
    \]
    by the same arguments, except we use Assumption \ref{ass:fourth_lap}. This concludes the proof.
    
    \end{proof}

\section{Proof and further details of examples}

\subsection{Proof of Corollary \ref{cor:logistic_laplace}}

\begin{proof}
    This proof uses the mle as the centering point instead of the posterior mode.
    This does not change the structure of the proof of Theorem \ref{thm:laplace}, but requires some slight modifications. 
    Denote the prior density for $\beta$ by $\pi(\beta)$ and the log-likelihood by $l_n(\beta)$.

    Lemma \ref{lemma:corr_ass} shows the mass outside of $B_{\hat\beta_{mle}}(\gamma_n \log(n))$ is negligible, therefore Assumption \ref{ass:delta_decay_lap} is satisfied for a smaller radius. 
    Assumption \ref{ass:hess_lap} now holds for $\beta \in B_{\hat\beta_{mle}}(\gamma_n\log(n)) $ by the same Lemma, and we can modify the proof of \ref{lemma:annulus} to show that the posterior mass in $B^C_{\hat\beta_{mle}}(\gamma_n) \cap B_{\hat\beta_{mle}}(\gamma_n \log(n))$ is $O(n^{-\eta_1 p/8})$.  
    Thus we only need to show, 
    \[
        \frac{\det\{ -l_n^{(2)}(\hat\beta_{mle}) \}^{1/2}}{(2\pi)^{p/2}} \int_{B_{\hat{\beta}_{mle}}(\gamma_n)  } \frac{\pi(\beta)}{\pi(\hat\beta_{mle})} \exp\{ l_n(\beta) - 
    l_n(\hat{\beta}_{mle})\} d\beta 
    = 1 + O\left( \frac{p^2\log(n)}{n} \right). \nonumber
    \] 
    We begin with,
    \*[
    \frac{\pi(\beta)}{\pi(\hat\beta_{mle})} &= \exp(-\beta^\top \beta/2 + \hat\beta_{mle}^\top\hat\beta_{mle}/2 ) \\
    &= \exp[ O(\gamma_n^2) + O\{ \gamma_n^2\log(n) \} ]\\
    &= 1 + O\left\{ \frac{p\log(n)^2}{n} \right\}, 
    \]
    as \cite{non-uniform} show that $\lVert \hat\beta_{mle}\rVert_{\infty} \leq \log(n)/n^{1/2}$ with probability tending to 1.
    Following this step we use the same expansions as in the proof of Theorem \ref{thm:laplace}, and need only calculate the order of the third and fourth derivatives. 
    
    We use the notation $\text{diag} \left( a_k \right)_{k = 1, \dots, n}$ to denote a square diagonal matrix of dimension $n$ with diagonal entries $a_k$, $k = 1, \dots, n$. For the third likelihood derivative, by a first order Taylor expansion,
    \*[
    &l^{(3)}_{\cdot\cdot j}(\maximizer) = X^{\top}\left[ \text{diag} \left\{ x_{kj} p^{(2)} (x_k^\top \hat\beta) \right\}_{k = 1, \dots, n}     \right] X \\
    &=  X^{\top}\left[ \text{diag} \left\{ x_{kj} p^{(2)} (0) +  x_{kj} p^{(3)} ( r_k )x_k^\top \hat\beta \right\}_{k = 1, \dots, n}     \right] X\\
    &=  X^{\top}\left[ \text{diag} \left\{  x_{kj} p^{(3)} ( r_k )(x_k^\top \hat\beta) \right\}_{k = 1, \dots, n}     \right] X,
    \]
    where $p^{(j)}$ is the $j^{th}$ derivative of the probability of success  in (\ref{eq:logistic_eq}), $p^{(2)} (0) = 0$ and $r_k$ lies between $0$ and $x_k^{\top} \hat\beta$. 
    Now, 
    \*[
    \max_{j} \norm{ l^{(3)}_{\cdot\cdot j}(\maximizer) }_{op} &= \max_{j} \norm{X^{\top}\left[ \text{diag} \left\{  x_{kj} p^{(3)} ( r_k )(x_k^\top \hat\beta) \right\}_{k = 1, \dots, n}     \right] X}_{op} \\
    &\leq \norm{ X^\top X}_{op} \max_{j = 1, \dots, p} \max_{k = 1, \dots, n} |x_{kj} p^{(3)} ( r_k )(x_k^\top \hat\beta)| = O[  \{\log(n)np\}^{1/2} ], 
    \]
    by Lemma \ref{lemma:gaussian_rm}, $\max_{k = 1, \dots, n} |x_k^\top \hat\beta| = O \{( p/n )^{1/2} \}$, boundedness of $p^{(3)}(\cdot)$ and $\lVert X^\top X \rVert_{op} = O(n) $ with probability tending to 1 (by Theorem 4.6.1 in \cite{vershynin2018high}). 
    Thus we have shown that $\thirdpower = (1 + \alpha)/2 + \log\{\log(n)\}/2\log(n)$, as defined in Assumption \ref{ass:third_lap}. As for the fourth derivative, for all $\beta \in B_{\hat\beta_{mle}}(\gamma_n)$
    \*[
    \max_{j,k =1 , \dots, p } \norm{ \loglike^{(4)}_{\cdot\cdot j k}(\maximizer) }_{op} &= \max_{j, k} \norm{X^{\top}\left[ \text{diag} \left\{  x_{mj} x_{mk} p^{(3)} (x_k^\top \beta) \right\}_{m = 1, \dots, n} \right] X}_{op} \\
    &\leq \norm{ X^\top X}_{op} \max_{j,k =1 , \dots, p } \max_{m = 1, \dots, n} |x_{mj}| |x_{mk}| |p^{(3)} ( x_m^\top \beta)| = O\{ \log(n) n \},
    \]
    by Lemma \ref{lemma:gaussian_rm} as $p^{(3)}(\cdot)$ is a bounded function, meaning that Assumption \ref{ass:fourth_lap} is satisfied with $\fourthpower = 1 + \log\{\log(n)\}/\log(n)$. Therefore, following the same computation as in the proof of Theorem \ref{thm:laplace}, and using the fact that by Lemma \ref{lemma:cor_infinity} $c_{\infty} = 0$, we have
    \*[
     \frac{f(\beta|X, Y)}{\hat{f}(\beta|X, Y) }  = 1 + O \left( \frac{p^2\log(n)}{n}  \right),
    \]
    for $\alpha < 2/5$.
    \end{proof}

\subsection{Proof of Corollary \ref{cor:logistic_glm}}

\begin{proof}
    We use the mle as the centering point for our expansions. 
    Denote the prior density for $\beta$ by $\pi(\beta)$ and the log-likelihood by $l_n(\beta)$.
    Assumptions \ref{ass:delta_decay_lap} and \ref{ass:hess_lap} are satisfied for the different centering point and radius $\gamma_n\log(n)$ by the same arguments as in Corollary \ref{cor:logistic_laplace}.  
    Thus we wish to show, 
    \[
        \frac{\det\{ -l_n^{(2)}(\hat\beta_{mle}) \}^{1/2}}{(2\pi)^{p/2}} \int_{B_{\hat{\beta}_{mle}}(\gamma_n)  } \frac{\pi(\beta)}{\pi(\hat\beta_{mle})} \exp\{ l_n(\beta) - 
    l_n(\hat{\beta}_{mle})\} d\beta 
    = 1 + O\left( \frac{p^3\log(n)}{n} \right). \nonumber
    \] 
    We begin with,
    \*[
    \frac{\pi(\beta)}{\pi(\hat\beta_{mle})} &= \exp(-\beta^\top \beta/2 + \hat\beta_{mle}^\top\hat\beta_{mle}/2 ) \\
    &= \exp\{-(\beta - \beta_0 + \beta_0)^\top (\beta - \beta_0 + \beta_0)/2 + (\hat\beta_{mle}  - \beta_0 + \beta_0)^\top (\hat\beta_{mle} - \beta_0 + \beta_0)/2 \} \\
    &=  exp\{-(\beta - \beta_0 )^\top (\beta - \beta_0 )/2 + (\hat\beta_{mle}  - \beta_0 )^\top (\hat\beta_{mle} - \beta_0)/2 - \beta_0^\top(\beta - \beta_0) + \beta_0^\top(\beta - \beta_0) \}\\
    &= \exp[ O(\gamma_n^2) + O\{ \gamma_n^2\log(n) \} +O( \gamma_n )+ O\{ \gamma_n\log(n) \} ]\\
    &= 1 + O\left\{ \frac{p\log(n)^2}{n} \right\}, 
    \]
    as \cite{non-uniform} show that $\lVert \hat\beta_{mle}- \beta_0\rVert_{\infty} \leq \log(n)/n^{1/2}$ with probability tending to 1.
    We then use the same expansions as in the proof of Theorem \ref{thm:laplace}, and only need to calculate the order of the third and fourth derivatives. 
    
    We use the notation $\text{diag} \left( a_k \right)_{k = 1, \dots, n}$ to denote a square diagonal matrix of dimension $n$ with diagonal entries $a_k$, $k = 1, \dots, n$. For the third likelihood derivative, 
    \*[
    \max_{j} \norm{ l^{(3)}_{\cdot\cdot j}(\maximizer) }_{op} &= \max_{j} \norm{X^{\top}\left[ \text{diag} \left\{  x_{kj} K^{(2)}(x_k^\top \hat\beta) \right\}_{k = 1, \dots, n}     \right] X}_{op} \\
    &\leq \norm{ X^\top X}_{op} \max_{j = 1, \dots, p} \max_{k = 1, \dots, n} |x_{kj} K^{(2)} (x_k^\top \hat\beta)| = O\{ \log(n)^{1/2}n \}, 
    \]
    by Lemma \ref{lemma:gaussian_rm}, where $K^{(2)}(x_k^\top \hat\beta)$ is the skewness of an observation for $\eta_k = x_k^\top \hat\beta$ which is bounded with probability tending to 1 by the first statement of Lemma \ref{glm:cor} as the cumulant generating is an infinitely smooth differentiable function. 
    Thus we have shown that $\thirdpower = 1 + \log\{\log(n)\}/2\log(n)$, as defined in Assumption \ref{ass:third_lap}. As for the fourth derivative the same argument shows that $\fourthpower = 1 + \log\{\log(n)\}/\log(n)$. By Lemma \ref{glm:cor} $c_{\infty} = 0$, therefore
    \*[
     \frac{f(\beta|X, Y)}{\hat{f}(\beta|X, Y) }  = 1 + O \left( \frac{p^3\log(n)}{n}  \right),
    \]
    for $\alpha < 1/3$.
    \end{proof}

\subsection{Proof of Example \ref{ex:exp_means}}
Note that although the saddlepoint approximation on the numerator and denominator depends on the value of $\lambda$, the ratio of the approximation does not. 
We derive the saddlepoint approximation under the null $H_0: \psi_1 = \dots = \psi_{g-1} = 0$. This derivation is valid for all $u$ such that $\max_{j} |\hat\eta_j/\lambda| \leq B$ for some $B> 0$. Recall that the likelihood under the $(\interest, \nuissance)$ parameterization is: 
\*[
        l(\psi, \lambda; y) &=  -\sum_{j = 1}^{g - 1}  \sum_{k = 1}^{j} u_k \psi_j - \sum_{j =1}^g u_i \lambda  + \sum_{j =1}^{g-1} \log\left\{ \lambda + \sum_{k =1}^j \psi_k \right\}  +\log(\lambda)\\
        &= \sum_{j = 1}^{g - 1} S_j \psi_j - S_g \lambda  + \sum_{j =1}^{g-1} \log\left\{ \lambda + \sum_{k =1}^j \psi_k \right\}  +\log(\lambda),
\]
and we wish to estimate the density of $S_1, \dots, S_g$ at some point $\approxpoint$ and $u_k = \sum_{j = 1}^m y_{kj}$. 
The dependence structure of the sufficient statistics can be decomposed as:
\*[
    f_S(s_1, \dots, s_g) &= f_{g| 1, \cdots, g-1}(s_g | s_{g - 1}, \dots, s_{1} ) f_{g-1| 1, \cdots, g-2}(s_{g-1} | s_{g - 2}, \dots, s_{1} ) \cdots f_{1}(s_1) \\
    &= f_{g| g-1}(s_g | s_{g - 1}) f_{g-1| g-2 }(s_{g-1} | s_{g - 2}) \cdots f_{1}(s_1), 
\]
as the $u_k$'s are independent sums of iid exponential distributions, hence the joint distribution factors into conditionally independent sums of exponential distributions. Thus,
\*[
    \hat{f}_S(s_1, \dots, s_g) &= \hat{f}_{g| g-1}(s_g | s_{g - 1}) \hat{f}_{g-1| g-2 }(s_{g-1} | s_{g - 2}) \cdots \hat{f}_{1}(s_1), 
\]
where approximate the joint density with the products of univariate density approximations. Under the null, the density approximations is uniformly accurate, thus:
\*[
    f_S(s_1, \dots, s_g; \psi, \lambda) &= \hat{f}_{g| g-1}(s_g | s_{g - 1}) \hat{f}_{g-1| g-2 }(s_{g-1} | s_{g - 2}) \cdots \hat{f}_{1}(s_1)\left\{1 + O\left( \frac{1}{m} \right)\right\}^g, \\
    &= \hat{f}_{g| g-1}(s_g | s_{g - 1}) \hat{f}_{g-1| g-2 }(s_{g-1} | s_{g - 2}) \cdots \hat{f}_{1}(s_1)\left\{1 + O\left( \frac{g}{m} \right)\right\}
\]
where the last equality is valid if $g/m \rightarrow 0$. But the denominator under the null is a gamma distribution meaning that:
\*[
    f(s_g; \lambda) = \hat{f}(s_g; \lambda)\left\{1 + O\left( \frac{1}{mg} \right) \right\}. 
\]
However the saddlepoint approximation is exact for the gamma distribution up to a normalizing constant so
\*[
    f_S(s_1, \dots| s_g; \psi, \lambda) = C \frac{\hat{f}_{g| g-1}(s_g | s_{g - 1}) \hat{f}_{g-1| g-2 }(s_{g-1} | s_{g - 2}) \cdots \hat{f}_{1}(s_1)}{\hat{f}(s_g; \lambda)},
\]
and the direction test as defined in \cite{vector_linear} only requires the integration of the ratio of conditional densities, thus the direction test is exact for $m > 1$.

\subsection{Exponential Regression}

We use the saddlepoint approximation to approximate the joint density of the sufficient statistics associated with the regression coefficients in an exponential regression model with the canonical inverse link function. 
Let $Y_j$, $j = 1, \dots, n$, be independent observations from an exponential distribution with rate parameter $\lambda_j = X_j^{\top}\beta_0$, where $X_j^{\top}$ is the $j$-th row of the design matrix $X$, whose dimensions are $n \times p$ whose rows are independently generated from an isotropic Gaussian distribution with covariance $\sigma_0 I$ such that $\sigma_0 > 0$, and $\beta_0$ is the data generating vector parameter of length $p$. For this proof we take $\sigma_0 = 1$ but the same arguments holds for any $\sigma_0 > 0$.

Recall that we assume we are estimating the density for a value of $S = s$ such that $A_1 < |X_j^\top \beta| < A_2$ for all $j = 1, \dots, n$. The log-likelihood for this model is:
\*[
    l(\beta; X, Y) = \sum_{j = 1}^n \log(X_j^\top \beta) - X_j^
    \top\beta y_j = \sum_{j = 1}^n \log(X_j^\top \beta) - \sum_{k = 1}^p \sum_{j = 1}^n x_{jk} y_j \beta_j,
\] 
and the vector of sufficient statistics is $S = ( -\sum^n_{j = 1}x_{j1}y_j, \dots, -\sum^n_{j = 1}x_{jp}y_j)$ under the $\lambda$ parameterization. 
Let $t = a + ib $ where $a, b \in \Reals^p$, then the cumulant generating function for the random variable $S$ evaluate at $t$ is
\begin{align*}
    K_{S}(t) &= K_{S}(a,b) = \sum_{j = 1}^n \log\left( \frac{1}{ 1 +  X_j^\top t/\lambda_j } \right) \\
    &= -\frac{1}{2} \sum_{j =1}^{n} \log\left\{ \left( 1 + X_j^\top a/\lambda_j \right)^2 + \left( X_j^\top b/\lambda_j \right)^2 \right\} \\
    &+ i \sum_{j = 1}^{n} \arcsin\left( X_j^\top b/\lambda_j  \right),
\end{align*}
where we have split the cumulant generating function into its real and imaginary parts.
Under the canonical link function, the saddlepoint for approximating the density at $S = s$ is $\hat\beta(s) - \beta_0$, where $\hat\beta(s)$ is the maximum likelihood estimate of $\beta$ for a value of the sufficient statistic $s$. 
We replace $\delta$ in Assumption \ref{ass:delta_decay} and \ref{ass:hess} with $\delta_n$ such that $\phi_n^4 p\delta_n^2 \rightarrow \infty$ with $\phi_n \rightarrow 0$, to be defined later, and $\delta_n \rightarrow 0$. It is sufficient to only consider the real portion of the integral as the imaginary part is necessarily $0$.
We first show that Assumption \ref{ass:hess} holds with $\delta_n$, note that
\*[
    U^{(x,2) }(\saddle, 0) = X^T D(\saddle, 0) X, 
\]
where $D$ is a diagonal matrix whose entires are variances of $Y_j$ with rate parameter $\lambda_j =X^\top_j \hat\beta(s)$.
Thus,
\*[
    K_{S}(\saddle,b) - K_{S}(\saddle,0) = \frac{1}{2}\sum_{j =1}^n \log\left\{ \frac{(X_j^\top \hat\beta(s)/\lambda_j)^2}{(X_j^\top \hat\beta(s)/\lambda_j)^2 + (X_j^\top b/\lambda_j)^2 } \right\},
\]
and
\[
    &\left|\int_{B^C_{\zerop}(\delta_n)} \exp\left\{ K_{S}(\saddle,b) - K_{S}(\saddle,0)  \right\} db \right| \nonumber \\
    &\leq \int_{B^C_{\zerop}(\delta_n)} \left| \exp\left\{ K_{S}(\saddle,b) - K_{S}(\saddle,0) - ib^\top s_n \right\} \right| db \nonumber \\
    &=  \int_{B^C_{\zerop}(\delta_n)} \exp\left\{ K_{S}(\saddle,b) - K_{S}(\saddle,0) \right\}  db \nonumber \\
    &= \int_{B^C_{\zerop}(\delta_n)} \prod_{j = 1}^n \left\{\frac{(X_j^\top \hat\beta(s)/\lambda_j)^2}{(X_j^\top \hat\beta(s)/\lambda_j)^2 + (X_j^\top b/\lambda_j)^2 } \right\}^{1/2} db, \label{eq:exp_regression_1}
\]
by our assumption that $ 0 < A_1 < \lambda_j < A_2$ and for our value of $s_n$, $0 < B_1 < X_j^\top\beta(s) < B_2$, it follows that $C_1 < X_j^\top\beta(s_n)/\lambda_i \leq C_2$ therefore, we can upper bound this integral by
\[
(\ref{eq:exp_regression_1}) &\leq \int_{B^C_{\zerop}(\delta_n)} \prod_{j = 1}^n \left\{\frac{(X_j^\top \hat\beta(s)/\lambda_j)^2}{(X_j^\top \hat\beta(s)/\lambda_j)^2 + (X_j^\top b/\lambda_j)^2 } \right\}^{1/2} db \nonumber \\
&\leq \int_{B^C_{\zerop}(\delta_n)} \prod_{j = 1}^n \left\{\frac{C_2^2}{C_1^2 + \cos^2\{\theta_j(b)\}\norm{X_j}_2^2 \norm{b}_2^2/A_2^2 } \right\}^{1/2} db, \label{eq:exp_regression_2}
\]
where $\theta_j(b)$ is the angle between $X_j$ and $b$.

Define the set $R_{j, n}: = \{ b: \text{angle between } X_j \text{ and } b \text{ is } \in (\pi/2 - \phi_n,\pi/2 + \phi_n ) \}$.
We now split the regions of integration into the following pieces: $ B^C_{\zerop}(\delta_n) \cap R_{1, n}, \dots, B^C_{\zerop}(\delta_n) \cap R_{n, n}$ and $B^C_{\zerop}(\delta_n) \cap (\bigcup_{j = 1}^n R_{j, n})^c$. For each of the contributions on $ B^C_{\zerop}(\delta_n) \cap R_{j, n}$, consider:  
\*[
&\int_{B^C_{\zerop}(\delta_n) \cap R_{1, n}} \prod_{j = 1}^n \left\{\frac{C_2^2}{C_1^2 + \cos^2\{\theta_j(b)\}\norm{X_j}_2^2 \norm{b}_2^2/A_2^2 } \right\}^{1/2} db \\
&\leq\int_{B^C_{\zerop}(\delta_n) \cap R_{j, n} } \left\{\frac{C_2^2}{C_1^2 + \cos^2\{\theta_1(b)\}\norm{X_1}_2^2 \norm{b}_2^2/A_2^2 } \right\}^{1/2}\prod_{j = 2}^n \left\{\frac{C_2^2}{C_1^2 + \phi_n^4\norm{X_j}_2^2 \norm{b}_2^2/A_2^2 } \right\}^{1/2} db \\
&\leq \frac{C_2}{C_1} \int_{B^C_{\zerop}(\delta_n) \cap R_{j, n} } \prod_{j = 2}^n \left\{\frac{C_2^2}{C_1^2 + \phi_n^4\norm{X_j}_2^2 \norm{b}_2^2/A_2^2 } \right\}^{1/2} db,
\]
where we have used the identity that $|\cos(\pi/2 - \phi)| \geq \phi^2$ for $\phi \in (-1/2, 1/2)$.
Changing the region of integration to $B^C_{\zerop}(\delta_n)$ and then further performing a change of variable to the hyper-spherical coordinate system, while noting that with probability tending to $1$, $\norm{X_j}_2^2 > Dp$ for any $0 <D < 1$ uniformly by Lemma \ref{lemma:chisq-lower} : 
\*[
&\frac{C_2}{C_1} \int_{B^C_{\zerop}(\delta_n) \cap R_{1, n}} \prod_{j = 2}^n \left\{\frac{C_2^2}{C_1^2 + \phi_n^4 \norm{X_i}_2^2 \norm{b}_2^2/A_2^2 } \right\}^{1/2} db \\
&\leq\frac{C_2}{C_1} \int_{B^C_{\zerop}(\delta_n)} \left\{\frac{C_2^2}{C_1^2 + D\norm{X_j}_2^2 \norm{b}_2^2/A_2^2 } \right\}^{(n-1)/2} db \\
&= \frac{C_2}{C_1} \int_{0}^{\pi} \cdots \int_{0}^{\pi} \int_{0}^{2\pi} \int_{\delta_n}^{\infty} \left\{\frac{C_2^2}{C_1^2 + Dp r^2/A_2^2 } \right\}^{(n - 1)/2} r dr \\
&\times \sin(\theta_2) \sin^2(\theta_3) \dots \sin^{p-2}(\theta_{p-1}) dr d\theta_1 d\theta_2 \cdots d\theta_{p - 1}\nonumber \\
&\leq \frac{2 A_2^2}{C_1 Dp} \frac{2 \pi^{p}}{\Gamma(p/2)} \int_{C_1^2 + Dp\delta_n^2/A_2^2 }^{\infty} \frac{C_2^n }{ z^{(n-1)/2}}   dz \nonumber\\
&= \frac{2 A_2^2}{C_1 Dp} \frac{2 \pi^{p}}{\Gamma(p/2)} \frac{C_2^n}{(C_1^2 + Dp\delta_n^2/A_2^2)^{(n - 3)/2}} \frac{2}{n - 3} = O\left\{\exp(-n)\right\}, \nonumber
\]
for any choice of $\phi_n = o(1)$, and
note the integral involving the $\sin$ terms is the surface area of an unit $n$-sphere. We performed a change of variable $z = Dpr^2/A_2^2 + C_1^2 $, and the final line follows from the restriction that: $\phi_n^4 p\delta_n^2 \rightarrow \infty$. All other contributions to the integral, including the contribution from $B_{\zerop}^C(\delta_n)$, can be upper bounded in the same manner, thus, summing over all of the contributions:
\*[
(\ref{eq:exp_regression_1}) \leq (n+1)O\{\exp(-n)\} = O\{\exp(-n)\},    
\]
verifying Assumption \ref{ass:delta_decay}.
\*[
    U^{(x,2) }(\saddle, b) = X^T D(\saddle, b) X^T,     
\]
where $D(\saddle, b)$ is a diagonal matrix with entries
\*[
    [D(\saddle, b)]_{ii} =  \frac{1}{\lambda_i^2}\frac{\{ X_j^\top\hat\beta(s) \}^2 - (X_j^\top b)^2 }{\{X_j^\top\hat\beta(s) \}^2 + (X_j^\top b)^2} \geq \frac{C_1^2}{2C_2^2}
\]
with probability tending to $1$, as we have assumed that the covariates are distributed according to a Gaussian distribution, therefore $X_j^\top b \sim N(0, \norm{b}_2^2)$ and $\norm{b}_2^2 = \delta_n^2 \rightarrow 0$, and $\delta_n$ can be chosen such that: $\max_{j = 1, \dots, n}(X_j^\top b)^2 \xrightarrow{p} 0$. Furthermore as we also have: 
\*[
    \norm{X^\top X/n - \Sigma_0}_\infty \leq p^{1/2} \norm{X^\top X/n - \Sigma_0}_{op} \leq O\left( \frac{p}{n^{1/2}} \right)
\]
and the entries of $D(\saddle, 0)$ are upper bounded, thus $\ppower = 0$ in Assumption \ref{ass:hess} with probability tending to 1 if $p = o(n^{1/2})$. 

The arguments involving the derivatives of the real valued functions for Assumptions \ref{ass:third_cum} and \ref{ass:fourth_cum} are similar to what we have above and in Example \ref{example:logis_laplace}, and $\thirdpower = \fourthpower = 1 +\log(\log(n))$. Finally for the derivative of the imaginary component in Assumption \ref{ass:fourth_cum} we have:
\*[
    V^{(x,4)}_{\cdot \cdot lm}(\saddle , b) = X_{j, l} X_{j, m} X^\top D^\prime(\saddle , b) X,   
\]
where the matrix $D^\prime(\saddle , b))$ is diagonal with entries:
\*[
    [D^\prime(\saddle , b)]_{jj}  = \frac{1}{\lambda_j^4} \frac{6 (X_j^\top b/\lambda_j)^3 +  9X_j^\top b/\lambda_j}{ \{1 - (X_j^\top b /\lambda_j)^2\}^{7/2} }, 
\]
once again noting that $(X_j^\top b)^2 \rightarrow 0$ uniformly with probability tending to $1$ for the range of $\gamma_n$ being considered, which implies the entries of this diagonal matrix are uniformly bounded.  This then further imply that the eigenvalues of $ V^{(x,4)}_{\cdot \cdot lm}(\saddle , b)$ can be uniformly bounded with $\fourthpower = 1 + \log\log n $ by examining the maximum of Gaussian random variables using similar arguments as in Example \ref{example:logis_laplace}

\begin{lemma}\label{lemma:sphere_cap}
    The ratio between the area of the unit $p$-sphere cap obtained by the intersection a double sided cone centered at the origin generated by an angle of $\phi$ with the surface of the unit $p$-sphere is: 
    \*[  \frac{A_n}{S_n} = O\left(\frac{\phi_n^{p - 1}}{p^{3/2}}\right), \]
    in the joint limit as both $p$ and $n$ tend to infinity, with $p = n^\alpha$ for $\alpha < 1$ and $\phi_n=o(1)$. 
\end{lemma}

\begin{proof}
    The ratio of the surface area of a spherical cap to the on the $p$-sphere of radius $1$ implied by the intersection with bi-directional cone of angle $\phi_n$ is given by: 
    \begin{align*}
        \frac{A_n}{S_n} &= \frac{\Gamma(\frac{p - 1}{2})\Gamma(1/2)}{\Gamma(p/2)} \int_{0}^{\sin^2(\phi_n)} \frac{t^{\frac{p - 3}{2}}}{(1 - t)^{1/2}} dt
        \leq \sqrt{2} \frac{\Gamma(\frac{p - 1}{2})\Gamma(1/2)}{\Gamma(p/2)}\int_{0}^{\phi_n^2} t^{\frac{p - 3}{2}} dt \\
        &\leq \sqrt{2} \frac{\Gamma(\frac{p - 1}{2})\Gamma(1/2)}{\Gamma(p/2)} \frac{\phi_n^{p - 1}}{\frac{p - 1}{2}} = O\left(\frac{\phi_n^{p - 1}}{p^{3/2}}\right),
    \end{align*}
    where we have used $(1 - t)^{-1/2} < \sqrt{2} $ for $t < 1/2$ and the last line follows by Stirling's approximation.
\end{proof}

\begin{lemma}
    Let $X_1, X_2, \dots, X_n$ be isotropic $p$ dimensional standard multivariate Gaussian vectors. The probability that no angle between any $X_j, X_k$ is not in
    $(\frac{\pi}{2} - \phi_n, \frac{\pi}{2} - \phi_n)$ is given by $\rightarrow 1$ if $\phi_n = o(1)$.
\end{lemma}

\begin{proof}
    Let $E_j$ be the event that angle between $X_j, X_k$ is not in $(\frac{\pi}{2} - \phi_n, \frac{\pi}{2} - \phi_n)$ for all $k = 1, \dots, j - 1, j+1, \dots, n$.
    The distribution of the Gaussian vectors $X_j/\norm{X_j}_2$ is uniform on the sphere, thus the probability of the vector $X_j$ from being not included in the set of angles is the ratio of the area of the spherical cap implied by the double sided cone with an angle of $\phi_n$. Thus from Lemma $\ref{lemma:sphere_cap}$ and the fact that the vectors are independent:
    \*[
       \mathbb{P}(E_j) &= \left\{ 1- O\left(\frac{\phi_n^{p - 1}}{p^{3/2}}\right)\right\}^n = \left\{ 1- O\left(\frac{n\phi_n^{p - 1}}{np^{3/2}}\right)\right\}^n,
    \] 
    for all $j = 1, \dots, n$ by symmetry. We now show that the complement of the event $\cap_{j = 1}^n E_j$ has probability $0$ asymptotically consider:
\*[
    P(\cup E_j^c ) &\leq n P(E_1^c) = n \left[1 - \left\{1 -O\left(\frac{\phi_n^{p - 1}}{p^{3/2}}\right) \right\}^n \right] 
    \leq O\left(\frac{n^2\phi_n^{p - 1}}{p^{3/2}}\right)\rightarrow 0
\]
as $ (1 + x/n)^n \geq 1 + x $, and this tends to $0$ for $\phi_n = o(n^{-1/3p}) = o(1)$, as $n^{-1/3p} \rightarrow 1$. 
\end{proof}

\begin{lemma}\label{lemma:chisq-lower}
For $p = n^{\alpha}$ for $\alpha < 1$ and for iid copies of a $\chi^2_p$ random variable: $\chi^2_{p, 1}, \dots, \chi^2_{p, n} $
\*[
P(\chi^2_{p, j} > Dp, \text{ for all } j = 1, \dots, n ) \rightarrow 1, 
\]
for any $0 < D < 1$ in the limit as both $n$ and $p$ tend to $\infty$.
\end{lemma}

\begin{proof}
    By Lemma 1 in \cite{massart-chi}, the following inequality holds 
    \*[
    P(\chi_{p, j}^2 \geq p - 2\sqrt{pt} ) \geq 1- \exp(-t).    
    \]
    Therefore:
    \*[
        &P\left[\chi_{p, j}^2 \geq p\left(1 - 2\sqrt{\frac{2\log(n)}{p}} \right), \text{ for all } j = 1, \dots, n \right] \\
        &\geq [1- \exp\{-2\log(n)\}]^n = \left( 1 - \frac{1}{n^2} \right)^n \rightarrow 1.
    \]
    Finally noting that the strictly increasing sequence $(1 - 2\sqrt{2\log(n)/p} ) \rightarrow 1$ gives us the desired result. 
\end{proof}

\section{Proof of lemmas used in Theorem \ref{thm:laplace} }
This lemma is also used in the proof of Theorem \ref{th:density_pointwise}.

\begin{lemma} \label{lemma:annulus}

Under Assumption \ref{ass:hess_lap}, $\gamma^2_n = \log(n)p/n$, $p =O(n^\alpha)$ for $ \alpha < 1$, we have
\*[
\frac{ \det\{-g_n^{(2)}(\maximizer) \}^{1/2}}{ (2\pi)^{p/2}  } \int_{ B^C_{\maximizer}(\gamma_n) \cap B_{\maximizer}(\delta) }  \exp\{ g_n(\theta; X_n)  - g_n(\maximizer ; X_n)  \} d\theta = O(n^{-\eta_1 p/4}),
\]
while under Assumption \ref{ass:hess},
\begin{align*}
     &\frac{ \det \{U^{(x,2) }(\saddle, 0)\}^{1/2} }{(2\pi)^{p/2}  }  \left\lvert \int_{B^C_{\zerop}(\gamma_n) \cap B_{\zerop}(\delta)}  \exp\{ K_{\rv} (\saddle, y)  - K_{\rv} (\saddle, 0) -iy^\top \approxpoint  \} dy \right\rvert \\
     &\quad = O(n^{-\eta_1 p/4}).
\end{align*}
\end{lemma}

\begin{proof}
Let $A = B_{\zerop}(\delta)$, and $D = B_{\zerop}(\gamma_n)$
\[
&\frac{ \det\{-g_n^{(2)}(\maximizer) \}^{1/2}}{ (2\pi)^{p/2}  } \int_{ B^C_{\maximizer}(\gamma_n) \cap B_{\maximizer}(\delta) }  \exp\{ g_n(\theta' ; X_n)  - g_n(\maximizer ; X_n)  \} d\theta' \nonumber \\
&\leq \frac{ \det\{-g_n^{(2)}(\maximizer) \}^{1/2}}{ (2\pi)^{p/2}  } \int_{ D^C \cap A }  \exp\left\{-\frac{1}{2}\theta^\top \loglike^{(2)}(\tilde{\theta}) \theta   \right\} d\theta \label{eq:lemma1_lp},
\]
by a change of variable $\theta = \theta' - \maximizer$ and where $\tilde{\theta} = \tau(\theta)\theta + \{1 - \tau(\theta) \} \maximizer$, for $0 \leq \tau(\theta) \leq 1$. By Assumption \ref{ass:hess_lap}, 
\*[
(\ref{eq:lemma1_lp})  &\leq  \frac{ \det\{-g_n^{(2)}(\maximizer) \}^{1/2}}{ (2\pi)^{p/2}  }   \int_{A \cap D^C}  \exp\left(-\frac{\eta_1 n}{2} \theta^\top \theta   \right)  d\theta \\
&\leq \left(\frac{\eta_2}{\eta_1}\right)^{p/2} \int_{B_{\zerop}(\gamma_n)^C}  \phi(\theta ; 0, \eta_1 I_p/n)  d\theta \\
&= \left(\frac{\eta_2}{\eta_1}\right)^{p/2} \mathbb{P} \left[\chi^2_p \geq n \eta_1 \gamma_n^2  \right] = \left(\frac{\eta_2}{\eta_1}\right)^{p/2}  P\left[\chi^2_p/p \geq 1+ \zeta_n \right],
\]
where $\zeta_n = n\gamma^2_n \eta_1/p - 1$, and the region of integration was changed to a larger one by using $D^C$ instead of $A \cap D^C$. By Lemma 3 in \cite{Fan},
\begin{align*}
    P\left[\chi^2_p/p \geq 1+ \zeta_n \right] \leq \exp\left[ \frac{p}{2} \{ \log(1 +\zeta_n) - \zeta_n \}  \right],
\end{align*}
and $n\gamma_n^2\eta_1/p = \eta_1 \log(n) \rightarrow \infty$, 
so there exists $N_0$ such that $ \log(1 +\zeta_n) - \zeta_n  \leq -\eta_1 \log(n)/2 $ for all $n > N_0$ which implies
\begin{align*}
    \left(\frac{\eta_2}{\eta_1}\right)^{p/2}  P\left[\chi^2_p/p \geq 1+ \zeta_n \right] \leq \left(\frac{\eta_2}{\eta_1}\right)^{p/2} \exp\{ -\eta_1p\log(n)/2 \} =O (n^{-\eta_1 p/4}),
\end{align*}
as eventually $p\log(\eta_2/\eta_1)/2 - \eta_1p\log(n)/2 \leq - \eta_1p\log(n)/4 $.

As for the second statement,
using a second-order Taylor series expansion for both the real and imaginary part of the integrand, 
\begin{align*}
    &= \frac{ \det \{U^{(x,2) }(\saddle, 0)\}^{1/2} }{(2\pi)^{p/2} } \left\lvert \int_{A \cap D^C}  \exp\left\{-\frac{1}{2} y^\top \left\{ U^{(2, x)}(\saddle, \tilde{y} ) +i V^{(2, x)}(\saddle, \tilde{y} )  \right\} y  \right\} dy \right\rvert,
\end{align*}
where $\tilde{y} = \tau(y) y $ for some $0\leq \tau(y) \leq 1$. The imaginary component will not contribute to the modulus when upper bounding the integral as its modulus is exactly 1,
\begin{align}
    &\frac{ \det \{U^{(x,2) }(\saddle, 0)\}^{1/2} }{(2\pi)^{p/2} } \left\lvert \int_{A \cap D^C}  \exp\left[-\frac{1}{2} y^\top U^{(2, x)}(\saddle, \tilde{y} ) y  \right] \exp\left[-\frac{i}{2} y^\top V^{(2, x)}(\saddle, \tilde{y} ) y \right] dy \right\rvert\nonumber \\
    &\leq \frac{ \det \{U^{(x,2) }(\saddle, 0)\}^{1/2} }{(2\pi)^{p/2} }  \int_{A \cap D^C}  \exp\left[-\frac{1}{2} y^\top U^{(2, x)}(\saddle, \tilde{y} ) y   \right]  dy \label{eq:lemma1}.
\end{align}
By Assumption \ref{ass:hess}, 
\begin{align*}
  (\ref{eq:lemma1}) &\leq  \frac{ \det \{U^{(x,2) }(\saddle, 0)\}^{1/2} }{(2\pi)^{p/2} }  \int_{A \cap D^C}  \exp\left(-\frac{\eta_1 n}{2} y^\top y   \right)  dy \leq \left(\frac{\eta_2}{\eta_1}\right)^{p/2} \int_{B_{\zerop}(\gamma_n)^C}  \phi(y; 0, \eta_1 I_p/n)  dy \\
  &= \left(\frac{\eta_2}{\eta_1}\right)^{p/2} \mathbb{P} \left[\chi^2_p \geq n \eta_1 \gamma_n^2  \right] = \left(\frac{\eta_2}{\eta_1}\right)^{p/2}  P\left[\chi^2_p/p \geq 1+ \zeta_n \right],
\end{align*}
where $\zeta_n = n\gamma^2_n \eta_1/p - 1$, and the region of integration was changed to a larger one by using $D^C$ instead of $A \cap D^C$. By Lemma 3 in \cite{Fan},
\begin{align*}
    P\left[\chi^2_p/p \geq 1+ \zeta_n \right] \leq \exp\left[ \frac{p}{2} \{ \log(1 +\zeta_n) - \zeta_n \}  \right],
\end{align*}
and $n\gamma_n^2\eta_1/p = \eta_1 \log(n) \rightarrow \infty$, 
so there exists $N_0$ such that $ \log(1 +\zeta_n) - \zeta_n  \leq -\eta_1 \log(n)/2 $ for all $n > N_0$ which implies
\begin{align*}
    \left(\frac{\eta_2}{\eta_1}\right)^{p/2}  P\left[\chi^2_p/p \geq 1+ \zeta_n \right] \leq \left(\frac{\eta_2}{\eta_1}\right)^{p/2} \exp\{ -\eta_1p\log(n)/2 \} =O (n^{-\eta_1 p/4}),
\end{align*}
as eventually $p\log(\eta_2/\eta_1)/2 - \eta_1p\log(n)/2 \leq - \eta_1p\log(n)/4 $, showing the desired result.
\end{proof}

\begin{lemma} \label{lemma:annulus_mod}
    Assume that the eigenvalues of $-g_n^{(2)}(\maximizer)$ has $q$ eigenvalues of order $n$ and $p - q$ eigenvalues of order $m$, and $m/n = p$. Let $\gamma^2_n = \log(m)p/m$, $p =O(n^\alpha)$ for $ \alpha < 1$, we have
    \*[
    \frac{ \det\{-g_n^{(2)}(\maximizer) \}^{1/2}}{ (2\pi)^{p/2}  } \int_{ B^C_{\maximizer}(\gamma_n) \cap B_{\maximizer}(\delta) }  \exp\{ g_n(\theta; X_n)  - g_n(\maximizer ; X_n)  \} d\theta = O(m^{-\eta_1 p/4}).
    \]
\end{lemma}

\begin{proof}
The proof is largely the same as Lemma \ref{lemma:annulus}, with slight modifications. Let $A = B_{\zerop}(\delta)$, and $D = B_{\zerop}(\gamma_n)$
\[
&\frac{ \det\{-g_n^{(2)}(\maximizer) \}^{1/2}}{ (2\pi)^{p/2}  } \int_{ B^C_{\maximizer}(\gamma_n) \cap B_{\maximizer}(\delta) }  \exp\{ g_n(\theta' ; X_n)  - g_n(\maximizer ; X_n)  \} d\theta' \nonumber \\
&\leq \frac{ \det\{-g_n^{(2)}(\maximizer) \}^{1/2}}{ (2\pi)^{p/2}  } \int_{ D^C \cap A }  \exp\left\{-\frac{1}{2}\theta^\top \loglike^{(2)}(\tilde{\theta}) \theta   \right\} d\theta \label{eq:lemma1_lp},
\]
by a change of variable $\theta = \theta' - \maximizer$ and where $\tilde{\theta} = \tau(\theta)\theta + \{1 - \tau(\theta) \} \maximizer$, for $0 \leq \tau(\theta) \leq 1$. By Assumption \ref{ass:hess_lap}, 
\*[
(\ref{eq:lemma1_lp})  &\leq  \frac{ \det\{-g_n^{(2)}(\maximizer) \}^{1/2}}{ (2\pi)^{p/2}  }   \int_{A \cap D^C}  \exp\left(-\frac{\eta_1 m}{2} \theta^\top \theta   \right)  d\theta \\
&\leq \left(\frac{n\eta_2^\prime}{m\eta_1^\prime}\right)^{q/2} \left(\frac{\eta_2}{\eta_1}\right)^{(p - q)/2} \int_{B_{\zerop}(\gamma_n)^C}  \phi(\theta ; 0, \eta_1 I_p/m)  d\theta \\
&= \left(\frac{n\eta_2^\prime}{m\eta_1^\prime}\right)^{q/2} \left(\frac{\eta_2}{\eta_1}\right)^{(p - q)/2} \mathbb{P} \left[\chi^2_p \geq m \eta_1 \gamma_n^2  \right] \\
&= \left(\frac{n\eta_2^\prime}{m\eta_1^\prime}\right)^{q/2} \left(\frac{\eta_2}{\eta_1}\right)^{(p - q)/2}  P\left[\chi^2_p/p \geq 1+ \zeta_n \right],
\]
where $\zeta_n = m\gamma^2_n \eta_1/p - 1$, and the region of integration was changed to a larger one by using $D^C$ instead of $A \cap D^C$. By Lemma 3 in \cite{Fan},
\begin{align*}
    P\left[\chi^2_p/p \geq 1+ \zeta_n \right] \leq \exp\left[ \frac{p}{2} \{ \log(1 +\zeta_n) - \zeta_n \}  \right],
\end{align*}
and $m\gamma_n^2\eta_1/p = \eta_1 \log(m) \rightarrow \infty$, 
so there exists $N_0$ such that $ \log(1 +\zeta_n) - \zeta_n  \leq -\eta_1 \log(m)/2 $ for all $n > N_0$ which implies
\begin{align*}
    &\left(\frac{n\eta_2^\prime}{m\eta_1^\prime}\right)^{q/2}\left(\frac{\eta_2}{\eta_1}\right)^{(p - q)/2}  P\left[\chi^2_p/p \geq 1+ \zeta_n \right] \\
    &\leq \left(\frac{n\eta_2^\prime}{m\eta_1^\prime}\right)^{q/2} \left(\frac{\eta_2}{\eta_1}\right)^{(p - q) /2} \exp\{ -\eta_1p\log(m)/2 \} =O (n^{-\eta_1 p/4}),
\end{align*}
as eventually $(p - q)\log(\eta_2/\eta_1)/2 + q\log(n/m)/2 - \eta_1p\log(m)/2 \leq - \eta_1p\log(m)/4 $.
\end{proof}

\begin{lemma}
In the notation of Theorem \ref{thm:laplace} and under Assumption \ref{ass:hess_lap}--\ref{ass:fourth_lap}, for the change of variable $\thetanew = n^{-1/2}\Sigma^{1/2}\theta $\label{lemma:spherical}
\*[
R_{3,n}(\theta, \maximizer) = \bar{R}_{3,n}(\thetanew), \quad R_{4,n}(\theta, \tilde\theta) = \bar{R}_{4,n}(\thetanew, \tilde\theta), 
\]
where
\*[
 \bar{R}_{3,n}(\thetanew) &= \frac{1}{6}\sum_{j = 1}^p \thetanew_j \left\{ \thetanew^\top A_j \thetanew \right\}, \quad \bar{R}_{4,n}(\thetanew, \tilde{\theta}) = \frac{1}{24}\sum_{j = 1}^p\sum_{k = 1}^p \thetanew_j \thetanew_k \left\{ \thetanew^\top B_{jk}(\tilde{\theta}) \thetanew \right\}, 
\]
for matrices $A_j$ and $B_{jk}(\tilde\theta)$ that satisfies 
\*[
 \norm{A_j}_{op} = O(p^{c_\infty}n^{\thirdpower}), \quad
\norm{B_{jk}(\tilde\theta)}_{op} = O(p^{2c_\infty}n^{\fourthpower }),
\]  for all $j, k = 1 \dots , p$ and for all $\thetanew \in E_{\zerop}(\gamma_n, n^{-1/2}\Sigma^{1/2})$, where $\tilde\theta = \tau(\theta)\theta + \{ 1 - \tau(\theta)\} \maximizer$, for $0 \leq \tau(\theta) \leq 1$.
\end{lemma}

\begin{proof}
Recall,
\*[
R_{3,n}(\theta, \maximizer) = \frac{1}{6}\sum_{j = 1}^p \theta_j \left\{ \theta^\top \loglike^{(3)}_{\cdot\cdot j}(\maximizer) \theta \right\}, \quad R_{4,n}(\theta, \tilde{\theta})= \frac{1}{24}\sum_{j = 1}^p\sum_{k = 1}^p \theta_j \theta_k \left\{ \theta^\top \loglike^{(4)}_{\cdot\cdot jk}(\tilde{\theta}) \theta \right\},
\]
and $\theta = n^{1/2} \Sigma^{-1/2} \thetanew$.
First consider $R_{3,n}(\theta, \maximizer)$
\*[ 
\frac{1}{6}\sum_{j = 1}^p \theta_j \left\{\theta^\top \loglike_{\cdot\cdot j}^{(3)}(\maximizer) \theta \right\} &= \frac{1}{6} n^{3/2} \sum_{j =1}^p \sum_{k = 1}^p \Sigma^{-1/2}_{j,k} \thetanew_k \left\{ \thetanew^\top \Sigma^{-1/2} \loglike_{\cdot\cdot j}^{(3)}(\maximizer) \Sigma^{-1/2} \thetanew \right\}\\
&= \frac{1}{6} \sum_{k = 1}^p \thetanew_k   \left[ \thetanew^\top \left\{ n^{3/2}\sum_{j =1}^p  \Sigma^{-1/2}_{j,k} \Sigma^{-1/2} \loglike_{\cdot\cdot j}^{(3)}(\maximizer) \Sigma^{-1/2} \right\} \thetanew \right],
\]
by changing the order of summation. Therefore,
\*[
A_j = 
n^{3/2}\sum_{k =1}^p  \Sigma^{-1/2}_{k,j}  \Sigma^{-1/2} \loglike_{\cdot\cdot k}^{(3)}(\maximizer) \Sigma^{-1/2}, 
\]
and its maximal singular value,
\*[
\norm{A_j}_{op} &= 
n^{3/2} \norm{\sum_{k =1}^p  \Sigma^{-1/2}_{k,j} \Sigma^{-1/2}  \loglike_{\cdot\cdot k}^{(3)}(\maximizer) \Sigma^{-1/2}}_{op}\\
&\leq \max_{k = 1, \dots, p} \norm{n \Sigma^{-1/2} \loglike_{\cdot\cdot k}^{(3)}(\maximizer) \Sigma^{-1/2} }_{op} \norm{n^{1/2}\Sigma^{-1/2}}_\infty = O(p^{c_{\infty}}n^{\thirdpower }), 
\]
by Assumptions \ref{ass:hess_lap}--\ref{ass:third_lap}, showing the first statement.
As for $R_{4,n}(\theta, \tilde{\theta})$,
\*[
R_{4,n}(\theta, \tilde{\theta}) &= \frac{n^2}{24}\sum_{j = 1}^p\sum_{k = 1}^p \theta_j \theta_k \left\{ \theta^\top \loglike^{(4)}_{\cdot\cdot jk}(\tilde{\theta}) \theta \right\}\\
&= \frac{n^2}{24}\sum_{j = 1}^p\sum_{k = 1}^p \sum_{l = 1}^p \Sigma^{-1/2}_{j,l} \thetanew_l  \sum_{m = 1}^p \Sigma^{-1/2}_{k,m}\thetanew_m \left\{ \thetanew^\top \Sigma^{-1/2} \loglike^{(4)}_{\cdot\cdot jk}(\tilde{\theta}) \Sigma^{-1/2} \thetanew \right\}\\
&= \frac{1}{24}\sum_{l = 1}^p\sum_{m = 1}^p  \thetanew_l\thetanew_m  \left[ \thetanew^\top\left\{n^2 \sum_{j = 1}^p \Sigma^{-1/2}_{j,l}   \sum_{k = 1}^p \Sigma^{-1/2}_{k,m} \left( \Sigma^{-1/2} \loglike^{(4)}_{\cdot\cdot jk}(\tilde{\theta}) \Sigma^{-1/2}\right) \right\}\thetanew \right], 
\]
thus, 
\*[
B_{jk}(\tilde\theta) =  n^2 \sum_{l = 1}^p \Sigma^{-1/2}_{l,j}   \sum_{m = 1}^p \Sigma^{-1/2}_{m,k} \left(\Sigma^{-1/2} \loglike^{(4)}_{\cdot\cdot jk}(\tilde{\theta}) \Sigma^{-1/2} \right) , 
\]
and $\norm{B_{jk}(\tilde\theta)}_{op} = O(p^{2c_{\infty}}n^{\fourthpower  })$ by the same argument as made for $R_{3,n}(\theta, \tilde\theta)$ using Assumptions \ref{ass:hess_lap} and \ref{ass:fourth_lap}.

\end{proof}

\begin{lemma}\label{lemma:expectation }
For any $p \times p$ matrices $A_{j}$ and $B_{jk}$, such that for all $j,k = 1, \cdots, p$, 
\begin{align*}
    \eta_3 p^{c_\infty} n^{\thirdpower } \leq \lambda_p(A_j)  \leq \lambda_1(A_j) \leq   \eta_4 p^{c_\infty}n^{\thirdpower} \\
    \eta_5 p^{2c_{\infty}} n^{\fourthpower} \leq \lambda_p(B_{jk})  \leq \lambda_1(B_{jk}) \leq   \eta_6 p^{2c_{\infty}}n^{\fourthpower},
\end{align*}
for constants $\eta_3, \eta_4, \eta_5, \eta_6 \in \mathbb{R}$ which are independent of $n$ and $p$, we have 
\begin{align*}
    \int_{\mathbb{R}^p} \sum_{j,k = 1}^p \theta_j \theta_k \{\theta^\top B_{jk} \theta\} \phi(\theta; 0, I_{p}/n) d\theta = O\left( \frac{p^{2 +2 c_{\infty}}}{n^{2 - \fourthpower}}\right),\\
    \int_{\mathbb{R}^p} \left[ \sum_{j,k = 1}^p \theta_j \theta_k \{\theta^\top B_{jk} \theta\} \right]^2 \phi(\theta; 0, I_{p}/n) d\theta = O\left( \frac{p^{4 +4 c_{\infty}}}{n^{4 - 2\fourthpower}}\right),\\
    \int_{\mathbb{R}^p} \left[\sum_{j,k = 1}^p \theta_j \theta_k \{\theta^\top B_{jk} \theta\} \right]^4 \phi(\theta; 0, I_{p}/n) d\theta = O\left( \frac{p^{8 + 8c_{\infty}}}{n^{8 - 4\fourthpower}} \right),\\
    \int_{\mathbb{R}^p} \left[\sum_{j,k = 1}^p \theta_j  \{\theta^\top A_{j} \theta\} \right]^4 \phi(\theta; 0, I_{p}/n) d\theta = O\left(\frac{p^{6 + 4c_{\infty}} }{n^{6 - 4\thirdpower}} \right).
\end{align*}

\end{lemma}

\begin{proof}
The maximal singular value bounds the magnitude of the entries of a matrix, so the elements of $A_j = O(p^{c_\infty}n^{\thirdpower})$ and $B_{jk} = O(p^{2c_\infty}n^{\fourthpower})$ uniformly for all $j,k = 1, \dots,p$. The calculation for the order of these quantities are quite similar, so we only perform the calculation for the first statement. Let $B_{jklm} = [B_{jk}]_{lm}$, 
\*[
&\int_{\mathbb{R}^p} \sum_{j,k =1}^p \theta_j \theta_k \{\theta^\top B_{jk} \theta\} \phi(\theta; 0, I_p/n) d\theta = \int_{\mathbb{R}^p} \sum_{j,k, l, m =1}^p \theta_j \theta_k \theta_l \theta_m  B_{jklm} \phi(\theta; 0, I_p/n) d\theta \\
 &=\int_{\mathbb{R}^p} \sum_{j,k = 1 }^p \theta_j^2 \theta_k^2  B_{jjkk} \phi(\theta; 0, I_p/n) d\theta + \int_{\mathbb{R}^p} \sum_{j=1 }^p \theta_j^4  B_{jjjj} \phi(\theta; 0, I_p/n) d\theta \\
 =& O\left( \frac{p^{2 + 2c_{\infty}} n^{\fourthpower }}{n^2}  \right) + O\left( \frac{p^{1 + 2c_{\infty}} n^{\fourthpower  }}{n^2}  \right) =  O\left( \frac{p^{2+ 2c_{\infty}} }{n^{2- \fourthpower}}  \right).
\]
Since the covariance matrix is diagonal only the expectation of indices which are repeated an even number of times will be non-zero. This principle can be applied to show all of the other statements.

\end{proof}

\begin{lemma} \label{lemma:th2_exp}
In the notation Theorem \ref{thm:laplace} and under Assumptions \ref{ass:hess_lap} -- \ref{ass:fourth_lap}, if $\alpha < \min\{ (3 - 2\thirdpower)/(3 + 2c_{\infty}) , (4 - 2\fourthpower)/(5 + 4c_{\infty} )\}$ then,
\begin{align*}
   &\int_{ E_{\zerop}(\gamma_n, n^{-1/2}\Sigma^{1/2})} \exp[ 2\max\{  0, R_{3,n}(\thetanew, \maximizer) + R_{4,n}(\thetanew, \tilde{\theta})\}]  \phi\left( \thetanew; 0,  I_p/n \right)  d\thetanew \\
   &\quad \leq 1 + O\left[ \max\left\{ \frac{p^{3 + 2c_{\infty} }\log(n)^2}{n^{3 - 2\thirdpower}}, \frac{p^{5+ 4c_{\infty}} \log(n)^2}{n^{4 - 2\fourthpower}} \right\} \right] ,
\end{align*}
where $\tilde{\theta} = \tau(\theta) \theta +\{ 1- \tau(\theta) \}\maximizer$ for  $0 \leq \tau(\theta) \leq 1$.
\end{lemma}

\begin{proof}
Note, 
\*[
&2\max\{  0, \bar{R}_{3,n}(\thetanew) + \bar{R}_{4,n}(\thetanew, \tilde{\theta}))\} \\
&\leq 2\left| \bar{R}_{3,n}(\thetanew) + \bar{R}_{4,n}(\thetanew, \tilde{\theta}))\right|\\
&\leq  \sum_{j = 1}^p |\thetanew_j|\left\{  \left|\thetanew^\top A_j \thetanew \right| + \left|\sum_{k = 1}^p \thetanew_k \left( \thetanew^\top B_{jk}(\tilde\theta) \thetanew \right)\right|  \right\}\\
&:= \sum_{j = 1 } |\thetanew_j| t_j(\thetanew, \tilde\theta).
\]
We can uniformly bound 
\*[
|t_j(\thetanew, \tilde\theta)| &\leq \sup_{\thetanew \in E_{\zerop}(\gamma_n, n^{-1/2}\Sigma^{1/2})} \left\{ \norm{\thetanew}_2^2 \norm{A_j}_{op}  + \norm{\thetanew}_1  \max_{j = 1, \dots, p} \norm{\thetanew}_2^2 \norm{B_{jk}(\tilde\theta)}_{op}  \right\}\\
&\leq \sup_{\thetanew \in E_{\zerop}(\gamma_n, n^{-1/2}\Sigma^{1/2})} \left\{ \norm{\thetanew}_2^2 \norm{A_j}_{op}  + p^{1/2}  \max_{k = 1, \dots, p} \norm{\thetanew}_2^3 \norm{B_{jk}(\tilde\theta)}_{op}  \right\} \\
&= O\left[ \max\left\{ \frac{p^{1 + c_\infty }\log(n)}{n^{1 - \thirdpower}}, \frac{p^{2+ 2c_\infty }\log(n)^{3/2}}{n^{3/2 - \fourthpower}} \right\} \right], 
\]
by Rayleigh's quotient, the $L^p$ inequality and Assumptions \ref{ass:third_lap}--\ref{ass:fourth_lap}. This upper bound is also uniform in $k$ by Assumption \ref{ass:fourth_lap}.
Using this upper bound on $|t(\thetanew, \tilde\theta)|$, we can upper bound the integral of interest by a product of moment generating distributions for the standard normal by, 
\*[
&\int_{ E_{\zerop}(\gamma_n, n^{-1/2}\Sigma^{1/2})}  \exp\left\{\sum_{j = 1}^k |\thetanew_j| |t_j(\thetanew, \tilde\theta)| \right\}  \phi\left( \thetanew; 0,  I_p/n \right)  d\thetanew \\
&\leq \int_{ E_{\zerop}(\gamma_n,n^{-1/2} \Sigma^{1/2})}  \exp\left( \sum_{j = 1}^p |\thetanew_j|  O\left[ \max\left\{ \frac{p^{1 + c_\infty }\log(n)}{n^{1 - \thirdpower}}, \frac{p^{2+ 2c_\infty }\log(n)^{3/2}}{n^{3/2 - \fourthpower}} \right\} \right] \right)  \phi\left( \thetanew; 0,  I_p/n \right)  d\thetanew \\
&\leq \int_{ \mathbb{R}^p}  \exp\left( \sum_{j = 1}^p |\thetanew_j|  O\left[ \max\left\{ \frac{p^{1 + c_\infty }\log(n)}{n^{1 - \thirdpower}}, \frac{p^{2+ 2c_\infty }\log(n)^{3/2}}{n^{3/2 - \fourthpower}} \right\} \right] \right)  \phi\left( \thetanew; 0,  I_p/n \right)  d\thetanew\\
&= \prod_{j= 1}^p \int_{ \mathbb{R}}  \exp\left( |\thetanew_j|  O\left[ \max\left\{ \frac{p^{1 + c_\infty }\log(n)}{n^{1 - \thirdpower}}, \frac{p^{2+ 2c_\infty }\log(n)^{3/2}}{n^{3/2 - \fourthpower}} \right\} \right] \right)  \phi\left( \thetanew_j; 0,  1/n \right)  d\thetanew_j \\
&\leq \prod_{j= 1}^p 2 \int_{ \mathbb{R}}  \exp\left(n^{1/2} \thetanew_j  O\left[ \max\left\{ \frac{p^{3/2 + c_\infty }\log(n)}{n^{3/2 - \thirdpower}}, \frac{p^{2+ 2c_\infty }\log(n)^{3/2}}{n^{2 - \fourthpower}} \right\} \right] \right)  \phi\left( \thetanew_j; 0,  1/n \right)  d\thetanew_j \\
&\leq 2 \left( \int_{ \mathbb{R}}  \exp\left( Z  O\left[ \max\left\{ \frac{p^{1 + c_\infty }\log(n)}{n^{3/2 - \thirdpower}}, \frac{p^{2+ 2c_\infty }\log(n)^{3/2}}{n^{2 - \fourthpower}} \right\} \right] \right)  \phi\left( Z; 0,  1 \right)  dZ \right)^p\\
&= \exp\left( p  O\left[ \max\left\{ \frac{p^{2 + 2c_{\infty} }\log(n)^2}{n^{3 - 2\thirdpower}}, \frac{p^{4 + 4c_{\infty}} \log(n)^3}{n^{4 - 2\fourthpower}} \right\} \right] \right) \\
&= 1 + O\left[ \max\left\{ \frac{p^{3 + 2c_{\infty} }\log(n)^2}{n^{3 - 2\thirdpower}}, \frac{p^{5+ 4c_{\infty}} \log(n)^2}{n^{4 - 2\fourthpower}} \right\} \right],
\]
for $\alpha < \min\{ (3 - 2\thirdpower)/(3 + 2c_{\infty}) , (4 - 2\fourthpower)/(5 + 4c_{\infty} )\}$, showing the desired result. 
\end{proof}

\section{Proof of lemmas used in Corollary \ref{cor:logistic_laplace} and \ref{cor:logistic_glm}}

\begin{lemma}\label{lemma:gaussian_rm}
Let $X$ be a $n \times p$ matrix of centered Gaussian entries with $\max_{j, k } Var(X_{jk}) = \sigma^2 < \infty$, then
\begin{align*}
\max_{j,k} |X_{jk}| = O \{\log(n)^{1/2}\},
\end{align*}
with probability $1 - O(p/n)$.
\end{lemma}

\begin{proof}
This is a classic result obtained by using
\*[ P(Z > \sigma t) \leq \frac{1}{2\pi \sigma t} \exp(-\sigma^2 t^2/2) , \]
where $Z$ is a standard Gaussian random variable \citep[Theorem 1.2.6]{probability}, along with the union bound to control the maximum.



\end{proof}

\begin{lemma}\label{lemma:corr_ass}
The logistic model in Corollary \ref{cor:logistic_laplace} satisfies:
\*[
    \frac{\det\{ -l_n^{(2)}(\hat\beta_{mle}) \}^{1/2}}{(2\pi)^{p/2}}\int_{B^C_{\hat{\beta}_{mle}}(\gamma_n \log(n))  } \frac{\pi(\beta)}{\pi(\hat\beta_{mle})} \exp\{ l_n(\beta) - 
l_n(\hat{\beta}_{mle} )\} d\beta = O(n^{-\eta_1p/8}),
\] 
and $\eta_1 n \leq \lambda_p\{-l_n^{(2)}(\beta)\} \leq \lambda_1\{ -l_n^{(2)}(\beta)\} \leq \eta_2 n$ for $\beta \in B_{\hat{\beta}_{mle}}(\gamma_n\log(n))$, with probability tending to 1.
\end{lemma}

\begin{proof}
It is shown in \cite{non-uniform} that $\lVert\hat\beta_{mle} - \beta_0\rVert_\infty \leq\log(n)/n^{1/2}$, with probability tending to 1, which implies $\lVert\hat\beta_{mle} - \beta_0\rVert_2 \leq p^{1/2}\log(n)/n^{1/2} = \gamma_n \log(n)^{1/2}$, therefore
\*[
\pi(\hat\beta_{mle}) &= \frac{1}{(2\pi)^{p/2}} \exp\left\{-\frac{1}{2}\hat\beta_{mle}^\top \hat\beta_{mle} \right\} \\
&= \frac{1}{(2\pi)^{p/2}} \exp\left\{-\frac{1}{2}\norm{\hat\beta_{mle} - \beta_0 }_2^2 \right\} \\
&\geq\frac{1}{(2\pi)^{p/2}} \exp\left\{ -\frac{1}{2} \gamma_n^2 \log(n)  \right\},
\]
next the maximum value of $l_n(\beta) - l_n(\hat\beta_{mle})$ in $B^C_{\hat\beta_{mle}}(\gamma_n\log(n))$ must lie on the boundary defined by $\lVert\beta - \hat\beta_{mle} \rVert_2 = \gamma_n\log(n)$ since the log-likelihood function is concave in $\beta$. Through a second order Taylor expansion, we have
\[\label{eq:hessian_corr_expan}
l_n(\beta) - l_n(\hat\beta_{mle}) = \frac{1}{2} \beta^\top l_n^{(2)}(\tilde\beta) \beta,
\]
where $\tilde\beta = \{1 - \tau(\beta)\}\hat\beta_{mle} + \tau(\beta)\beta$ for $0 \leq \tau(\beta) \leq 1$. Note, 
\[\label{eq:hessian_corr}
-l_n^{(2)}(\beta) = X^\top D X,
\] 
where $[D]_{jj} = p(x_j^\top\beta)\{1 -p(x_j^\top\beta)\}$ and (\ref{eq:hessian_corr}) is positive definite with eigenvalues which are $O(n)$ if $\max_{j = 1, \dots, n}|x_j^\top\beta|$ is bounded and the matrix $X^\top X$ is also positive definite with eigenvalues which are $O(n)$. For $\beta \in B_{\hat\beta_{mle}}(\gamma_n\log(n))$, 
\*[
\max_{j = 1, \dots, n}|x_j^\top\beta| &\leq \max_{j = 1,\dots, n } \norm{x_j}_2 \norm{\beta}_2 \leq \max_{j = 1,\dots, n } \norm{x_j}_2 \left\{ \norm{\beta - \hat\beta_{mle}}_2 + \norm{\hat\beta_{mle}}_2 \right\} \\
&\leq p^{1/2} \max_{j,k = 1 , \dots, p} | x_{j,k}| \frac{2p^{1/2}\log(n)^{3/2}}{n^{1/2}} = O\left( \frac{p\log(n)^{2}}{n^{1/2}} \right),
\]
which is bounded if $\alpha < 1/2$, and $X^\top X$ satisfies the necessary criteria with probability tending to $1$ by Theorem 4.6.1 in \cite{vershynin2018high}. 
Therefore, for some $\eta_1 > 0$, and for $\lVert \beta- \hat\beta_{mle} \rVert_2 = \gamma_n \log(n)$, by Rayleigh's quotient,
\*[
(\ref{eq:hessian_corr_expan}) &\leq - \frac{1}{2} \norm{\beta}_2^2 \eta_1 n \\
&\leq - \frac{\eta_1}{2} \left\{\norm{\beta - \hat\beta_{mle}}^2_2  - \norm{\hat\beta_{mle}}_2^2  \right\} n \\
&\leq - \frac{\eta_1}{2} \left\{\gamma_n^2\log(n)^2  - \gamma_n^2 \log(n) \right\}  n\\
&\leq - \frac{\eta_1}{4} \gamma^2_n\log(n)^2 n, 
\]
for $n$ sufficiently large.
Therefore, noting that $ \det\{ -l_n^{(2)}(\hat\beta_{mle}) \}^{1/2} \leq (\eta_2 n)^{p/2}$
\*[
&\frac{\det\{ -l_n^{(2)}(\hat\beta_{mle}) \}^{1/2}}{(2\pi)^{p/2}} \int_{B^C_{\hat{\beta}_{mle}}(\gamma_n \log(n))  } \frac{\pi(\beta)}{\pi(\hat\beta_{mle})} \exp\{ l_n(\beta) - 
l_n(\hat{\beta}_{mle})\} d\beta \\
&\leq \det\{ -l_n^{(2)}(\hat\beta_{mle}) \}^{1/2} \exp\left\{ \frac{1}{2} \gamma_n^2 \log(n) \right\} \int_{B^C_{\hat{\beta}_{mle}}(\gamma_n \log(n))  } \pi(\beta) \exp\{- \eta_1\gamma_n^2n\log(n)^2/4\} d\beta \\
&= \exp\left\{ \frac{p}{2}\log(n) + \frac{p}{2}\log(\eta_2)+ \frac{1}{2} \gamma_n^2 \log(n) - \eta_1\gamma_n^2n\log(n)^2/4  \right\} \int_{B^C_{\hat{\beta}_{mle}}(\gamma_n \log(n))  } \pi(\beta) d\beta \\
&\leq \exp\left\{ - \eta_1\gamma_n^2n\log(n)^2/8  \right\} \leq O\left( n^{-\eta_1p/8}\right),
\]
where the last line holds for $n$ sufficiently large, and by the fact that the integral of a density is bounded by $1$.
\end{proof}

\begin{lemma}\label{lemma:cor_infinity}
Under the notation and assumptions of Corollary \ref{cor:logistic_laplace}, 
\*[
\norm{\{X^\top D X\}^{-1/2} }_{\infty} = O(n^{-1/2}),
\]
where $[D]_{jj} = p(x_j^\top \hat\beta_{mle})\{ 1- p(x_j^\top \hat\beta_{mle})\} =p^{(1)}(x_j^\top \hat\beta_{mle}) $  .
\end{lemma}

\begin{proof}
By a second order Taylor expansion centered at 0, 
\*[
{D} &=\text{diag}\left\{ p^{(1)} ( 0 ) +  p^{(3)} ( r_j )(x_j^\top \hat\beta_{mle})^2 \right\}_{j = 1, \dots, n}\\
&=\frac{1}{4}I_n + \text{diag}\left\{ p^{(3)} ( r_j )(x_j^\top \hat\beta_{mle})^2 \right\}_{j = 1, \dots, n} := \frac{1}{4}	I_n + \frac{1}{4}  R ,
\]
where $p^{(j)}$ is $j$-th derivative of the probability of success  in (\ref{eq:logistic_eq}), $p^{(2)} (0) = 0$ and $r_k$ lies between $0$ and $x_k^{\top} \hat\beta_{mle}$. 
We have $\lVert R \rVert_{op} = O(p/n)$ from $\max_{j = 1, \dots, n} |p^{(3)} ( r_j )(x_j^\top \hat\beta_{mle})^2| = O(p/n)$ implied by the boundedness of $p^{(3)}(\cdot)$ and $\max_{j = 1, \dots, n}|x_j^\top \hat\beta_{mle}| = O(p^{1/2}/n^{1/2})$ \citep{non-uniform}. 

\[
\norm{ \{ X^\top D X \}^{-1/2}  }_{\infty} &= 2 n^{-1/2} \norm{ \{ I_p + 4 X^\top D X/n - I_p \}^{-1/2}  }_{\infty} \nonumber \\
&:= 2 n^{-1/2} \norm{ \{ I_p + E \}^{-1/2}  }_{\infty},  \label{eq:norm_calc} 
\]
the maximal singular value of $E$ is bounded by,
\*[
\norm{E}_{op} &= \norm{4X^\top D X/n - I_p}_{op} = \norm{X^\top X/n + X^\top R X/n - I_p}_{op} \\
&\leq \norm{X^\top X/n - I_p}_{op} + \norm{X^\top R X/n}_{op} = O\left( \frac{p^{1/2}}{n^{1/2}} \right),
\]
with probability tending to 1 as $\lVert X^\top X/n - I_p \rVert_{op} = O(p^{1/2}/n^{1/2})$ and $\lVert X^\top X/n \rVert_{op} = 1 + O(p^{1/2}/n^{1/2})$ with probability tending to 1 by Theorem 4.6.1 in \cite{vershynin2018high}.
We use the following expansions, which are valid if $\lVert E \rVert_{op}< 1$ and $\lVert I - A \rVert_{op} \leq 1$,
\*[
 (I_p - E)^{-1} &= I_p + \sum_{j = 1}^\infty E^j,\\
 A^{1/2} = I_p - \sum_{j = 1}^{\infty} \left| {1/2 \choose j} \right| (I_{p} - A)^{j}, &\text{ where } {1/2 \choose j} = {2j \choose j} \frac{(-1)^{j+1}}{2^{2j}(2j - 1)}, 
\]
to write
\*[ 
(\ref{eq:norm_calc}) &= 2 n^{-1/2} \norm{ I_p - \sum_{k = 1}^{\infty}\left| {1/2 \choose k} \right|\left(  -\sum_{j = 1}^\infty (-E)^j  \right)^k  }_{\infty} \leq 2 n^{-1/2} \left\{  1 + \sum_{k = 1}^{\infty} \norm{\left(  \sum_{j = 1}^\infty (-E)^j  \right)^k  }_{\infty} \right\} \\
&\leq 2 n^{-1/2} \left\{  1 + p^{1/2} \sum_{k = 1}^{\infty} \norm{\left(  \sum_{j = 1}^\infty (-E)^j  \right)^k  }_{op} \right\} \leq 2 n^{-1/2} \left\{  1 + p^{1/2} \sum_{k = 1}^{\infty} \left(  \sum_{j = 1}^\infty \lVert E  \rVert_{op}^j \right)^k \right\} \\
&\leq 2 n^{-1/2} \left[  1 + p^{1/2} \sum_{k = 1}^{\infty} \left\{ O\left( \frac{p^{1/2}}{n^{1/2}} \right) \sum_{j = 0}^\infty O\left( \frac{p^{1/2}}{n^{1/2}} \right)^j \right\}^k \right] \\
 &\leq 2 n^{-1/2} \left\{  1 + O\left(\frac{p}{n^{1/2}} \right) \right\} =  O(n^{ -1/2}), 
\]
for values of $\alpha < 2/5$, by using the convergence of a geometric series and the fact that magnitude of the binomial coefficient for $1/2$ choose $j$ are bounded by $1$ for all $ j = 1, 2, \dots$.
\end{proof}

\begin{lemma}\label{glm:cor}
    For the GLM model as specified and under the same conditions as in Corollary \ref{cor:logistic_glm},
    \*[
        \max_{j = 1, \dots, n}|x^\top_j(\hat\beta_{mle} - \beta_0)| = O\left( \frac{p\log(n)^{3/2}}{n^{1/2}} \right)
    \]
    and,
    \*[
        \norm{ \{XD(\hat\beta_{mle})X\}^{-1/2}  }_\infty = O\left( n^{-1/2} \right), 
    \]
    with probability tending to $1$.
\end{lemma}

\begin{proof}
    First note that: 
    \*[\label{eq:max_glm_bound}
        \max_{j = 1, \dots, n} |x_j^\top (\hat\beta_{mle} - \beta_0)| \leq  \max_{j = 1, \dots, n} \norm{x_j}_2 \norm{\hat\beta_{mle} - \beta_0}_2.
    \]
     Given that the vectors $x_j$ are jointly independent and follows a normal distribution, $\lVert x_j \rVert_2^2 \sim \chi^2_p$, it follows from Lemma \ref{lemma:max_exp} that $\max_{j = 1, \dots, n} \norm{x_j}_2 = O_p[\{p\log(n)\}^{1/2}]$.
    From \cite{Fan} we know that $\lVert \hat\beta_{mle} - \beta_0 \rVert_{\infty} = O\{\log(n)/n^{1/2} \}$ with probability tending to 1, thus by the $L^p$ inequality $\lVert \hat\beta_{mle} - \beta_0 \rVert_2 = O(\log(n)p^{1/2}/n^{1/2})$ implying: 
    \*[
        \max_{j = 1, \dots, n} |x_j^\top (\hat\beta_{mle} - \beta_0)| = O\left\{ \frac{p\log(n)^{3/2}}{n^{1/2}} \right\}.
    \]
    As for the second statement, consider the first order Taylor expansion centered at $\beta_0$, 
    \*[
    {D} &=\text{diag}\left\{ K^{(2)} ( x_j^\top\beta_0 ) +  K^{(3)} ( r_j )(x_j^\top \hat\beta_{mle} - x_j^\top\beta_{0}) \right\}_{j = 1, \dots, n}\\
    &=D(\beta_0 ) + \text{diag}\left\{ K^{(3)} ( r_j )(x_j^\top \hat\beta_{mle}- x_j^\top\beta_{0}) \right\}_{j = 1, \dots, n} := D(\beta_0) +  R ,
    \]
    where $K^{(j)}$ is $j$-th derivative of the cumulant generating function in [give reference] and $r_j$ lies between $x_j^{\top}\beta_0$ and $x_j^{\top} \hat\beta_{mle}$. 
    We have $\lVert R \rVert_{op} = O(p/n)$ from
     \*[
         \max_{j = 1, \dots, n} |K^{(3)} ( r_j )(x_j^\top \hat\beta_{mle}- x_j^\top\beta_{0})| 
         &\leq \max_{j = 1, \dots, n} |K^{(3)} ( r_j )| \norm{x_j}_2 \norm{\hat\beta_{mle} - \beta_{0}}_2 \\
         &= O\left\{ \frac{\log(n)^{3/2} p}{n^{1/2}} \right\},
    \] 
    as $\max_{j = 1, \dots, p}\lVert x_j \rVert_2 = O\{\log(n)^{1/2}p^{1/2}\}$ and $K(r_j) = O(1)$ as $\max_{j = 1, \dots, n} x_j^\top (\hat\beta_{mle} -\beta_0 ) \rightarrow 0$ if $\alpha < 1/2$ and $\max_{j = 1, \dots, n} x_j^\top (\beta_0 ) = O(1)$ by assumption and the derivatives of the cumulant generating function are continuos functions. 
    \[
    \norm{ \{ X^\top D(\hat\beta_{mle}) X \}^{-1/2}  }_{\infty} &\leq \norm{\{X^\top D(\beta_{0}) X\}^{-1/2}}_{\infty} \norm{ \left[ I_p + \{X^\top D(\beta_{0}) X\}^{-1}\left\{ X^\top R X \right\} \right]^{-1/2}   }_{\infty} \nonumber \\
    &:= \norm{\{X^\top D(\beta_{0}) X\}^{-1/2}}_{\infty} \norm{ \{ I_p + E \}^{-1/2}  }_{\infty},  
    \]
    the maximal singular value of $E$ is bounded by,
    \*[
    \norm{E}_{op} &= \norm{\{X^\top D(\beta_{0}) X\}^{-1}\left\{ X^\top R X \right\}}_{op}  \\
    &\leq\norm{\{X^\top D(\beta_{0}) X\}^{-1}}_{op} \norm{ X^\top R X }_{op}  = O\left\{ \frac{\log(n)^{3/2} p}{n^{1/2}} \right\},
    \]
    with probability tending to 1 as $\lVert X^\top X/n \rVert_{op} = 1 + O(p^{1/2}/n^{1/2})$ with probability tending to 1 by Theorem 4.6.1 in \cite{vershynin2018high} and our bound on the operator norm of $R$. Using the same expansion argument as in the proof of \ref{lemma:cor_infinity}, we show
    \*[
        \norm{ \left[ I_p + \{X^\top D(\beta_{0}) X\}^{-1}\left\{ X^\top R X \right\} \right]^{-1/2}   }_{\infty} = O\left\{ \frac{\log(n)^{3/2} p^{3/2}}{n^{1/2}} \right\},
    \]
    which combined with our assumption that $\lVert \{X^\top D(\beta_{0}) X\}^{-1/2}\rVert_{\infty} = O(n^{-1/2})$ shows the desired result.
\end{proof}

\begin{lemma} \label{lemma:max_exp}
    Let $\chi^2_{p, j}$ for $j = 1, \dots, n$ be independent copies of a Chi-Square variable with $p$ degrees of freedom, then: 
    \*[ \max_{j = 1, \dots, n} (\chi^2_{p, j})^{1/2} = O\{p^{1/2}\log(n)^{1/2}\}, \]
    with probability  $1 - O(p/n)$.
\end{lemma}
\begin{proof}
    Let $p$ be an even number, there exits a probability space such that: $\chi^2_{p,j} = \sum_{k = 1}^{p/2} e_{j,k}$, where $e_j$ are exponential random variables with rate parameters $1/2$. Then
    \*[
       \max_{j = 1, \dots, n} \chi^2_{p,j} &= \max_{j = 1, \dots, n}\sum_{j = 1}^{p/2} e_{j, k} \leq \frac{p}{2} \max_{j, k} e_{j, k}, 
    \]
    where $j$ and $k$ ranges over $1, \dots, n$ and $1, \dots, p$ respectively. Noting that: 
    \*[
        &P\left[\frac{2\max_{k = 1, \dots, n}\{ \chi^2_{p, j} \}}{p} > \log(n) \right] \leq P\left[ \max_{j, k} e_{j, k} > \log(n) \right]  \\
         &\leq \sum_{j, k = 1}^{n, p/2} P\{ e_{j, k} > \log(n) \} = np \exp\{ -2\log(n) \}/2 = \frac{p}{2n},
    \]
    showing the desired result. If $p$ is odd, then we note that there exits a probability space such that $\chi^2_{p,j} \leq \sum_{k = 1}^{(p+1)/2} e_{j,k}$, and proceed as before.
\end{proof}

\section{Proof of lemmas needed for Theorem \ref{thm:ratio_exp_laplace} }

\begin{lemma}\label{lemma:laplace_ratio_sup}
Under Assumptions \ref{ass:third_lap}--\ref{ass:exp_ratio} on the numerator of  (\ref{eq:marg_posterior}) for $\alpha < 1/2 - 1/(2\zeta - 2)$,
\*[
 \exp\left\{ R_{\zeta ,n}(\theta, \tilde{\theta}) \right\} =1 + O \left( \frac{p^{\zeta -1}\log(n)^{\zeta/2}}{n^{ (\zeta -2)/2 }} \right),\]
for $\theta \in [-\gamma_n, \gamma_n] \times B_{ \zeropminus}(\gamma_n)$ and
\*[ 
\exp\left\{  R^{\nuissance}_{\zeta ,n}(\lambda, \tilde{\theta}) \right\} =1 + O \left( \frac{p^{\zeta -1}\log(n)^{\zeta/2}}{n^{ (\zeta -2)/2 }} \right), 
\]
for $\lambda \in B_{ \zeropminus}(\gamma_n)$, where $\tilde\theta = \tau(\theta) + \{1 - \tau(\theta)\}\maximizer$ for $0 \leq  \tau(\theta) \leq 1$.
\end{lemma}

\begin{proof}
Note that $[-\gamma_n, \gamma_n] \times B_{ \zeropminus}(\gamma_n) \subset B_{\zerop}(2^{1/2}\gamma_n)$, thus
 \*[
 \left| R_{\zeta ,n}(\theta, \tilde{\theta})\right| &= \left|  \sum_{j_1 \dots j_{\zeta - 2} = 1 }^{p -1} \theta_{j_1}\cdots\theta_{j_{k - 2}} \left\{ \theta^\top g_{\cdot\cdot j_1 \dots j_{\zeta - 2}}^{(\zeta)}(\tilde{\theta}) \theta \right\} \right|\\
 &\leq n \norm{\theta}_2^2 \left| \sum_{j_1 \dots j_{\zeta - 2} = 1 }^{p -1} \theta_{j_1}\cdots\theta_{j_{k - 2}} \right| \leq n \norm{\theta}_2^2  \norm{\theta}_1^{\zeta -2} \leq n \norm{\theta}_2^2  \norm{\theta}_2^{\zeta -2} p^{(\zeta -2)/2}\\
 &= O \left( \frac{p^{\zeta - 1} \log(n)^{\zeta/2}}{n^{(\zeta - 2)/2}} \right), 
\]
where $\tilde\theta = \tau(\theta)\theta + \{1 - \tau(\theta)\} \maximizer$, for some $0 \leq \tau(\theta) \leq 1$.
Since $\exp(a_n) = 1+ O(a_n)$ for a sequence $a_n \rightarrow 0$ completes the proof for the first statement. The second statement of the lemma can be shown in the same manner.
\end{proof}

\begin{lemma}\label{lemma:laplace_ratio_center}
Under Assumptions \ref{ass:third_lap}--\ref{ass:exp_ratio} for $\alpha < 1/2 - 1/(2\zeta - 2)$
\*[
\left|\sum_{j = 3}^{\zeta - 1} R^{\nuissance}_{j,n}(\lambda, \maximizer) - \sum_{j = 3}^{\zeta - 1} R^{\nuissance}_{j,n}(\lambda, \maximizercons)\right| = O\left\{ \max\left( \frac{\log(n)^{2}p^2}{n^{2 - \fourthpower}}, \frac{\log(n)^{5/2}p^3}{n^{3/2}} \right) \right\},
\]
for all $\lambda \in B_{\zeropminus}(\gamma_n)$, where $\gamma_n^2 = p\log(n)/n$. 
\end{lemma}

\begin{proof}
First consider $ 4 \leq j \leq \zeta - 1$, 
\*[
R^{\nuissance}_{j,n}(\lambda,\maximizer) &= \frac{1}{j!}\sum_{k_1 \dots k_j = 1 }^{p-1 } \nuissance_{k_1} \cdots \nuissance_{k_j} \loglike^{(j)}_{k_1 \dots k_j} (\maximizer)\\
&=\frac{1}{j!}\sum_{k_1 \dots, k_j = 1}^{p-1 } \nuissance_{k_1} \cdots \nuissance_{k_j} \loglike^{(j)}_{k_1 \dots k_j} (\maximizercons) +  \frac{(\hat\psi  -\psi)}{j!}\sum_{k_1 \dots k_j = 1 }^{p-1 } \nuissance_{k_1} \cdots \nuissance_{k_j} \loglike^{(j + 1)}_{\psi k_1 \dots k_j} (\tilde{\theta}), \\
&= R^{\nuissance}_{j,n}(\lambda, \maximizercons) +  \frac{(\hat\psi  -\psi)}{j!}\sum_{k_1 \dots k_j = 1 }^{p-1 } \nuissance_{k_1} \cdots \nuissance_{k_j} \loglike^{(j + 1)}_{\psi k_1 \dots k_j} (\tilde{\theta})
\]
where $\tilde{\theta} = (\tilde{\psi}, \hat\lambda)$ and $\tilde{\psi} = \tau(\psi)\psi + \{1 - \tau(\psi)\} \hat\psi $ for $0 \leq \tau(\psi) \leq 1$. Thus,
\*[
&\left| R^{\nuissance}_{j,n}(\lambda, \maximizer) - R^{\nuissance}_{j,n}(\lambda, \maximizercons)\right| =  \left|\frac{(\hat\psi  - \psi)}{j!}\sum_{k_1 \dots k_j =1}^{p-1 } \nuissance_{k_1} \cdots \nuissance_{k_j} \loglike^{(j + 1)}_{\psi k_1 \dots k_j} (\tilde{\theta}) \right| \\
&\leq O\left\{ \frac{\log(n)^{1/2}}{n^{1/2}} \right\}  \left|\sum_{k_1 \dots k_{j - 2} = 1}^{p-1 } \nuissance_{k_1} \cdots \nuissance_{k_{j - 2}} \left\{ \lambda^\top \loglike^{(j + 1)}_{\cdot \cdot\psi k_1 \dots k_{j - 2} } (\tilde{\theta}) \lambda \right\} \right| \\
&\leq O\left\{ \frac{\log(n)^{1/2}}{n^{1/2}} \right\} C_j n \norm{\nuissance}_2^2 \sum_{k_1 \dots k_{j - 2} =1 }^{p-1 } |\nuissance_{k_1}| \cdots |\nuissance_{k_{j - 2}}| \\
&\leq O\left\{ \frac{\log(n)^{1/2}}{n^{1/2}} \right\} C_j n \norm{\nuissance}_2^2 \{ (p - 1)^{(j - 2)/2}\norm{\nuissance}_2^{ (j - 2) } \} \leq  O\left\{ \frac{\log(n)^{1/2}}{n^{1/2}} \right\} O\{n\gamma_n^{j} p^{(j - 2)/2}\}\\
&= O \left\{ \frac{\log(n)^{(j +1)/2}  p^{j - 1}}{n^{(j- 1)/2}} \right\}, 
\]
by using Assumption \ref{ass:exp_ratio} with Rayleigh's quotient and $\norm{\nuissance}_1 \leq (p - 1)^{1/2} \norm{\nuissance}_2$. If $\alpha \leq 1/2 - 1/(2\zeta - 2)$,  
 \*[
 \left|\sum_{j = 4}^{\zeta - 1} R^{\nuissance}_{j,n}(\lambda, \maximizer) - \sum_{j = 4}^{\zeta - 1} R^{\nuissance}_{j,n}(\lambda, \maximizercons)\right| = O\left\{ \frac{\log(n)^{5/2 }p^3 }{n^{3/2}} \right\},
 \]
 by applying the triangle inequality.
 As for $j = 3$, the same series inequality holds, except we use Assumption \ref{ass:fourth_lap} instead of Assumption \ref{ass:exp_ratio} to when applying Rayleigh's quotient to obtain
 \*[
 \left| R^{\nuissance}_{3,n}(\lambda, \maximizer) -  R^{\nuissance}_{3,n}(\lambda, \maximizercons)\right| = O\left( \frac{\log(n)^2 p^2}{n^{2 - \fourthpower}}\right);
\]
using the triangle inequality gives the desired result. 

\end{proof}

\begin{lemma}\label{lemma:ratio_cancellation}
Under Assumptions \ref{ass:third_lap}--\ref{ass:exp_ratio},
\*[
\int_{B_{\zeropminus}(\gamma_n) }  \exp\left\{ \sum_{j = 3}^{\zeta - 1} R^{\nuissance}_{j,n}(\lambda, \maximizer)  \right\} \phi\left[ \nuissance; 0, \{-\loglikeminus^{(2)}(\maximizer)\}^{-1} \right] d\nuissance  < \infty,
 \]
if $p = O(n^{\alpha})$ for $\alpha < 1/2 - 1/(2\zeta - 2)$.
\end{lemma}

\begin{proof}
We will relate the above quantity to the moment generating function of a $\chi^2_p$ distribution in order to show that it is finite. 
Each of the terms
\*[
\left| R^{\nuissance}_{j,n}(\lambda, \maximizer) \right|  &\leq O(n) \norm{\lambda}_2^2  \norm{\lambda}_1^{j -2} \\
&\leq O (n) \norm{\lambda}_2^2  \norm{\lambda}_2^{j -2} p^{(j -2)/2} = \left( n\norm{\lambda}_2^2 \right)  O\left( \frac{p^{j - 2} \log(n)^{(j - 2)/2}}{n^{ (j-2)/2}} \right),
\]
under the assumptions that $\alpha < 1/2 - 1/(2\zeta - 2)$ 
\*[
 \left| \sum_{j = 1}^{\zeta - 1} R^{\nuissance}_{j,n}(\lambda, \maximizer) \right| = \left( n\norm{\lambda}_2^2 \right)  O\left( \frac{p \log(n)^{1/2}}{n^{ 1/2}} \right),
 \] 
therefore
\*[
&\left|\int_{B_{\zeropminus}(\gamma_n) }  \exp\left\{ \sum_{j = 3}^{\zeta - 1} R^{\nuissance}_{j,n}(\nuissance, \maximizer)  \right\} \phi\left[ \nuissance; 0, \{-\loglikeminus^{(2)}(\maximizer)\}^{-1} \right] d\nuissance \right|\\
&\leq \int_{B_{\zeropminus}(\gamma_n) } \exp\left\{  \left( n\norm{\lambda}_2^2 \right)  O\left( \frac{p \log(n)^{1/2}}{n^{ 1/2}} \right)\right\} \phi\left[ \nuissance; 0, \{-\loglikeminus^{(2)}(\maximizer)\}^{-1} \right] d\nuissance\\
&\leq  \int_{\mathbb{R}^{p-1} } \exp\left\{ n [Z^\top \{-\loglikeminus^{(2)}(\maximizer)\}^{-1} Z]   O\left( \frac{p \log(n)^{1/2}}{n^{ 1/2}} \right)\right\} \phi\left[ Z; 0, I_{p - 1} \right] dZ\\
&\leq \int_{\mathbb{R}^{p-1} } \exp\left\{\norm{Z}_2^2   O\left( \frac{p \log(n)^{1/2}}{n^{ 1/2}} \right)\right\} \phi\left[ Z; 0, I_{p - 1} \right] dZ,
\] 
where the last equality follows from a change of variable $Z = \{-\loglikeminus^{(2)}(\maximizer)\}^{1/2} \lambda$, Rayleigh's quotient and Assumption \ref{ass:hess_lap}. Note that the distribution of $Z$ is that of a vector of independent standard normal random variables. Thus, the above integral is equivalent to evaluating the moment generating function of a $\chi^2_{p - 1}$ distribution at $t =O(p\log^{1/2}(n)/n^{1/2})$. Recalling, 
\*[
E[ \exp(t\norm{Z}_2^2) ] = \left( \frac{1}{1 - 2t} \right)^{p - 1} \text{ for } t < 1/2,
\]
we obtain:
\*[
&\int_{\mathbb{R}^{p-1} } \exp\left\{  \norm{Z}_2^2   O\left( \frac{p \log(n)^{1/2}}{n^{ 1/2}} \right)\right\} \phi\left[ \nuissance; 0, I_{p - 1} \right] d\nuissance \\
&= \left( \frac{1}{1 - O(p\log^{1/2}(n)/n^{1/2})} \right)^{p - 1} < \infty,
\]
as $O\{p\log^{1/2}(n)/n^{1/2}\} \rightarrow 0$, showing the desired result.

\end{proof}

\begin{lemma} \label{Lemma:ratio_determinants}
Under Assumptions \ref{ass:hess_lap} and \ref{ass:third_lap}, for $\psi \in \{ \psi : |\psi - \maxpsi| = O\{ \log^{1/2}(n)/n^{1/2}  \} \}$,
\*[ \left[\frac{ \det\{\loglikeminus^{(2)}(\maximizer)\}}{\det\{\loglikeminus^{(2)}(\maximizercons)\}}\right]^{1/2} = 1 + O\left\{ \frac{p\log^{1/2}(n)}{n^{3/2 - \thirdpower}} \right\}, \]
under the orthogonal parametrization for the linear exponential family.
\end{lemma}
\begin{proof}
We use a first order Taylor series expansion of the numerator,
\*[
\det\{-\loglikeminus^{(2)}(\maximizer)\} &= \det\{-\loglikeminus^{(2)}(\maximizercons) - (\maximizer - \maximizercons) \frac{\partial}{\partial\psi} \loglikeminus^{(2)}(\maximizercons)|_{\psi = \tilde\psi} \}  \\
&=\det \left\{-\loglikeminus^{(2)}(\maximizercons) - (\maxpsi - \psi) \conthirdmat(\maximizerconstild) \right\}\\
&= \det\{-\loglikeminus^{(2)}(\maximizercons) \} \det\{ I  + (\maxpsi - \psi) \{- \loglikeminus^{(2)}(\maximizercons) \}^{-1} \conthirdmat(\maximizerconstild) \}\\
&=\det \{-\loglikeminus^{(2)}(\maximizercons) \} \det ( I  + A ),
\]
It remains to examine the size of the term, $\det( I  + A )$. We use the expansion
\*[
\det \left( I  + A \right)= \sum_{k = 0}^{\infty} \frac{1}{k!} \left( - \sum_{j =1}^\infty \frac{ (-1)^{j} }{j} \text{tr}\left[ A^j \right] \right)^k,
\]
which is a valid expansion if the magnitudes of the entries of $A$ are less than 1. In our case since 
\*[
\norm{A}_{op} = O\{ \log(n)^{1/2}/n^{3/2 - \thirdpower}\} , 
\]
by Assumptions \ref{ass:hess_lap} and \ref{ass:third_lap} on the denominator, the entries of $A$ are $o(1)$. First examining the inner summation over $j$, and using $|\text{tr}[A^j]| \leq (p -1 ) \norm{A}^j_{op}$,
we have
\*[ 
&  \left| \sum_{j =1}^\infty \frac{ (-1)^{j} }{j} \text{tr}\left[ A^j \right]  \right| \leq   \sum_{j =1}^\infty (p -1 ) \norm{A}^j_{op}\\
&\leq p \sum_{j =1}^\infty   O \left\{ \frac{ \log(n)^{1/2}}{n^{3/2 - \thirdpower}}  \right\}^j = O\left\{\frac{p \log(n)^{1/2}}{n^{3/2 - \thirdpower}} \right\}   \sum_{j =1}^\infty   O \left\{ \frac{ \log(n)^{1/2}}{n^{3/2 - \thirdpower}}  \right\}^{j -1}\\
&\leq O \left\{ \frac{p \log(n)^{1/2}}{n^{3/2 - \thirdpower}}  \right\}, 
\]
as $\sum_{j =1}^\infty  O \left\{ \log(n)^{1/2}/n^{3/2 - \thirdpower} \right\}^{j -1} < \infty$, since it is the sum of a convergent geometric sequence. 
The original summation can be bounded as follows,
\*[
|\det (I  + A) | &= \left| \sum_{k = 0}^{\infty} \frac{1}{k!} \left( - \sum_{j =1}^\infty \frac{ (-1)^{j} }{j} \text{tr}\left[ A^j \right] \right)^k  \right|  \\
&\leq 1 + O\left\{ \frac{ p \log(n)^{1/2}}{n^{3/2 - \thirdpower}} \right\} \sum_{k = 1}^{\infty} \frac{1}{k!} O\left\{ \frac{ p \log(n)^{1/2}}{n^{3/2 - \thirdpower}}  \right\}^{k - 1} =  1 + O\left\{ \frac{p \log(n)^{1/2}}{n^{3/2 - \thirdpower}}  \right\} , 
\]
where we have used the fact that $\sum_{k = 1}^{\infty} O \left( p \log(n)^{1/2}/n^{3/2 - \thirdpower} \right)^{k - 1}/k! < \infty$ as it can be upper bounded by the sum of a convergent geometric series. This shows the desired result.

\end{proof}

\begin{lemma}\label{lemma:marg_prob}
Under Assumptions \ref{ass:hess_lap}--\ref{ass:exp_ratio} for the numerator of (\ref{eq:laplace_approxiamtion}),
\*[
\int_{ [ - \gamma_n,  \gamma_n]} 
   \exp \left\{\sum_{j = 3}^{\zeta - 1}R^\interest_{j,n}(\theta, \maximizer) \right\} \phi\left[ \interest; 0, \{-\loglike_{\interest\interest}^{(2)}(\maximizer)\}^{-1} \right] d\interest = 1 + O\left\{ \frac{p^2\log(n)^2}{n}\right\},
\]
for $\alpha < 1/2$ and for all $\lambda \in B_{\zeropminus}(\gamma_n)$.
\end{lemma}

\begin{proof}
We will relate the above integral to the moment generating function of a standard normal distribution. We claim,
\[
 &\sum_{j = 3}^{\zeta - 1}R^\interest_{j,n}(\theta, \maximizer) =  n^{1/2}\psi \sum_{j = 3}^{\zeta - 1}\frac{R^\interest_{j,n}(\theta, \maximizer)}{n^{1/2}\psi} = n^{1/2}\psi \ O\left\{ \frac{p\log(n)}{n^{1/2}} \right\},  \label{eq:mgf_t}
\]
which can be shown by considering the terms in the summation,
\[
\frac{R^\interest_{j,n}(\theta, \maximizer)}{n^{1/2}\psi} = \frac{1}{j! }\sum_{k = 1}^{j} {j \choose k} \psi^{k - 1} \sum_{l_{1}\dots l_{j - k} = 1}^{p - 1} \frac{\lambda_{l_1 }\dots \lambda_{l_{j - k}} \loglike^{(j)}_{\psi \dots \psi l_1 \dots l_{j - k}}(\maximizer)}{n^{1/2}}  \label{eq:lemma11}.
\]
We now break the terms involved in the summation in (\ref{eq:lemma11}) into 3 cases. First, for all $3 \leq j \leq \zeta - 1$ and $k = j$ we have the following upper bound by Assumptions \ref{ass:third_lap}--\ref{ass:exp_ratio},
\*[
\frac{|\psi^{j - 1}|g^{(j)}_{\psi\dots\psi}}{n^{1/2}} \leq \frac{C_j \gamma_n^{j - 1} n}{n^{1/2}} = \frac{p^{(j-1)/2}\log(n)^{(j -1)/2}}{n^{j/2 - 1}} = O\left\{ \frac{p\log(n)}{n^{1/2}} \right\} .   
\]
Secondly for all $3 \leq j \leq \zeta - 1$ and $k = j - 1$,
\*[
&\left|\psi^{j - 2} \sum_{l_{1} = 1}^{p - 1} \frac{\lambda_{l_1 } \loglike^{(j)}_{\psi \dots \psi l_1}(\maximizer)}{n^{1/2}}\right| \leq \frac{\gamma_n^{j - 2}}{n^{1/2}} \norm{\lambda}_2 \norm{\loglike^{(j)}_{\psi \dots \psi \cdot}(\maximizer)}_2 \leq\frac{\gamma_n^{j - 2}}{n^{1/2}} \norm{\lambda}_2 \norm{\loglike^{(j)}_{\psi \dots \psi \cdot \cdot}(\maximizer)}_{op}\\
&= O\left\{\frac{p\log(n)}{n^{1/2}} \right\} ,
\]
by Assumptions \ref{ass:third_lap}--\ref{ass:exp_ratio}, the fact that the maximum singular value of a vector is its $L^2$ norm and that the maximum singular value of a sub-matrix is always smaller than the full matrix. Lastly for all $3 \leq j \leq \zeta - 1$ and $1\leq k \leq j - 2$,
\*[
&\left|\psi^{k - 1} \sum_{l_{1}\dots l_{j - k} = 1}^{p - 1} \frac{\lambda_{l_1 }\dots \lambda_{l_{j - k}} \loglike^{(j)}_{\psi \dots \psi l_1 \dots l_{j - k}}(\maximizer)}{n^{1/2}} \right| \\
&\leq \frac{\gamma_n^{k-1}}{n^{1/2}}  \sum_{l_{1}\dots l_{j - k - 2} = 1}^{p - 1}|\lambda_{l_1 }|\dots |\lambda_{l_{j - k -2}}| \left|\left\{\lambda^\top \loglike^{(j)}_{\psi \dots \psi l_1 \dots l_{j - k -2}\cdot\cdot}(\maximizer) \lambda \right\}\right|  \\
&\leq C_j \frac{\gamma_n^{k-1}}{n^{1/2}}\norm{\lambda}_1^{j - k-2}  \gamma_n^2 n \leq C_j \gamma_n^{k +1} p^{(j - k - 2)/2} \norm{\lambda}_2^{j - k - 2} n^{1/2} \\
&\leq C_j \gamma_n^{j- 1} p^{(j - k - 2)/2 }n^{1/2} =O\left\{ \frac{p^{ (2j - k - 3)/2}\log(n)^{(j - 1)/2}}{n^{j/2 - 1}} \right\}\\
&\leq  O\left\{ \frac{p^{j - 2}\log(n)^{(j - 1)/2}}{n^{j/2 - 1}} \right\} \leq O\left\{ \frac{p\log(n)}{n^{1/2}}\right\},
\]
by Rayleigh's quotient and $\norm{\lambda}_1 \leq (p - 1)^{1/2} \norm{\lambda}_2$. Therefore we have shown (\ref{eq:mgf_t}) holds. Thus, 
\*[
&\int_{ [ - \gamma_n,  \gamma_n]} 
\exp \left\{\sum_{j = 3}^{\zeta - 1}R^\interest_{j,n}(\theta, \maximizer) \right\} \phi\left[ \interest; 0, \{-\loglike_{\interest\interest}^{(2)}(\maximizer)\}^{-1} \right] d\interest \\
&= \int_{ [ - \gamma_n,  \gamma_n]} 
   \exp \left\{n^{1/2}\psi \ O\left( \frac{p\log(n)}{n^{1/2}}\right) \right\} \phi\left[ \interest; 0, \{-\loglike_{\interest\interest}^{(2)}(\maximizer)\}^{-1} \right] d\interest \\
   &=\int_{ [ - c_n,  c_n]} 
   \exp \left\{z \ O\left( \frac{p\log(n)}{n^{1/2}}\right) \right\} \phi\left[ z; 0, 1 \right] d\interest
\]
where, $c_n =p^{1/2} \{-\loglike_{\interest\interest}^{(2)}(\maximizer)\}^{1/2} \log(n)^{1/2}/n^{1/2} $, and we performed a change of variable by defining $z =\{-\loglike_{\interest\interest}^{(2)}(\maximizer)\}^{1/2} \psi $.  Now by Lemma \ref{lemma:truncation},
\*[
&\int_{ [ - c_n,  c_n]} 
   \exp \left\{z \ O\left( \frac{p\log(n)}{n^{1/2}}\right) \right\} \phi\left[ z; 0, 1 \right] d\interest \\
   &= \int_{ \mathbb{R}} 
   \exp \left\{z \ O\left( \frac{p\log(n)}{n^{1/2}}\right) \right\} \phi\left[ z; 0, 1 \right] d\interest + O(n^{-\eta_1 p/4}),
\]
and noting that,
\*[
&\int_{ \mathbb{R}} 
   \exp \left\{z \ O\left( \frac{p\log(n)}{n^{1/2}}\right) \right\} \phi\left[ z; 0, 1 \right] d\interest \\
   &= \exp\left[ \frac{1}{2} \left\{O\left( \frac{p\log(n)}{n^{1/2}}\right)\right\}^2 \right] = 1 + O\left\{\frac{p^2\log(n)^2}{n} \right\},
\]
gives the desired result.
\end{proof}

\begin{lemma}\label{lemma:truncation}
Under Assumption \ref{ass:hess_lap}, if $t_n = O(p\log(n)/n^{1/2})$
\*[
&\int_{ [ - c_n,  c_n]} 
   \exp \left\{z t_n \right\} \phi\left[ z; 0 , 1 \right] d\interest = \int_{ \mathbb{R}} 
   \exp \left\{z t_n \right\} \phi\left[ z; 0, 1 \right] d\interest + O(n^{-\eta_1p/4}),
   \]
   where $c_n = p^{1/2} \{-\loglike_{\interest\interest}^{(2)}(\maximizer)\}^{1/2} \log(n)^{1/2}/n^{1/2}$.
\end{lemma}
\begin{proof}
By Assumption \ref{ass:hess_lap} $\{-\loglike_{\interest\interest}^{(2)}(\maximizer)\}^{1/2} \geq (\eta_1 n)^{1/2}$, therefore $c_n \geq \eta_1^{1/2} p^{1/2} \log(n)^{1/2}:= c'_n$ and it follows that
\*[
&\int_{ [ - c_n,  c_n]^C} 
   \exp \left\{z t_n \right\} \phi\left( z; 0 , 1 \right) d\interest \leq \int_{[- c'_n, c'_n]^C} \exp\left\{z t_n \right\} \phi\left( z; 0 , 1 \right) d\interest \\
   &= \exp(t_n^2/2) \int_{[- c'_n, c'_n]^C}  \phi\left( z; t_n , 1 \right) d\interest \\
   &= \exp(t_n^2/2) \mathbb{P}[ \{N\left( z; t_n , 1 \right) > c'_n\}\cup \{N\left( z; t_n , 1 \right) < -c'_n\}] \\
   &\leq \exp(t_n^2/2) \mathbb{P}[ N\left( z; 0 , 1 \right) > \min\{|c'_n - t_n|, |c'_n +t_n| ]\\
   &\leq \exp(t_n^2/2) \mathbb{P}[ \chi^2_1 > \min\{(c'_n - t_n)^2, (c'_n +t_n)^2\} ]\\
   &= \exp(t_n^2/2) \mathbb{P}[ \chi^2_1 > (c'_n)^2 \min\{(1 - t_n/c'_n)^2, (1 + t_n/c_n)^2\} ],
\]
and by Lemma 3 in \cite{Fan},
\begin{align*}
   \mathbb{P}\left[\chi^2_1 \geq 1+ \zeta_n \right] \leq \exp\left[ \frac{1}{2} \{ \log(1 +\zeta_n) - \zeta_n \}  \right],
\end{align*}
where $\zeta_n = (c'_n)^2 \min\left\{ (1 - t_n/c'_n)^2, (1+ t_n/c'_n)^2 \right\}  - 1 \leq \eta_1p\log(n)/2 $, for large $n$, since $t_n/c'_n \rightarrow 0$ by assumption, and $ c'_n  \rightarrow \infty$.  Therefore,
\*[
\mathbb{P}\left[\chi^2_1 \geq 1+ \zeta_n \right]\leq  \exp\left\{ -\eta_1p\log(n)/4 \right\}  = O\left[ n^{-\eta_1 p/4} \right], 
\]
by the same arguments as used in the proof of Lemma \ref{lemma:annulus}, showing the desired result.
\end{proof}

\section{Proof of Theorem \ref{th:general_laplace}}

\begin{lemma}\label{lemma:nuissance_size}
Under Assumption \ref{ass:general_lap_const}:
\*[
\norm{\loglike^{(2)}_{\psi\lambda}(\maximizer)}_2 = O\left\{ (pn)^{1/2} \right\}
\text{ and } 
\norm{\frac{d\hat\lambda_\psi}{d\psi}(\psi)}_2 = O\left(\frac{p^{1/2}}{n^{1/2}}\right), 
\]
for $\psi \in \{\psi: |\psi - \hat\psi| < O(\log(n)^{1/2}/n^{1/2}) \}$.
\end{lemma}

\begin{proof}
Using a first order Taylor series, 
\*[
\loglike^{(2)}_{\psi\lambda}(\maximizer) = \loglike^{(2)}_{\psi\lambda}(\theta_0) + \loglike^{(3)}_{\cdot\psi\lambda}(\tilde\theta)(\maximizer - \theta_0), 
\]
where $\tilde{\theta} = \tau\theta_0 + (1 - \tau(\theta))\maximizer$ for some $0\leq \tau \leq1$. Therefore,
\*[ 
\norm{\loglike^{(2)}_{\psi\lambda}(\maximizer)}_2 &\leq \norm{\loglike^{(2)}_{\psi\lambda}(\theta_0)}_2 + \norm{\loglike^{(3)}_{\cdot\psi\lambda}(\tilde\theta)(\maximizer - \theta_0)}_2\\
&\leq \norm{\loglike^{(2)}_{\psi\lambda}(\theta_0)}_2 + \norm{\loglike^{(3)}_{\cdot\psi\lambda}(\tilde\theta)}_{op}\norm{(\maximizer - \theta_0)}_2\\
&= O\{ (pn)^{1/2} \} + O(n)O(p^{1/2}/n^{1/2}) = O\{ (pn)^{1/2} \}, 
\]
by Assumption \ref{ass:third_lap} and \ref{ass:general_lap_const}  as $\tilde{\theta} \in B_{\maximizer}(p^{1/2}/n^{1/2}) \in B_{\maximizer}(\gamma_n)$.

The proof of the second statement is similar to that of \citet[Lemma 1]{tang2020modified}; we use the identity $g^{(1)}_{\lambda}(\hat\theta_\psi) = 0$, which implies
\*[
\frac{d\hat\nuissance_\interest}{d\interest}(\psi) = - \{g^{(2)}_{\nuissance\nuissance}(\hat\theta_\psi)\}^{-1} g^{(2)}_{\interest\nuissance}(\hat\theta_\interest), 
\]
thus
\*[
\norm{\frac{d\hat\nuissance_\psi}{d\psi}(\psi)}_2 \leq \norm{\{g^{(2)}_{\nuissance\nuissance}(\hat\theta_\psi)\}^{-1}}_{op} \norm{g^{(2)}_{\interest\nuissance}(\hat\theta_\interest)}_2 = O\left( \frac{p^{1/2}}{n^{1/2}}\right), 
\]
by Assumption \ref{ass:general_lap_const}.

\end{proof}

\begin{lemma}\label{lemma:cond_prob_ration}
Under Assumptions \ref{ass:hess_lap} and \ref{ass:general_lap_const}, for all $\nuissance \in  B_{\zeropminus}(\gamma_n)$ and $\alpha < 1/2$
\*[
&\int_{ [ - \gamma_n,  \gamma_n]} 
   \exp \left\{\sum_{j = 3}^{\zeta - 1}R^\interest_{j,n}(\theta, \maximizer) \right\} \phi\left[ \interest; -\loglike^{(2)}_{\psi\lambda}(\maximizer) \{\loglike^{(2)}_{\psi\psi}(\maximizer)\}^{-1}\lambda, \{-\loglike_{\interest\interest}^{(2)}(\maximizer)\}^{-1} \right] d\interest \\
   &= 1 + O\left\{ \frac{p^2\log(n)^{2}}{n} \right\}.
\]
\end{lemma}

\begin{proof}
Let $\mu_n = - \loglike^{(2)}_{\psi\lambda}(\maximizer)\{\loglike^{(2)}_{\psi\psi}(\maximizer)\}^{-1}\lambda$, then
\*[
&\frac{\phi\left[\psi; \mu_n , \{-\loglike^{(2)}_{\interest\interest}(\maximizer)\}^{-1} \right]}{\phi\left[\psi; 0, \{-\loglike^{(2)}_{\interest\interest}(\maximizer)\}^{-1} \right]} = \exp\left\{\frac{-\loglike^{(2)}_{\interest\interest}(\maximizer)}{2} (2\psi \mu_n - \mu_n^2)  \right\}, 
\]
since, $\{-\loglike^{(2)}_{\interest\interest}(\maximizer)\}^{-1} = O(n^{-1})$ and
\*[
|\mu_n| \leq \norm{\loglike^{(2)}_{\psi\lambda}(\maximizer)}_2 \{\loglike^{(2)}_{\psi\psi}(\maximizer)\}^{-1} \norm{\lambda}_2 = O\{(pn)^{1/2}\} O\left( \frac{1}{n} \right) O(\gamma_n) =  O\left\{\frac{p\log(n)^{1/2}}{n} \right\}, 
\]
by Lemma \ref{lemma:nuissance_size} and Assumption \ref{ass:hess_lap}, we have 
\*[
&\frac{\phi\left[\psi; \mu_n , \{-\loglike^{(2)}_{\interest\interest}(\maximizer)\}^{-1} \right]}{\phi\left[\psi; 0, \{-\loglike^{(2)}_{\interest\interest}(\maximizer)\}^{-1} \right]} = \exp\left\{ n^{1/2}\interest \ O\left( \frac{p\log(n)^{1/2}}{n^{1/2}} \right)  \right\}\left[1 + O\left\{\frac{p^2\log(n)}{n} \right\} \right].
\]
Therefore,
\*[
&\int_{ [ - \gamma_n,  \gamma_n]} 
   \exp \left\{\sum_{j = 3}^{\zeta - 1}R^\interest_{j,n}(\theta, \maximizer) \right\} \phi\left[ \interest; \mu_n, \{-\loglike_{\interest\interest}^{(2)}(\maximizer)\}^{-1} \right] d\interest \\
   &= \left[1 + O\left\{\frac{p^2\log(n)}{n} \right\} \right] \\
   &\quad \times \int_{ [ - \gamma_n,  \gamma_n]} 
   \exp \left\{\sum_{j = 3}^{\zeta - 1}R^\interest_{j,n}(\theta, \maximizer) \right\} \exp\left\{ n^{1/2}\interest \ O\left( \frac{p\log(n)^{1/2}}{n^{1/2}} \right)  \right\}  \phi\left[ \interest; 0, \{-\loglike_{\interest\interest}^{(2)}(\maximizer)\}^{-1} \right] d\interest\\
   &= 1 + O\left\{ \frac{p^2\log(n)^2}{n} \right\},
\]
by applying the same steps as in Lemma \ref{lemma:marg_prob}.
\end{proof}

\begin{lemma}\label{lemma:radon_n_gen}
Under Assumptions \ref{ass:hess_lap} and \ref{ass:general_lap_const}, for $\nuissance \in B_{\zerop}(\gamma_n)$
\*[
\frac{\phi\left[\nuissance;  0 , \{-\loglike^{(2)}_{\nuissance\nuissance}(\maximizer) + \loglike^{(2)}_{\lambda\psi}(\maximizer)\loglike^{(2)}_{\psi\psi}(\maximizer)^{-1} \loglike^{(2)}_{\psi\lambda}(\maximizer) \}^{-1}  \right] }{\phi\left[\nuissance; 0, \{-\loglike^{(2)}_{\nuissance\nuissance}(\maximizer) \}^{-1}  \right] } = 1  + O\left\{ \frac{p^2\log(n)}{n} \right\}.
\]
\end{lemma}

\begin{proof}
\*[
&\frac{\phi\left(\nuissance;  0 , [-\loglike^{(2)}_{\nuissance\nuissance}(\maximizer) + \loglike^{(2)}_{\lambda\psi}(\maximizer)\loglike^{(2)}_{\psi\psi}(\maximizer)^{-1} \loglike^{(2)}_{\psi\lambda}(\maximizer) ]^{-1}  \right)}{\phi\left(\nuissance; 0, [-\loglike^{(2)}_{\nuissance\nuissance}(\maximizer) ]^{-1}  \right)}\\
&= \frac{\det\{ -\loglike^{(2)}_{\nuissance\nuissance}(\maximizer) + \loglike^{(2)}_{\lambda\psi}(\maximizer)\loglike^{(2)}_{\psi\psi}(\maximizer)^{-1} \loglike^{(2)}_{\psi\lambda}(\maximizer) \}^{1/2}}{\det\{ -\loglike^{(2)}_{\nuissance\nuissance}(\maximizer) \}^{1/2}} \exp \left[- \frac{1}{2} \nuissance^\top\{ \loglike^{(2)}_{ \lambda\psi}(\maximizer)\loglike^{(2)}_{\psi\psi}(\maximizer)^{-1} \loglike^{(2)}_{\psi\lambda}(\maximizer)\} \nuissance \right]\\
&=\det\left[ I_{p - 1} -  \{\loglike^{(2)}_{\nuissance\nuissance}(\maximizer) \}^{-1} \loglike^{(2)}_{\lambda\psi}(\maximizer)\loglike^{(2)}_{\psi\psi}(\maximizer)^{-1} \loglike^{(2)}_{\psi\lambda}(\maximizer) \right]^{1/2}\\
&\times
\exp \left[- \frac{1}{2} \nuissance^\top\{ \loglike^{(2)}_{ \lambda\psi}(\maximizer)\loglike^{(2)}_{\psi\psi}(\maximizer)^{-1} \loglike^{(2)}_{\psi\lambda}(\maximizer)\} \nuissance \right],
\]
first, 
\[
&\exp \left[- \frac{1}{2} \nuissance^\top\{ \loglike^{(2)}_{ \lambda\psi}(\maximizer)\loglike^{(2)}_{\psi\psi}(\maximizer)^{-1} \loglike^{(2)}_{\psi\lambda}(\maximizer)\} \nuissance \right] \leq \exp \left[ \frac{1}{2} \norm{ \nuissance^\top\{ \loglike^{(2)}_{ \lambda\psi}(\maximizer)\loglike^{(2)}_{\psi\psi}(\maximizer)^{-1} \loglike^{(2)}_{\psi\lambda}(\maximizer)\} \nuissance}_2 \right]\nonumber\\
&\leq \exp \left[ \frac{\gamma_n^2}{2} \norm{ \loglike^{(2)}_{ \lambda\psi}(\maximizer)\loglike^{(2)}_{\psi\psi}(\maximizer)^{-1} \loglike^{(2)}_{\psi\lambda}(\maximizer)}_{op}  \right] \leq \exp \left[ \frac{\gamma_n^2}{2}  \norm{\loglike^{(2)}_{\psi\psi}(\maximizer)^{-1}}_{op} \norm{\loglike^{(2)}_{\psi\lambda}(\maximizer)}_{2}^2  \right]\nonumber\\
&\leq \exp\left\{ O\left( \frac{p\log(n)}{n} \right) O\left(\frac{1}{n}\right) O( p n)  \right\} = \exp\left\{ \frac{p^2\log(n)}{n}\right\} = 1 + O\left( \frac{p^2\log(n)}{n} \right) \label{eq:A15_1}.
\]
A lower bound can also be established using the same argument. For
\*[
\det\left[ I_{p - 1} -  \{\loglike^{(2)}_{\nuissance\nuissance}(\maximizer) \}^{-1} \loglike^{(2)}_{\lambda\psi}(\maximizer)\loglike^{(2)}_{\psi\psi}(\maximizer)^{-1} \loglike^{(2)}_{\psi\lambda}(\maximizer) \right]^{1/2} , 
\]
we consider the operator norm
\*[
&\norm{\{\loglike^{(2)}_{\nuissance\nuissance}(\maximizer) \}^{-1} \loglike^{(2)}_{\lambda\psi}(\maximizer)\loglike^{(2)}_{\psi\psi}(\maximizer)^{-1} \loglike^{(2)}_{\lambda\psi}(\maximizer)}_{op} \\
&\leq \norm{\{\loglike^{(2)}_{\nuissance\nuissance}(\maximizer) \}^{-1}}_{op} \norm{\loglike^{(2)}_{\lambda\psi}(\maximizer)}_2^2 \norm{\loglike^{(2)}_{\psi\psi}(\maximizer)^{-1}}_2 \\
&\leq O\left( \frac{1}{n} \right)O(pn )O(\frac{1}{n}) = O\left( \frac{p}{n} \right), 
\]
and following the same argument as in Lemma \ref{Lemma:ratio_determinants}, we obtain
\[
\det\left[ I_{p - 1} -  \{\loglike^{(2)}_{\nuissance\nuissance}(\maximizer) \}^{-1} \loglike^{(2)}_{\lambda\psi}(\maximizer)\loglike^{(2)}_{\psi\psi}(\maximizer)^{-1} \loglike^{(2)}_{\psi\lambda}(\maximizer) \right]^{1/2} = 1 + O\left(\frac{p^2}{n} \right); \label{eq:A15_2}
\]
combining (\ref{eq:A15_1}) and (\ref{eq:A15_2}) gives the desired result.
\end{proof}

\section{Proof of lemmas for Theorem \ref{th:density_pointwise}}

We use the following version of the Cauchy-Riemann equations to relate the directional derivative of a complex function along the real and imaginary axes. 
Let $z_0 \in \mathbb{C}^p$  be a fixed imaginary number, $x, y \in \mathbb{R}^p$, and $f(z) = f(x + iy)$ a complex differentiable function at the point $z_0$ then
\begin{align*}
    \frac{\partial^k f(z)}{\partial y_{j_1} \cdots y_{j_k}}|_{z = z_0} =i^k \frac{\partial f(z)}{\partial x_{j_1} \cdots x_{j_k}}|_{z = z_0},
\end{align*}

\begin{lemma}\label{lemma:cauchy-riemann}
The following identities hold as a consequence of the Cauchy Riemann equations:
\begin{align*}
    &i)\ y^\top K^{(y, 1)}(\saddle, 0) = iy^\top\approxpoint, \\
    &ii)\  K^{(y, k)}(\saddle, 0) = i^k U^{(x, k)}(\saddle, 0), \\
    &iii)\  K^{(y, k)}(\saddle, y) =i^k\{ U^{(x, k)}(\saddle, y) + i V^{(x, k)}(\saddle, y) \},
\end{align*}
for $ \saddle \in \mathbb{R}^p$ and $y \neq \zerop$. 
\end{lemma}

\begin{proof}
i) The Cauchy Riemann equations imply 
$K^{(y, 1)}(\saddle, 0)  = iK^{(x, 1)}(\saddle, 0) $ 
therefore combining this with the saddlepoint equation (\ref{eq:saddle_def}), we obtain $y^\top K^{(y, 1)}(\saddle, 0)  = iy^\top \approxpoint$.

ii) The second identity follows from the $k$-th order Cauchy Riemann identity
\*[ K^{(y, k)}(\saddle, 0) = i^k K^{(x, k)}(\saddle, 0), \]
since along the $x$ (real) component, the function $K(x, 0) \in \mathbb{R}$, it follows that the derivative of the imaginary component must be 0.

iii) The third identity follows from the $k$-th order Cauchy Riemann identity, except that the imaginary component is no longer necessarily $0$. 
\end{proof}

\begin{lemma} \label{lemma:th1_exp}
In the notation of Theorem \ref{th:density_pointwise}, under Assumptions \ref{ass:hess}--\ref{ass:fourth_cum},
\begin{align*}
 &\int_{ E_{\zerop}(\gamma_n, n^{-1/2}\Sigma^{1/2}) } \exp(2 \max[  0, \Re\{ \bar{R}_{4,n}(\ynew, \tilde{y}, \saddle) \} ] )\phi\left( \ynew; 0, I_p/n \right) d\ynew \\
 &\leq 1 + O\left\{  \frac{p^{5+ 4c_{\infty}} \log(n)^2}{n^{4 - 2\fourthpower}} \right\},
\end{align*}
where $\tilde{y} = \tau(y) y $ for  $0 \leq \tau(y) \leq 1$ and $\alpha < (4- 2\fourthpower)/(5+ 4c_{\infty})$.
\end{lemma}

\begin{proof}
 Note, 
\*[
2\max[   0, \Re\{ \bar{R}_{4,n}(\ynew, \tilde{y}, \saddle)] &\leq 2\left| \Re\{ \bar{R}_{4,n}(\ynew, \tilde{y}, \saddle)\} \right|\\
&\leq  \sum_{j = 1}^p |\ynew_j|  \left|\sum_{k = 1}^p \ynew_k \left( \ynew^\top B_{jk}(\tilde{y}) \ynew \right)\right| := \sum_{j = 1 } |\ynew_j| |t_j(\ynew, \tilde{y})|.
\]
We can uniformly bound 
\*[
|t_j(\ynew, \tilde{y})| &\leq \sup_{\ynew \in E_{\zerop}(\gamma_n, n^{-1/2}\Sigma^{1/2})} \left\{ \norm{\ynew}_1  \max_{j = 1, \dots, p} \norm{\ynew}_2^2 \norm{B_{jk}(\tilde{y})}_{op}  \right\}\\
&\leq \sup_{\ynew \in E_{\zerop}(\gamma_n, n^{-1/2}\Sigma^{1/2})} \left\{ p^{1/2}  \max_{k = 1, \dots, p} \norm{\ynew}_2^3 \norm{B_{jk}(\tilde{y})}_{op}  \right\} \\
&= O\left\{  \frac{p^{2+ 2c_\infty }\log(n)^{3/2}}{n^{3/2 - \fourthpower}} \right\}, 
\]
by Rayleigh's quotient, the $L^p$ inequality and Assumption \ref{ass:fourth_cum}. This upper bound is also uniform in $k$ by Assumption \ref{ass:fourth_cum}.
Using this upper bound on $|t(\ynew, \tilde{y})|$, we can upper bound the integral of interest by a product of moment generating distributions for the standard normal by, 
\*[
&\int_{ E_{\zerop}(\gamma_n, n^{-1/2}\Sigma^{1/2})}  \exp\left\{\sum_{j = 1}^k |\ynew_j| |t_j(\ynew, \tilde\theta)| \right\}  \phi\left( \ynew; 0,  I_p/n \right)  d\ynew \\
&\leq \int_{ E_{\zerop}(\gamma_n,n^{-1/2} \Sigma^{1/2})}  \exp\left[ \sum_{j = 1}^p |\ynew_j|  O\left\{ \frac{p^{2+ 2c_\infty }\log(n)^{3/2}}{n^{3/2 - \fourthpower}} \right\} \right]  \phi\left( \ynew; 0,  I_p/n \right)  d\ynew \\
&\leq \int_{ \mathbb{R}^p}  \exp\left[ \sum_{j = 1}^p |\ynew_j|  O\left\{ \frac{p^{2+ 2c_\infty }\log(n)^{3/2}}{n^{3/2 - \fourthpower}} \right\} \right]   \phi\left( \ynew; 0,  I_p/n \right)  d\ynew\\
&= \prod_{j= 1}^p \int_{ \mathbb{R}}  \exp\left[ |\ynew_j|  O\left\{  \frac{p^{2+ 2c_\infty }\log(n)^{3/2}}{n^{3/2 - \fourthpower}} \right\} \right]  \phi\left( \ynew_j; 0,  1/n \right)  d\ynew_j \\
&\leq \prod_{j= 1}^p 2 \int_{ \mathbb{R}}  \exp\left[ n^{1/2} \ynew_j  O\left\{  \frac{p^{2+ 2c_\infty }\log(n)^{3/2}}{n^{2 - \fourthpower}} \right\} \right]  \phi\left( \ynew_j; 0,  1/n \right)  d\ynew_j \\
&\leq 2 \left( \int_{ \mathbb{R}}  \exp\left[ Z  O\left\{ \frac{p^{2+ 2c_\infty }\log(n)^{3/2}}{n^{2 - \fourthpower}} \right\} \right]  \phi\left( Z; 0,  1 \right)  dZ \right)^p\\
&= \exp\left[ p  O\left\{  \frac{p^{4 + 4c_{\infty}} \log(n)^3}{n^{4 - 2\fourthpower}} \right\} \right] = 1 + O\left\{ \frac{p^{5+ 4c_{\infty}} \log(n)^2}{n^{4 - 2\fourthpower}} \right\},
\]
for $\alpha <  (4 - 2\fourthpower)/(5 + 4c_{\infty} )$, showing the desired result. 
\end{proof}
\end{document}